\documentclass[11pt]{article}
\usepackage{etex}
\usepackage{mathtools}

\usepackage{enumitem}

\usepackage{amsmath,amsthm}
\usepackage{csquotes}

\usepackage{amssymb}

\usepackage{graphicx}

\usepackage{subfig}
\usepackage[final]{showkeys}

\usepackage{appendix}
 
\newtheorem{innercustomgeneric}{\customgenericname}
\providecommand{\customgenericname}{}
\newcommand{\newcustomtheorem}[2]{%
	\newenvironment{#1}[1]
	{ 
		\renewcommand\customgenericname{#2}%
		\renewcommand\theinnercustomgeneric{##1}%
		\innercustomgeneric
	}
	{\endinnercustomgeneric}
}

\newcustomtheorem{customthm}{Theorem}
\newcustomtheorem{customlemma}{Lemma}

\usepackage{etoolbox}

\usepackage{wasysym}

\usepackage{mathrsfs}

\usepackage{mathpazo}

\usepackage[doc]{optional}

\usepackage{soul}

\usepackage{colortbl,booktabs,sectsty,multirow}
\usepackage{xcolor}

\usepackage{cancel}

\usepackage{empheq}

\definecolor{myblue}{rgb}{.8, .8, 1}
  \newcommand*\mybluebox[1]{
    \colorbox{myblue}{\hspace{1em}#1\hspace{1em}}}

\usepackage[obeyspaces,hyphens,spaces]{url}

\usepackage{hyperref}
\hypersetup{
    colorlinks=true,
    linkcolor=blue,
    citecolor=blue,
    filecolor=magenta,
    urlcolor=cyan
}

\usepackage[
  open,
  openlevel=2,
  atend,
  numbered
]{bookmark}

\usepackage[capitalize, nameinlink, noabbrev]{cleveref}
\crefname{equation}{}{}
\crefname{chapter}{Chapter}{Chapters}
\crefname{figure}{Figure}{Figures}
\crefname{theorem}{Theorem}{Theorems}
\crefname{lemma}{Lemma}{Lemmas}
\crefname{proposition}{Proposition}{Propositions}
\crefname{corollary}{Corollary}{Corollarys}
\crefname{definition}{Definition}{Definitions}
\crefname{fact}{Fact}{Facts}
\crefname{example}{Example}{Examples}
\crefname{algorithm}{Algorithm}{Algorithms}
\crefname{remark}{Remark}{Remarks}
\crefname{note}{Note}{Notes}
\crefname{notation}{Notation}{Notations}
\crefname{case}{Case}{Cases}
\crefname{exercise}{Exercise}{Exercises}
\crefname{question}{Question}{Questions}
\crefname{claim}{Claim}{Claims}
\crefname{enumi}{}{}

\usepackage[top= 2cm, bottom = 2 cm, left = 2.2 cm, right= 2.2 cm]{geometry}

\usepackage{float}

\usepackage{pgf}

\usepackage{array}
\usepackage{tabu}

\allowdisplaybreaks

\numberwithin{equation}{section}

\theoremstyle{plain}
\newtheorem{theorem}{Theorem}[section]

\newtheorem{corollary}[theorem]{Corollary}
\newtheorem{fact}[theorem]{Fact}
\newtheorem{lemma}[theorem]{Lemma}
\newtheorem{proposition}[theorem]{Proposition}

\theoremstyle{definition}
\newtheorem{definition}[theorem]{Definition}
\newtheorem{example}[theorem]{Example}

\newtheorem{remark}[theorem]{Remark}

\newcommand{\inte}{\ensuremath{\operatorname{int}}}

\newcommand{\aff}{\ensuremath{\operatorname{aff} \,}}

\newcommand{\spn}{\ensuremath{{\operatorname{span} \,}}}
\newcommand{\weakly}{\ensuremath{{\;\operatorname{\rightharpoonup}\;}}}

\newcommand{\Fix}{\ensuremath{\operatorname{Fix}}}
\newcommand{\Id}{\ensuremath{\operatorname{Id}}}

\newcommand{\dist}{\ensuremath{\operatorname{d}}}

\newcommand{\Pro}{\ensuremath{\operatorname{P}}}
\newcommand{\R}{\ensuremath{\operatorname{R}}}
\newcommand{\I}{\ensuremath{\operatorname{I}}}

\newcommand{\pa}{\ensuremath{\operatorname{par}}}
\newcommand{\card}{\ensuremath{\operatorname{card}}}
\newcommand{\Range}{\ensuremath{\operatorname{ran}}}

\newcommand{\CCO}[1]{CC{#1}}

\newcommand{\CC}[1]{CC_{#1}}

\providecommand{\norm}[1]{\lVert#1\rVert}
\providecommand{\Norm}[1]{{\Big\lVert}#1{\Big\rVert}}

\providecommand{\innp}[1]{\langle#1\rangle}
\providecommand{\Innp}[1]{\Big\langle#1\Big\rangle}

\begin{document}

\title{ \sffamily Finite convergence of locally proper  circumcentered methods}

\author{
         Hui\ Ouyang\thanks{
                 Mathematics, University of British Columbia, Kelowna, B.C.\ V1V~1V7, Canada.
                 E-mail: \href{mailto:hui.ouyang@alumni.ubc.ca}{\texttt{hui.ouyang@alumni.ubc.ca}}.} 
                 }

\date{December 7, 2020}

\maketitle

\begin{abstract}
\noindent
In view of the great performance of circumcentered isometry methods 
for solving the best approximation problem, in this work we further investigate the locally proper circumcenter mapping and circumcentered method. 
Various examples of locally proper circumcenter mapping are presented and studied.  
Inspired by some results on circumcentered-reflection method by Behling, Bello-Cruz, and Santos in their recent papers \cite{BCS2019}, \cite{BBCS2020}, and \cite{BBCS2020CRMbetter}, we provide sufficient conditions for one-step convergence of circumcentered isometry methods for finding the best approximation point onto the intersection of fixed point sets of related isometries. 
In addition, we elaborate the performance of circumcentered reflection methods induced by reflectors associated with hyperplanes and halfspaces for finding the best approximation point onto (or a point in) the intersection of hyperplanes and halfspaces.
\end{abstract}

{\small

}

\section{Introduction} \label{sec:Introduction}
Throughout this paper, we assume that
\begin{empheq}[box = \mybluebox]{equation*}
\text{$\mathcal{H}$ is a real Hilbert space},
\end{empheq}
with inner product $\innp{\cdot,\cdot}$ and induced norm $\|\cdot\|$. Set $\mathbb{N} :=\{0,1,2,\ldots\}$, let $m \in \mathbb{N} \smallsetminus \{0\}$, and denote by $\I:= \{1, \ldots, m\}$. 

 Let $C_{1}, \ldots, C_{m} $ be closed and convex subsets of $\mathcal{H}$ with $\cap^{m}_{i=1} C_{i} \neq \varnothing$. Let $x \in \mathcal{H}$.   The \emph{best approximation problem} is to find the projection $\Pro_{ \cap^{m}_{i=1} C_{i} }x$, and
the \emph{feasibility problem} is to find a point in $\cap^{m}_{i=1} C_{i} $.

Denote by $\mathcal{P}(\mathcal{H})$ the set of nonempty subsets of $\mathcal{H}$ containing finitely many elements.  The circumcenter operator  $\CCO{} \colon \mathcal{P}(\mathcal{H}) \to \mathcal{H} \cup \{ \varnothing \}$
maps every $K \in \mathcal{P}(\mathcal{H})$ to the  \emph{circumcenter $\CCO{(K)}$ of $K$}, where
$\CCO{(K)}$  is either the empty set or the unique point $\CCO{(K)}$ such that $\CCO{(K)} \in \aff (K)$ and $\CCO{(K)}$ is equidistant  from all points  in $K$. 

Let  $\mathcal{K} $ and $\mathcal{D} $ be nonempty subsets of $\mathcal{H}$ with $\mathcal{D} \subseteq \mathcal{K} $. Let $(\forall i \in \I)$ $T_{i} : \mathcal{K} \to \mathcal{H}$. Set $\mathcal{S}:=\{ T_{1}, \ldots,  T_{m} \}$ and$(\forall x \in \mathcal{K} )$ $\mathcal{S}(x) :=\{ T_{1}x, \ldots,   T_{m}x\}$. The \emph{circumcenter mapping $\CC{\mathcal{S}}$ induced by $\mathcal{S}$} is defined as $(\forall x \in \mathcal{K})$ 	$\CC{\mathcal{S}} (x)= \CCO(\mathcal{S}(x))$. We say $\CC{\mathcal{S}}$ is   \emph{proper  over $\mathcal{D}$},  if $(\forall x \in \mathcal{D} )$ $\CC{\mathcal{S}}x \in \mathcal{D}$. In particular, if $\mathcal{D} =\mathcal{H}$, then we omit the phrase \enquote{over $\mathcal{H}$}; if $  \mathcal{D} \subsetneqq \mathcal{H}$, we also say that  $\CC{\mathcal{S}}$ is \emph{locally proper}.
Assume that   $\CC{\mathcal{S}}$ is    proper  over $\mathcal{D}$. Let $x \in \mathcal{D}$. The \emph{circumcentered method induced by $\mathcal{S}$} generates the sequence $(\CC{\mathcal{S}}^{k} x)_{k \in \mathbb{N}}$ of iterations. When $\mathcal{D} \neq \mathcal{H}$, we also say the related circumcentered method is \emph{locally proper}.
 Particularly, if all $T_{1}, \ldots, T_{m}$ are isometries (resp.\,reflectors), then we call the method \emph{circumcentered isometry method} (\emph{resp.\,circumcentered reflection method}) and CIM (resp.\,CRM) for short.

For every $ i \in \I$, let $U_{i}$ be an affine subspace and let $\R_{U_{i}}$ be the reflector associated with $U_{i}$. 
 In \cite{BCS2017}, Behling, Bello Cruz and Santos defined the circumcentered Douglas-Rachford method (C-DRM), which is the first circumcentered method in the literature and is actually the circumcentered method induced by $S:=\{\Id, \R_{U_{1}}, \R_{U_{2}}\R_{U_{1}} \}$. Then they  introduced the circumcentered-reflection method which is   the circumcentered method induced by $S:=\{\Id, \R_{U_{1}}, \R_{U_{2}}\R_{U_{1}}, \cdots,   \R_{U_{m}} \cdots \R_{U_{2}}\R_{U_{1}} \}$.  Motivated by the C-DRM, Bauschke, Ouyang and Wang introduced the circumcenter mapping and circumcentered method defined above  and showed the global properness of circumcenter mappings induced by sets of isometries in \cite{BOyW2019Isometry}.
Results on circumcentered methods accelerating   the well-known Douglas-Rachford method and method of alternating projections can be found in  \cite{BCS2017}, \cite{BOyW2019LinearConvergence}, and \cite{BBCS2020CRMbetter}.
Note that circumcenter mappings induced by sets of operators that are not isometric are generally not globally proper, which implies that the related circumcentered methods may not be well-defined on the whole space. But clearly if a circumcenter mapping is proper over a nonempty set, say $\mathcal{D}$, then the related circumcentered method is well-defined with initial points in $\mathcal{D}$. Considering the great performance of CIMs for solving best approximation problems, naturally one may want to know how will the locally proper circumcentered methods perform when solving  best approximation or feasibility problems?

In this work, \emph{our goal is to study locally proper circumcenter mappings and to explore the finite convergence of  circumcentered methods (not necessarily proper over $\mathcal{H}$) for solving the best approximation and  feasibility problems}. 
The main results in this work are the following:
\begin{enumerate}
	\item[\textbf{R1:}] \cref{lemma:pxyz} presents an explicit formula for the  circumcenter of three points satisfying an equidistant condition, which plays a critical role in proving the one-step convergence of the CIM  with initial point from a certain set   in 
	\cref{theorem:CCSHRT12}.

	\item[\textbf{R2:}] \cref{prop:Tn:CCSP} states sufficient conditions for the one-step convergence of circumcenter methods induced by $\mathcal{S}$ being $\{ \Id,  T_{1},T_{2},  \ldots, T_{m} \}$ or $\{ \Id,  T_{1}, T_{2}T_{1},  \ldots, T_{m}\cdots T_{2}T_{1} \} $ where $(\forall i \in \I)$ $T_{i}$ is isometric.  Moreover, \Cref{prop:Rn:CCSTmT1:Condition,prop:Rn:CCSP:Condition} illustrate that the sets of initial points required in  \cref{prop:Tn:CCSP} are dense in $\mathcal{H}$ and specify clear procedures to find practical initial points.
		
	\item[\textbf{R3:}] \cref{theorem:ConclusionWH} illustrates that we can always find CRMs induced by sets of appropriate  reflectors for finding the best approximation point onto or feasibility point in the intersection of hyperplanes and halfspaces in at most three steps. 
\end{enumerate}	

The organization of the rest of the work is the following. Some preliminary results are presented in \cref{sec:AuxiliaryResults}. In particular,
 \cref{theorem:quasidemiweakconver} generalizes Bauschke et al.  \cite[Theorem~4.7]{BOyW2019Isometry} from circumcenter mapping to general operator and specifies   sufficient conditions for the weak convergence of the related  iteration sequence.  
 In \cref{sec:CMLocalProper}, we obtain more practical formula for the circumcenter of three points satisfying a equidistant condition.
  We also  verify that some results on circumcenter mappings with global properness assumption in \cite{BOyW2019Isometry} and \cite{BOyW2019LinearConvergence} actually hold with only local properness assumption (see \cref{fact:CCSRestrict}  and \cref{prop:conver:normconve}).
In \cref{sec:OneStepConvergence}, sufficient conditions for the one-step convergence of CIMs are exhibited in \Cref{theorem:CCSHRT12,prop:Tn:CCSP}. 
 In addition, \Cref{prop:Rn:CCSTmT1:Condition,prop:Rn:CCSP:Condition} confirm that the conditions required in   \cref{prop:Tn:CCSP} are not strong and state clear steps to find desired initial points.
In fact, \cref{theorem:CCSHRT12}\cref{theorem:CCSHRT12:P}, \cref{prop:Tn:CCSP}\cref{prop:Tn:CCSP:T12mTm}, and \cref{prop:CUC}\cref{prop:CUC:CCS} are generalizations of  \cite[Lemma~2]{BBCS2020}, \cite[Lemma~3]{BCS2019}, and  \cite[Lemma~3]{BBCS2020}, respectively. The  detailed relations between our results and the related known results are elaborated in \cref{sec:OneStepConvergence}.  Last but not least, the finite convergence of circumcentered reflection methods induced by sets of reflectors associated with hyperplanes and halfspaces is explored in \cref{sec:CRM:hyperplanes}. 

We now turn to the notation used in this work. Let $C$ be a nonempty subset
of $\mathcal{H}$.   
The \emph{orthogonal complement} of $C$ is the set $ C^{\perp} :=\{y \in \mathcal{H}~|~ (\forall x \in C)~\innp{x,y}=0 \}. $ 
$C$ is an \emph{affine subspace} of
$\mathcal{H}$ if $C \neq \varnothing$ and $(\forall \rho\in\mathbb{R})$ $\rho
C + (1-\rho)C = C$. The intersection of all affine subspaces of $\mathcal{H}$ containing $C$ is denoted by $\aff C$ and called the \emph{affine hull} of $C$.    
An affine subspace $U$ is said to be \emph{parallel} to an affine subspace $
M $ if $U = M +a $ for some $ a \in \mathcal{H}$.
Every affine subspace $U$ is parallel to a unique linear
subspace $\pa U$, which is given by
$(\forall y \in U)$  $\pa U := U - y = U - U$.
Suppose that  $C$ is a nonempty closed convex subset of $\mathcal{H}$. The \emph{projector} (or \emph{projection operator}) onto $C$ is the operator, denoted by
$\Pro_{C}$,  that maps every point in $\mathcal{H}$ to its unique projection onto $C$. $\R_{C} :=2 \Pro_{C} -\Id$ is the \emph{reflector} associated with $C$. 
Let $x \in \mathcal{H}$ and let $\epsilon \in \mathbb{R}_{+}$. $\ker x := \{y \in
\mathcal{H} ~:~ \innp{y,x} =0 \}$ and $\mathbf{B}[x;\epsilon] := \{ y \in  \mathcal{H} ~:~ \norm{y-x} \leq \epsilon  \}$. 
For two subsets $A$ and $B$ of $\mathcal{H}$, 
$\dist (A,
B) := \inf \{\norm{a-b} ~:~ a \in A \text{ and } b \in B  \}$.

Let $\mathcal{D}$ be a nonempty subset of $\mathcal{H}$, and let 
$T: \mathcal{D} \rightarrow \mathcal{H}$.  
$T$ is said to be \emph{isometric} or an \emph{isometry}, if $(\forall x, y \in \mathcal{D})$ $\norm{T(x) -T(y)} =\norm{x-y}$. The \emph{set of fixed
	points of the operator $T$} is denoted by $\Fix T$, i.e., $\Fix T := \{x \in
\mathcal{D} ~:~ Tx=x\}$. The \emph{range of $T$} is defined as $\Range T :=\{Tx ~:~ x \in \mathcal{H} \}$;  moreover, $\overline{\Range} \, T$ is the closure of $\Range T$. Denote by $\mathcal{B} (\mathcal{H}, \mathcal{H}) := \{ T: \mathcal{H} \rightarrow \mathcal{H} ~:~ T ~\text{is bounded and linear} \}$. For every $T \in \mathcal{B} (\mathcal{H}, \mathcal{H})$, the \emph{operator norm} $\norm{T}$ of $T$ is defined by $\norm{T} := \sup_{\norm{x} \leq  1} \norm{Tx}$.  A sequence $(x_{k})_{k \in \mathbb{N}}$ in $\mathcal{H}$ \emph{converges
	weakly} to a point $x \in \mathcal{H}$ if, for every $u \in \mathcal{H}$,
$\innp{x_{k},u} \rightarrow \innp{x,u}$; in symbols, $x_{k} \weakly x$. 

For other notation not explicitly defined here, we refer the reader to \cite{BC2017}.

\section{Auxiliary results} \label{sec:AuxiliaryResults}

In this section, we collect some  results to be used subsequently.

\subsection*{Projections}
 
 \begin{fact} {\rm \cite[Fact~1.8]{DH2006II}}  \label{fact:AsubseteqB:Projection}
 	Let $A$ and $B$ be two nonempty closed convex subsets of $\mathcal{H}$. Let $A \subseteq B $ and $x \in \mathcal{H}$. Then  $\Pro_{B}x \in A$ if and only if $ \Pro_{B}x  =\Pro_{A}x $.
 \end{fact}

\begin{fact} {\rm \cite[Proposition~3.19]{BC2017}} \label{fac:SetChangeProje}
	Let $C$ be a nonempty closed convex subset of $\mathcal{H}$ and let $x \in \mathcal{H}$. Set $D:=z+C$, where $z \in \mathcal{H}$. Then $\Pro_{D}x=z+\Pro_{C}(x-z)$.
\end{fact}

\begin{fact} {\rm \cite[Example~29.18]{BC2017}} \label{fact:Projec:Hyperplane}
Let $u \in \mathcal{H} \smallsetminus \{0\}$, let $\eta \in \mathbb{R}$, and set $H := \{ x \in \mathcal{H} ~:~ \innp{x,u} = \eta \}$. Then $	(\forall x \in \mathcal{H}) $ $ \Pro_{H} x = x+ \frac{\eta - \innp{x,u}}{\norm{u}^{2}}u.$
\end{fact}

\begin{fact} {\rm  \cite[Remark~2.9]{Oy2020ProjectionHH}} \label{fact:FactWFactH}
	Let $u \in \mathcal{H} \smallsetminus \{0\}$ and $\eta \in \mathbb{R}$. Set $W:=\{x \in \mathcal{H} ~:~ \innp{x,u} \leq \eta\}$ and $H := \{ x \in \mathcal{H} ~:~ \innp{x,u} = \eta \}$. 
	Then $W \neq \varnothing$,  $W^{c} \neq \varnothing$,  and 
	$(\forall x \in W^{c})$ $\Pro_{W}x =x+ \frac{\eta - \innp{x,u}}{\norm{u}^{2}}u= \Pro_{H}x$. 
\end{fact}

\begin{lemma} \label{lem:WH:mathcalH:RWinW}
	Let $u \in \mathcal{H} \smallsetminus \{0\}$ and let $\eta \in \mathbb{R}$. Set $W:=\{x \in \mathcal{H} ~:~ \innp{x, u} \leq \eta \} $ and $H:=\{x \in \mathcal{H} ~:~ \innp{x, u} = \eta \}.  $
	Let $x \in \mathcal{H} \smallsetminus W$. Then $\R_{H}x = \R_{W}x \in \inte W$.
\end{lemma}

\begin{proof}
	According to \cref{fact:FactWFactH}, $\R_{W}x= \R_{H}x= 2\Pro_{H}(x)-x =x + 2\frac{\eta -\innp{x,u}}{\norm{u}^{2}}u$. Clearly, $x \notin W$ means $ \innp{x, u} -\eta >0$. Hence, $	\innp{\R_{W}x, u} -\eta  = \Innp{x +2 \frac{\eta- \innp{x, u}}{\norm{u}^{2}} u , u} -\eta  
	= \innp{x,u} -\eta +2 \frac{\eta- \innp{x, u}}{\norm{u}^{2}} \innp{u , u}   
	=\eta- \innp{x, u} <0,$
	which implies that $\R_{H}x=\R_{W}x \in \inte W$. 
\end{proof}

\subsection*{Sufficient conditions for weak convergence}
\begin{definition} {\rm \cite[Definition~5.1]{BC2017}} \label{defn:Fejer}
	Let $C$ be a nonempty  subset of $\mathcal{H}$ and let $(x_{k})_{k \in \mathbb{N}}$ be a sequence in  $\mathcal{H}$.  Then $(x_{k})_{k \in \mathbb{N}}$  is \emph{Fej\'er monotone with respect to $C$} if 
	\begin{align*}
	(\forall x \in C) (\forall k \in \mathbb{N}) \quad \norm{x_{k+1} -x} \leq \norm{x_{k} -x}.
	\end{align*}
\end{definition}

\begin{definition} \label{defn:demiclosedness} \cite[Definition~4.26]{BC2017}
	Let $\mathcal{D}$ be a nonempty weakly sequentially closed subset of $\mathcal{H}$, let $T:\mathcal{D} \to \mathcal{H}$, and let $u \in \mathcal{H}$. Then $T$ is \emph{demiclosed at $u$} if,  for every sequence $(x_{k})_{k \in \mathbb{N}}$ in $\mathcal{D}$ and every $x \in \mathcal{D} $ such that $x_{k} \weakly x$ and $Tx_{k} \to u$, we have $Tx =u$. In addition, $T$ is \emph{demiclosed} if it is demiclosed at every point in $\mathcal{D}$. 
\end{definition}

\begin{definition} {\rm \cite[Definition~4.1(iv)]{BC2017} and \cite[Definition~2.2]{CRZ2018}} \label{defn:SQNE}
	Let $\mathcal{D}$ be a nonempty subset of $\mathcal{H}$ and let $\rho \in \mathbb{R}_{+}$.	We say  $T:\mathcal{D} \to \mathcal{H}$ is \emph{$\rho$-strongly quasinonexpansive} ($\rho$-SQNE), if 
	\begin{align}\label{eq:defn:SQNE}
	(\forall x \in \mathcal{D}) (\forall z \in \Fix T) \quad \norm{T(x )-z}^{2} + \rho \norm{T(x)-x}^{2} \leq \norm{x -z}^{2}.
	\end{align}
	In particular, if $\rho=0$ and $\rho=1$ in \cref{eq:defn:SQNE}, then $T$ is called \emph{quasinonexpansive} and  \emph{firmly quasinonexpansive}, respectively.
\end{definition}

\begin{theorem}\label{theorem:quasidemiweakconver}
	Let $\mathcal{D}$ be a nonempty weakly sequentially closed subset of $\mathcal{H}$, let $T:\mathcal{D} \to \mathcal{D}$ and let $\rho \in \mathbb{R}_{++}$.  Assume that $T$ is $\rho$-SQNE, that $\Fix T \neq \varnothing$, and that $\Id -T$ is  demiclosed at $0$.  Let $x  \in \mathcal{D}$.  Then $(T^{k}x)_{k \in \mathbb{N}}$ weakly converges to one point in $\Fix T$. Moreover, if $\Fix T$ is a closed affine subspace of $\mathcal{H}$, then 
	$T^{k}x \weakly \Pro_{\Fix T}x$, and $(\forall k \in \mathbb{N})$ $\Pro_{\Fix T}T^{k}x= \Pro_{\Fix T}x$.
\end{theorem}

\begin{proof}
	Note that, by \cref{defn:SQNE}, $T$ is $\rho$-SQNE implies that $T$ is  quasinonexpansive. Moreover, because $\Fix T \neq \varnothing$, using \cite[Example~5.3]{BC2017}, we know that  $(T^{k}x)_{k \in \mathbb{N}}$ is Fej\'er monotone with respect to $\Fix T$. Now, in view of \cite[Theorem~5.5 and Proposition~5.9]{BC2017}, it remains to show that very weak sequential cluster point of $(T^{k}x)_{k \in \mathbb{N}}$ belongs to $\Fix T$.
	
	Let $y \in \mathcal{H}$ and  let $(T^{k_{i}}x)_{i \in \mathbb{N}}$ be a subsequence of $(T^{k}x )_{k \in \mathbb{N}}$ such that $T^{k_{i}}x \weakly y$.  Let $z \in \Fix T$. According to  \cite[Proposition~5.4]{BC2017},  $(T^{k}x)_{k \in \mathbb{N}}$ is Fej\'er monotone with respect to $\Fix T$ implies that the limit $L_{z}:= \lim_{k \to \infty} \norm{T^{k}(x)  - z} $ exists. 
	Because $T:\mathcal{D} \to \mathcal{D}$ is $\rho$-SQNE, we have that 
	\begin{align}\label{eq:theorem:quasidemiweakconver}
	(\forall i \in \mathbb{N})  \quad \rho \norm{T(T^{k_{i}}x) -T^{k_{i}}x }^{2} \leq  \norm{T^{k_{i}}(x )  - z}^{2} -\norm{T(T^{k_{i}}x ) - z}^{2}.
	\end{align}
	Summing over $i$ from $0$ to infinity in both sides of \cref{eq:theorem:quasidemiweakconver}, we obtain that 
	\begin{align*}
	\sum_{i=0}^{\infty} \norm{T(T^{k_{i}}x ) -T^{k_{i}}x }^{2} \leq \frac{1}{ \rho} \left(  \norm{T^{k_{0}}(x ) -z}^{2}  -L_{z } \right),
	\end{align*}
	which implies that $T(T^{k_{i}}x ) -T^{k_{i}}x  \to 0$.
	Furthermore, using the assumption, $\Id -T$ is  demiclosed at $0$, and \cref{defn:demiclosedness}, we see that 
	\begin{equation*}
	\label{e:Tdemiclosed}
	\left. 
	\begin{array}{c} T^{k_{i}}x  \weakly y\\
	T(T^{k_{i}}x ) -T^{k_{i}}x  \to 0
	\end{array}
	\right\}
	\;\;  \Rightarrow  \;\; (\Id -T)y=0, \text{ i.e., }
	y\in \Fix T. 
	\end{equation*}
	Altogether, the proof is complete.
\end{proof}

\begin{remark}
	Let $\mathcal{D}$ be a nonempty weakly sequentially closed subset of $\mathcal{H}$ and let $T:\mathcal{D} \to \mathcal{D}$ with $\Fix T \neq \varnothing$.	According to \cite[Definition~4.1 and Theorem~4.27]{BC2017}, $T$ is $1$-SQNE implies that $\Id -T$ is demiclosed. Hence, \cref{theorem:quasidemiweakconver} deduces that if $T$  is $\rho$-SQNE with $\rho \geq 1$, then $(T^{k}x)_{k \in \mathbb{N}}$ weakly converges to one point in $\Fix T$.
\end{remark}

\section{Circumcenter mapping with local properness} \label{sec:CMLocalProper}

In this section, we investigate the locally proper circumcenter mapping. 

Recall that $\mathcal{P}(\mathcal{H})$ is the set of all nonempty 
subsets of $\mathcal{H}$ containing \emph{finitely many}
elements, and that $\I:=\{1,2, \ldots, m\}$.
\subsection*{Circumcenters}
According to  \cite[Proposition~3.3]{BOyW2018},  for every $K$ in $ \mathcal{P}(\mathcal{H})  $, there is at most one point $p \in \aff (K) $ such that $\{\norm{p-x} ~|~x \in K \}$ is a singleton. Hence, the following \cref{defn:Circumcenter} is well-defined.

\begin{definition}[circumcenter operator]  {\rm \cite[Definition~3.4]{BOyW2018}} \label{defn:Circumcenter}
The \emph{circumcenter operator} is 
	\begin{empheq}[box=\mybluebox]{equation*}
	\CCO{} \colon \mathcal{P}(\mathcal{H}) \to \mathcal{H} \cup \{ \varnothing \} \colon K \mapsto \begin{cases} p, \quad ~\text{if}~p \in \aff (K)~\text{and}~\{\norm{p-x} ~|~x \in K \}~\text{is a singleton};\\
	\varnothing, \quad~ \text{otherwise}.
	\end{cases}
	\end{empheq}
	In particular, when $\CCO(K) \in \mathcal{H}$, that is, $\CCO(K) \neq \varnothing$, we say that the circumcenter of $K$ exists and we call $\CCO(K)$ the \emph{circumcenter of $K$}.
\end{definition}
\begin{fact} {\rm \cite[Theorem~8.4]{BOyW2018}} 
	\label{thm:SymForm}
	Suppose that $S:=\{x,y,z\} \subseteq \mathcal{H}$ and denote the cardinality of $S$ by
	$l$. 
	Then exactly one of the following cases occurs: 
	\begin{enumerate}
		\item \label{thm:SymForm:1} $l=1$ and $\CCO(S) =x$.
		\item \label{thm:SymForm:2} $l=2$, say $S=\{u,v\}$, where $u,v \in S$ 
		and $u \neq v$, and $\CCO(S)=\frac{u+v}{2}$.
		\item \label{thm:SymForm:3} 
		$l=3$ and exactly one of the following two cases occurs: 
		\begin{enumerate}
			\item \label{thm:SymForm:3:a} 
			$x, y, z$ are affinely independent; equivalently, 
			$\norm{y-x}\norm{z-x} > \innp{y-x,z-x}$, and 
			\begin{align*}
			\CCO(S) = \frac{\norm{y-z}^{2} \innp{x-z,x-y}x+ \norm{x-z}^{2} \innp{y-z,y-x}y+ \norm{x-y}^{2} \innp{z-x,z-y}z }{2(\norm{y-x}^{2}\norm{z-x}^{2}- \innp{y-x, z-x}^{2})}.
			\end{align*}
			\item \label{thm:SymForm:3:b} 
			$x, y, z$ are affinely dependent; equivalently, 
			$\norm{y-x}\norm{z-x} = \innp{y-x,z-x}$, and $\CCO(S) =
			\varnothing $. 
		\end{enumerate}
	\end{enumerate}
\end{fact}

\begin{lemma}  \label{lemma:normeq}
	Let $\{p, x, y\} \subseteq \mathcal{H}$. Then the following statements are equivalent:
	\begin{enumerate}
		\item\label{lemma:normeq:norm} $\norm{p-x} =\norm{p-y}$.
		\item\label{lemma:normeq:innp} $\innp{p-x,y-x} =\frac{1}{2} \norm{y-x}^{2}$.
		\item\label{lemma:normeq2}  $\innp{p-\frac{x+y}{2},\frac{x+y}{2}-x} =0$.
		\item\label{lemma:normeq:0} $\innp{p-\frac{x+y}{2},y-x} =0$.
		\item \label{lemma:normeq:innp-}  $\innp{x-p,y-p} =\norm{p-x}^{2} - \frac{1}{2} \norm{y-x}^{2}$.
	\end{enumerate}
\end{lemma}

\begin{proof}
	By \cite[Proposition~3.1]{BOyW2018}, we know that \cref{lemma:normeq:norm} $\Leftrightarrow$ \cref{lemma:normeq:innp}.
	Because $y-x= 2 \left( \frac{x+y}{2} -x \right)$, we have that 
	\begin{align*}
	\innp{p-x,y-x} =\frac{1}{2} \norm{y-x}^{2} & \Leftrightarrow  2	\Innp{p-x,\frac{x+y}{2} -x } =2 \Norm{\frac{x+y}{2} -x }^{2}\\
	& \Leftrightarrow   	\Innp{p-\frac{x+y}{2} +\frac{x+y}{2} -x,\frac{x+y}{2} -x } =  \Norm{\frac{x+y}{2} -x }^{2}\\
	& \Leftrightarrow   \Innp{p-\frac{x+y}{2},\frac{x+y}{2}-x} =0\\
	& \Leftrightarrow  \Innp{p-\frac{x+y}{2},y-x} =0,
	\end{align*} 
	which shows the equivalence of \cref{lemma:normeq:innp}, \cref{lemma:normeq2} and \cref{lemma:normeq:0}.  
	
	On the other hand, 
	\begin{align*}
		\innp{p-x,y-x} =\frac{1}{2} \norm{y-x}^{2}  &\Leftrightarrow  \innp{p-x,y-p+p-x} =\frac{1}{2} \norm{y-x}^{2}\\
		& \Leftrightarrow  \innp{x-p,y-p} =\norm{p-x}^{2} - \frac{1}{2} \norm{y-x}^{2},
	\end{align*}
	which necessitates that  \cref{lemma:normeq:innp} $\Leftrightarrow $ \cref{lemma:normeq:innp-}.
	Altogether,  all of the equivalences hold.
\end{proof}

\begin{lemma} \label{lemma:eqequivalence}
	Let $\{p, x, y\} \subseteq \mathcal{H}$. Assume that $\norm{p-x}=\norm{p-y}$. Then the following statements hold:
	\begin{enumerate}
		\item \label{lemma:eqequivalence:equi} $\norm{y-x}=2\norm{x-p}$ if and only if $p=\frac{x+y}{2}$.
		\item \label{lemma:eqequivalence:inequi} $\norm{y-x}<2\norm{x-p}$ if and only if $p\neq \frac{x+y}{2}$.
	\end{enumerate}
\end{lemma}

\begin{proof}
	\cref{lemma:eqequivalence:equi}: If $p=\frac{x+y}{2}$, then clearly $2\norm{x-p}=2\Norm{x-\frac{x+y}{2}} =\norm{y-x}$.
	
	Assume that $\norm{y-x}=2\norm{x-p}$. Because $\norm{p-x}=\norm{p-y}$, 
	\begin{subequations}\label{eq:lemma:eqequivalence:equi}
		\begin{align}
		\norm{y-x}=2\norm{x-p}  &\Leftrightarrow \norm{y-p}^{2} +2\innp{y-p,p-x} +\norm{p-x}^{2} =4 \norm{x-p}^{2}\\
		&\Leftrightarrow \innp{y-p,p-x} =\norm{x-p}^{2}\\
		&\Leftrightarrow  \innp{y-p +x-p ,p-x} =0\\
		& \Leftrightarrow  \Innp{\frac{x+y}{2}-p , p-x} =0. \label{eq:lemma:eqequivalence:equi:x}
		\end{align}
	\end{subequations}
	Note that $\norm{p-x}=\norm{p-y}$ implies that $	\norm{y-x}=2\norm{x-p}  \Leftrightarrow \norm{x-y}=2\norm{y-p}$. Combine this with the calculation of \cref{eq:lemma:eqequivalence:equi} to ensure that 
	\begin{align}\label{eq:lemma:eqequivalence:equi:y}
	\Innp{\frac{y+x}{2}-p , p-y} =0.
	\end{align}
	Now \cref{eq:lemma:eqequivalence:equi:x} and \cref{eq:lemma:eqequivalence:equi:y} add up to 
	\begin{align*}
	\Innp{\frac{y+x}{2}-p , 2p-x-y} =0 \Leftrightarrow	\Innp{\frac{y+x}{2}-p , p-\frac{y+x}{2}} =0   \Leftrightarrow p=\frac{x+y}{2}.
	\end{align*}
	Altogether, \cref{lemma:eqequivalence:equi} holds.
	
	\cref{lemma:eqequivalence:inequi}: Using the assumption  $\norm{p-x}=\norm{p-y}$ and the triangle inequality, we know that $\norm{y-x}\leq \norm{y-p} +\norm{p-x}=2\norm{x-p}$. Then \cref{lemma:eqequivalence:inequi} follows directly from 	\cref{lemma:eqequivalence:equi}.
\end{proof}

\begin{theorem}\label{lemma:pxyz}
	Let $\{ x, y,z\} \subseteq \mathcal{H}$. 	Assume that $\norm{x-y}=\norm{x-z}$.  Then 
	\begin{align*}
	\CCO{( \{  x, y,z \} )} =\begin{cases}
	x, \quad \quad\quad\quad   \text{if } x=y;\\
	\varnothing, \quad \quad\quad\quad   \text{if } x\neq y \text{ and } x = \frac{y+z}{2};\\
	x +  \alpha ( y -x)+\alpha ( z -x),  ~\text{where } \alpha:= \left( 4 -\frac{ \norm{y-z}^{2}}{ \norm{y-x}^{2}}  \right)^{-1} \in \left[ \frac{1}{4},+\infty  \right[\, , \quad  \text{otherwise}.
	\end{cases}
	\end{align*}
\end{theorem}

\begin{proof}
	If $x=y$, then  $\norm{x-y}=\norm{x-z}$ implies that $x=y=z$, and that $	\CCO{( \{  x, y,z \} )} =x$.

	Suppose that $x\neq y$ and  $x = \frac{y+z}{2}$. Then clearly, $y\neq z$ and $\aff \{x,y,z  \} = \aff \{ y,z  \}$. Assume to the contrary that $	\CCO{( \{  x, y,z \} )} \neq \varnothing$, say $q:=	\CCO{( \{  x, y,z \} )} $. Then   \cite[Proposition~3.1(i)$\Leftrightarrow$(iii)]{BOyW2018} yields  $q=\frac{1}{2} (y+z)$, and so  $\Norm{x-q}=\Norm{\frac{1}{2} (y+z)- \frac{1}{2} (y+z)} =0 \neq \frac{1}{2} \norm{y-z} =\norm{y-q}=\norm{ z-q}$ which contradicts \cref{defn:Circumcenter}. Hence,  $	\CCO{( \{  x, y,z \} )} =\varnothing$.
	
	Assume that $x \neq y$ and $x \neq  \frac{y+z}{2}$.  
	Applying \cref{lemma:eqequivalence}\cref{lemma:eqequivalence:inequi} with replacing $p, x,$ and $y$   by $x, y,$ and $z$, respectively, we obtain that $\norm{z-y}^{2} < 4 \norm{y-x}^{2}$,
	which implies that $4 -\frac{ \norm{y-z}^{2}}{ \norm{y-x}^{2}} \in \left]0, 4\right]$.  Hence, $\alpha:= \left( 4 -\frac{ \norm{y-z}^{2}}{ \norm{y-x}^{2}}  \right)^{-1} \in \left[ \frac{1}{4},+\infty  \right[\,$ is well-defined. 
	Set 
	\begin{align}\label{eq:lemma:pxyz:p}
	p:=x +  \alpha ( y -x)+\alpha ( z -x).
	\end{align}
	Because, clearly, $	p \in \aff \{x,y,z  \}$, in order to prove $	\CCO{( \{  x, y,z \} )} =p$, it suffices to prove that 	
	\begin{align} \label{eq:lemma:pxyz}
	\norm{p-x}=\norm{p-y}=\norm{p-z}.
	\end{align}
	Applying \cref{lemma:normeq}\cref{lemma:normeq:norm}$\Leftrightarrow$\cref{lemma:normeq:innp-} with $p, x,$ and $y$ substituted by $x, y,$ and $z$ respectively,  we have that 
	\begin{align}\label{eq:lemma:pxyz:norm}
	\norm{x-y}=\norm{x-z} 
	\Leftrightarrow \innp{y-x, z-x} = \norm{x-y}^{2}-\tfrac{1}{2} \norm{z-y}^{2}.
	\end{align}
	Now, 
	\begin{subequations}\label{eq:lemma:pxyz:xy}
		\begin{align}
		\norm{p-x}=\norm{p-y} &~\Leftrightarrow~ \innp{p-x,y-x} =\tfrac{1}{2} \norm{y-x}^{2} \quad \text{by (\cref{lemma:normeq}\cref{lemma:normeq:norm}$\Leftrightarrow$\cref{lemma:normeq:innp})}\\
		&\stackrel{\cref{eq:lemma:pxyz:p}}{\Leftrightarrow} \alpha \innp{z-x, y-x} =  (\tfrac{1}{2} -\alpha  ) \norm{y-x}^{2}\\
		&\stackrel{\cref{eq:lemma:pxyz:norm}}{\Leftrightarrow}-\tfrac{\alpha}{2} \norm{z-y}^{2}  =(\tfrac{1}{2} -2\alpha  ) \norm{y-x}^{2}\\
		&~\Leftrightarrow~ \frac{ \norm{z-y}^{2}  }{\norm{y-x}^{2}} =\tfrac{ 2 (2\alpha -\frac{1}{2}  ) }{\alpha} =4 - \frac{1}{\alpha}\\
		&~\Leftrightarrow~ \alpha =\left( 4 -\frac{ \norm{y-z}^{2}}{ \norm{y-x}^{2}}  \right)^{-1}.
		\end{align}
	\end{subequations}
	
	Moreover, repeat similar calculation in \cref{eq:lemma:pxyz:xy} for  $	\norm{p-x}=\norm{p-z} $ and use $\norm{y-x}=\norm{z-x}$  to obtain that 
	\begin{align} \label{eq:lemma:pxyz:xz}
	\norm{p-x}=\norm{p-z} \Leftrightarrow \alpha =\left( 4 -\frac{ \norm{y-z}^{2}}{ \norm{y-x}^{2}}  \right)^{-1}.
	\end{align}
	Combine \cref{eq:lemma:pxyz:xy} and \cref{eq:lemma:pxyz:xz} with the definition of $\alpha$ to see  that \cref{eq:lemma:pxyz} holds. 	
\end{proof}

\subsection*{Circumcenter mapping}
\begin{definition}[induced   circumcenter mapping] \cite[Definition~3.1]{BOyW2018Proper} \label{defn:cir:map}
	Let  $\mathcal{K} $ and $\mathcal{D} $ be nonempty subsets of $\mathcal{H}$ with $\mathcal{D} \subseteq \mathcal{K} $. Let $(\forall i \in \I)$ $T_{i} : \mathcal{K} \to \mathcal{H}$. Set $\mathcal{S}:=\{ T_{1}, \ldots,  T_{m} \}$ and$(\forall x \in \mathcal{K} )$ $\mathcal{S}(x) :=\{ T_{1}x, \ldots,   T_{m}x\}$. 
	The \emph{circumcenter mapping $\CC{\mathcal{S}}$ induced by $\mathcal{S}$} is 
	\begin{empheq}[box=\mybluebox]{equation*}
	\CC{\mathcal{S}} \colon \mathcal{K}  \to \mathcal{H} \cup \{ \varnothing \} \colon x \mapsto \CCO(\mathcal{S}(x)),
	\end{empheq}
	that is,  for every $x
	\in \mathcal{K} $, if the circumcenter of the set $\mathcal{S}(x)$ defined in
	\cref{defn:Circumcenter} does not exist in $\mathcal{H}$, then $\CC{\mathcal{S}}x= \varnothing
	$. Otherwise, $\CC{\mathcal{S}}x$ is the unique point satisfying the two
	conditions below:
	\begin{enumerate}
		\item $\CC{\mathcal{S}}x \in \aff(\mathcal{S}(x))=\aff\{T_{1}x, \ldots, T_{m-1}x, T_{m}x\}$, and 
		\item $\norm{\CC{\mathcal{S}}x -T_{1}x}=\cdots =\norm{\CC{\mathcal{S}}x -T_{m-1}x}=\norm{\CC{\mathcal{S}}x -T_{m}x}$.
	\end{enumerate}
If $(\forall x \in \mathcal{D} )$ $\CC{\mathcal{S}}x \in \mathcal{D}$, then we say $\CC{\mathcal{S}}$ is   \emph{proper  over $\mathcal{D}$}. In particular, if $\mathcal{D} =\mathcal{H}$, then we omit the phrase \enquote{over $\mathcal{H}$} and also say $\CC{\mathcal{S}}$ is    \enquote{globally proper}; otherwise,  we  say that  $\CC{\mathcal{S}}$ is \emph{locally proper}.
\end{definition}

\begin{example}
Suppose that $\mathcal{H} =\mathbb{R}^{2}$.	Denote by $B_{1}:=\mathbf{B}[(-1,0);1]$ and $B_{2}:=\mathbf{B}[(1,0);1]$. Set $\mathcal{S}_{1} := \{ \Pro_{B_{1}} , \Pro_{B_{2}} \}$, $\mathcal{S}_{2} := \{\Id,  \Pro_{B_{1}} ,  \Pro_{B_{2}}  \}$, and $\mathcal{D} := \mathbb{R}\cdot(0,1)$. 
Then,  by \cref{thm:SymForm} and \cite[Example~3.18]{BC2017}, it is easy to see that $\CC{\mathcal{S}_{1}}$  is both globally proper and proper over $\mathcal{D}$, and that $\CC{\mathcal{S}_{2}}$ is not globally proper but it is proper over $\mathcal{D}$.
\end{example}

Let $\mathcal{D}$ be a nonempty subset of $\mathcal{H}$, let $t \in \mathbb{N} \smallsetminus \{0\}$, and let $(\forall i \in \{1, \ldots, t\})$ $F_{i} :\mathcal{D} \to \mathcal{H}$. Denote by
\begin{empheq}[box=\mybluebox]{equation} \label{eq:Omega}
\Omega ( F_{1}, \ldots, F_{t}) :=  \Big\{ F_{i_{r}}\cdots F_{i_{2}}F_{i_{1}}  ~\Big|~ r \in \mathbb{N},~ i_{0}:=0, ~\mbox{and}~ i_{1}, \ldots,  i_{r} \in \{1, \ldots,t\}    \Big\}
\end{empheq}
which is the set consisting of all finite composition of operators from $\{F_{1}, \ldots, F_{t}\}$. We use the empty product convention, so for $r=0$, $F_{i_{0}}\cdots F_{i_{1}} = \Id$. Hence, $\Id \in \Omega ( F_{1}, \ldots, F_{t}) $.

Note that in the original papers of results presented in  the following \cref{fact:CCSRestrict}, we have the assumption that the corresponding $\CC{\mathcal{S}}$ is proper.
It is easy to see that the global properness assumption in the original papers is too strong, and that using the related proofs, we actually obtain the following results. 
\begin{fact}\label{fact:CCSRestrict} 
	Let $\mathcal{D}$ be a nonempty weakly sequentially closed subset of $\mathcal{H}$. Let $(\forall i \in \I)$ $T_{i} :\mathcal{D} \to \mathcal{H}$. 
	Then the following statements hold:
	\begin{enumerate}
		\item \label{fact:CCSRestrict:fix} {\rm \cite[Proposition~3.10(iii)]{BOyW2018Proper} and \cite[Proposition~2.29]{BOyW2019Isometry}} Suppose that $	\mathcal{S} $ is a subset of $ \Omega ( T_{1}, \ldots, T_{m})$ such that $\{\Id, T_{1}, \ldots, T_{m}  \} \subseteq 	\mathcal{S} $ or   $ \{\Id, T_{1},T_{2}T_{1},  \ldots, T_{m} \cdots T_{2}T_{1} \} \subseteq 	\mathcal{S} $. Then $\Fix \CC{\mathcal{S}} = \cap_{i \in \I} \Fix T_{i}$.  
		\item \label{fact:CCSRestrict:demiclosed} {\rm \cite[Theorem~3.17]{BOyW2018Proper}} Set $	\mathcal{S} := \{   T_{1}, \ldots, T_{m}  \}$. Assume that $(\forall x \in \mathcal{D})$ $\CC{\mathcal{S}} (x) \in \mathcal{D}$,    that    $(\forall i \in \I)$ $\Id -T_{i}$ is  demiclosed at $0$, and that for every $ (  x_{k}  ) _{k \in \mathbb{N}} $ in $ \mathcal{D}$, $x_{k} - \CC{\mathcal{S}} x_{k} \to 0$ implies that $(\forall i \in \I)$  $\CC{\mathcal{S}} x_{k}  -T_{i}x_{k} \to 0$. Then  $\Id -\CC{\mathcal{S}}$ is  demiclosed at $0$, and $\Fix \CC{\mathcal{S}} =\cap_{i \in \I} \Fix T_{i}$
		\item \label{fact:CCSRestrict:Iddemiclosed} {\rm \cite[Theorem~3.20]{BOyW2018Proper}}  Set $	\mathcal{S} := \{  \Id, T_{1}, \ldots, T_{m}  \}$. Assume that $(\forall x \in \mathcal{D})$ $\CC{\mathcal{S}}  (x) \in \mathcal{D}$ and   that    $(\forall i \in \I)$ $ T_{i}$ is nonexpansive. Then    $\Id -\CC{\mathcal{S}}$ is  demiclosed at $0$.	
		
\item \label{fact:CCSRestrict:P} Suppose that $(\forall i \in \I)$ $T_{i} :\mathcal{D} \to \mathcal{H}$ is isometric with $\cap_{i \in \I} \Fix T_{i} \neq \varnothing$. Set $	\mathcal{S} := \{\Id, T_{1}, \ldots, T_{m}  \}$.   Then the following statements hold:
\begin{enumerate}
	\item {\rm \cite[Theorem~3.3]{BOyW2019Isometry}}  \label{fact:CCSRestrict:P:CCS}  $(\forall x \in \mathcal{D})$ $(\forall z \in \cap_{i \in \I} \Fix T_{i} ) $ $ \CC{\mathcal{S}} x =\Pro_{\aff (\mathcal{S} (x))} (z)$.
	\item {\rm \cite[Lemma~3.5]{BOyW2019Isometry}}  \label{fact:CCSRestrict:P:eq}  Let $F: \mathcal{D} \to \mathcal{H}$ satisfy $(\forall x \in \mathcal{D})$  $F(x) \in \aff (\mathcal{S}(x))$. Then $(\forall x \in \mathcal{D})$ $(\forall z \in \cap_{i \in \I} \Fix T_{i} ) $ $ \norm{z - \CC{\mathcal{S}} x}^{2} +\norm{\CC{\mathcal{S}} x- Fx}^{2} =\norm{z -Fx}^{2}$.
\end{enumerate}
	\item \label{fact:CCSRestrict:firmlyquasinonexpan} {\rm \cite[Proposition~3.8]{BOyW2019Isometry}} 	  Suppose that $(\forall i \in \I)$ $T_{i} :\mathcal{D} \to \mathcal{H}$ is isometric  with $\cap_{i \in \I} \Fix T_{i} \neq \varnothing$, and that	$	\mathcal{S} $ is a subset of $ \Omega ( T_{1}, \ldots, T_{m})$ such that $\{\Id, T_{1}, \ldots, T_{m}  \} \subseteq 	\mathcal{S} $ or   $ \{\Id, T_{1},T_{2}T_{1}, \ldots, T_{m} \cdots T_{2}T_{1} \} \subseteq 	\mathcal{S} $. Then 
	\begin{align*}
	(\forall x \in \mathcal{D}) (\forall z \in \Fix \CC{\mathcal{S}} ) \quad \norm{\CC{\mathcal{S}}(x) -z }^{2} + \norm{\CC{\mathcal{S}}(x) -x }^{2} = \norm{x -z}^{2}.
	\end{align*}
	Consequently, $ \CC{\mathcal{S}} $ is firmly quasinonexpansive.
	\end{enumerate}
\end{fact}

The following  \cref{prop:conver:normconve}\cref{prop:conver:normconve:weakly} generalizes \cite[Theorem~4.7]{BOyW2019Isometry} where the global properness assumption is required.
\begin{proposition} \label{prop:conver:normconve}
	Let $\mathcal{D}$ be a nonempty weakly sequentially closed subset of $\mathcal{H}$.  	Let $(\forall i \in \I)$ $T_{i} : \mathcal{D} \to \mathcal{H}$ be isometric such that $\cap_{i \in \I} \Fix T_{i} \neq \varnothing$. Suppose $\mathcal{S}$ is $\{\Id, T_{1},\ldots, T_{m} \}$ or $ \{\Id, T_{1},T_{2}T_{1},  \ldots, T_{m} \cdots T_{2}T_{1} \} $.  Assume that $(\forall x \in \mathcal{D})$  $\CC{\mathcal{S}}x \in \mathcal{D}$. Let $x \in \mathcal{D}$. Then the following statements hold:
	\begin{enumerate}
		\item \label{prop:conver:normconve:weakly} $(\forall k \in \mathbb{N} \smallsetminus \{0\})$ $\Pro_{  \cap^{m}_{i=1}\Fix T_{i} }( \CC{\mathcal{S}}^{k}x) = \Pro_{  \cap^{m}_{i=1}\Fix T_{i} }x$, and
		$(\CC{\mathcal{S}}^{k}x)_{k \in \mathbb{N}}$ weakly converges to $\Pro_{\cap^{m}_{i=1} \Fix T_{i}}x$.
		\item \label{prop:conver:normconve:strong} Assume that
		$\limsup_{k \to \infty} \norm{\CC{\mathcal{S}}^{k}x} \leq \norm{\Pro_{\cap^{m}_{i=1} \Fix T_{i}}x} $ or that $\mathcal{H} =\mathbb{R}^{n}$. Then $(\CC{\mathcal{S}}^{k}x)_{k \in \mathbb{N}}$ converges to $\Pro_{\cap^{m}_{i=1} \Fix T_{i}}x$.
	\end{enumerate}
\end{proposition}

\begin{proof}
	\cref{prop:conver:normconve:weakly}: 
	Because $(\forall i \in \I)$ $T_{i} : \mathcal{H} \to \mathcal{H}$ is isometric such that $\cap_{i \in \I} \Fix T_{i} \neq \varnothing$,  by \cite[Lemma~2.24]{BOyW2019Isometry},  \cite[Proposition~3.4]{BOyW2019LinearConvergence} and \cref{fact:CCSRestrict}\cref{fact:CCSRestrict:fix}, we know that $(\forall k \in \I)$ $T_{k}\cdots T_{1}$ is isometric and $\Fix \CC{\mathcal{S}} = \cap_{i \in \I} \Fix T_{i} $ is a  closed affine subspace of $\mathcal{H}$. Then using \cref{fact:CCSRestrict}\cref{fact:CCSRestrict:Iddemiclosed}$\&$\cref{fact:CCSRestrict:firmlyquasinonexpan}, we obtain respectively that  $\Id -\CC{\mathcal{S}}$ is  demiclosed at $0$, and that $ \CC{\mathcal{S}}$ is firmly quasinonexpansive.  
	Combine these results with \cref{defn:SQNE} and \cref{theorem:quasidemiweakconver} to see that  \cref{prop:conver:normconve:weakly} holds.

	\cref{prop:conver:normconve:strong}: The required result follows from  \cref{prop:conver:normconve:weakly}, \cite[Lemma~2.51]{BC2017} and the well-known fact that in $\mathbb{R}^{n}$ weak convergence is equivalent to strong convergence. 
\end{proof}

 \begin{lemma}\label{lemma:affineTOlinear}
 Let $(\forall i \in \I)$ $G_{i} : \mathcal{H} \to \mathcal{H}$ with $\cap_{i\in \I} \Fix G_{i} \neq \varnothing$ and let $z \in \mathcal{H}$. Define $(\forall i \in \I)$ $(\forall x \in \mathcal{H})$  $F_{i}x := G_{i} (x+z) -z$. Set $\mathcal{S}:=\{ G_{1}, \ldots,  G_{m} \}$ and 
 	$\mathcal{S}_{F} := \{F_{1}, \ldots, F_{m}\}$.  Let $C $ be a nonempty subset of $\mathcal{H}$, and let $k \in \mathbb{N}$.     Then the following statements hold:
 	\begin{enumerate}
 		\item \label{lemma:affineTOlinear:F} $(\forall i \in \I)$  $\Fix F_{i} =\Fix G_{i} -z$, and 	$\cap_{i\in \I} \Fix F_{i} =\cap_{i\in \I} \Fix G_{i}   -z$. 
 		Moreover, if  $(\forall i \in \I)$ $G_{i}$ is affine and   $z \in \cap_{i\in \I} \Fix G_{i} $, then $(\forall i \in \I)$   $F_{i}$ is linear.
 		\item \label{lemma:affineTOlinear:k}  $(\forall x \in \mathcal{H})$  $ \Pro_{  \cap^{m}_{i=1}\Fix G_{i} }x = z + \Pro_{\cap^{m}_{i=1}\Fix F_{i}}(x-z)$ and $\CC{\mathcal{S}}^{k} x =  z+ \CC{\mathcal{S}_{F}}^{k}(x-z)$.
 		\item \label{lemma:affineTOlinear:EQ} $(\forall x \in C  ) $ $\CC{\mathcal{S}}^{k}x =\Pro_{  \cap^{m}_{i=1}\Fix G_{i} }x$ if and only if $(\forall x \in C-z ) $ $\CC{\mathcal{S}_{F}}^{k}x =\Pro_{\cap^{m}_{i=1}\Fix F_{i}}x $.
 		Consequently, 	 $(\forall x \in \cap_{i\in \I} \Fix G_{i}  ) $ $\CC{\mathcal{S}}^{k}x =\Pro_{  \cap^{m}_{i=1}\Fix G_{i} }x$ if and only if $(\forall x \in \cap_{i\in \I} \Fix F_{i}  ) $ $\CC{\mathcal{S}_{F}}^{k} x =\Pro_{\cap^{m}_{i=1}\Fix F_{i}}x $.
 		
 		\item \label{lemma:affineTOlinear:Fix} Then $(\forall x \in C )$ $\CC{\mathcal{S}}^{k}x \in C $ if and only if $(\forall x \in  C-z)$ $\CC{\mathcal{S}_{F}}^{k} x \in C-z$. 		
 		Consequently, $(\forall x \in \cap_{i\in \I} \Fix G_{i} )$ $\CC{\mathcal{S}}^{k} x \in \cap_{i\in \I} \Fix G_{i} $ if and only if $(\forall x \in  \cap_{i\in \I} \Fix F_{i})$ $\CC{\mathcal{S}_{F}}^{k} x \in \cap_{i\in \I} \Fix F_{i}$.
 	\end{enumerate}
 \end{lemma}
 \begin{proof}
 	\cref{lemma:affineTOlinear:F}: This follows from  \cite[Lemma~3.8(ii)$\&$(iii)]{BOyW2019LinearConvergence}.
 	
 	\cref{lemma:affineTOlinear:k}: This is from \cite[Lemma~4.8(ii)$\&$(iv)]{BOyW2019LinearConvergence}.
 	
 	\cref{lemma:affineTOlinear:EQ}$\&$\cref{lemma:affineTOlinear:Fix}: The desired results are clear from 	\cref{lemma:affineTOlinear:F} and 	\cref{lemma:affineTOlinear:k} above.
 \end{proof}

According to \cite[Theorem~4.3 and Corollary~5.2]{BOyW2018Proper},   the $\CC{\mathcal{S}}$  in \cref{prop:CCS:CP} below is globally proper. The following result states that this $\CC{\mathcal{S}}$ is also locally proper.
 
 \begin{proposition}  \label{prop:CCS:CP}
 Let $(\forall i \in \I)$ $u_{i} \in \mathcal{H} \setminus\{0\}$ and $\eta_{i} \in \mathbb{R}$. Set  $(\forall i \in \I)$ $H_{i}:= \{ x \in \mathcal{H} ~:~ \innp{x, u_{i}} =\eta_{i} \}$, 
 	and $C:= \{ x \in \mathcal{H} ~:~ (\forall i \in \I)  \quad \norm{x - \Pro_{H_{i}}x  } = \norm{x - \Pro_{H_{1}}x  }  \}$.
 	Assume that $\cap_{i \in \I} H_{i}  \neq \varnothing$. Set $\mathcal{S}:= \{\Id,  \Pro_{H_{1}}, \ldots, \Pro_{H_{m}}  \}$.
 	Then $(\forall x \in C)$ $\CC{\mathcal{S}}(x) \in C$.
 \end{proposition}
 
 \begin{proof}
 	Take $z \in \cap_{i \in \I} H_{i}$. Set $(\forall i \in \I)$ $L_{i}:= H_{i} -z$, that is, $(\forall i \in \I)$ $L_{i}:= \{ x \in \mathcal{H} ~:~ \innp{x, u_{i}} =0 \}$.
 Now, $(\forall i \in \I)$ $L_{i}$ is a linear subspaces of $\mathcal{H}$, $\Pro_{L_{i}}$ is linear, and by \cref{fac:SetChangeProje}, $	(\forall x \in \mathcal{H})$ $ \Pro_{H_{i}}(x) =\Pro_{z+L_{i}}(x)= z + \Pro_{L_{i}}(x-z)$, which implies that 
 	\begin{align*}
 	C-z  =& \{ y \in \mathcal{H} ~:~ y+z \in C \}\\
 	 =& \{  y \in \mathcal{H} ~:~ (\forall i \in \I)  \quad \norm{y+z - \Pro_{H_{i}}(y+z) } = \norm{y+z - \Pro_{H_{1}} (y+z) } \}\\
 	 =& \{  y \in \mathcal{H} ~:~ (\forall i \in \I)  \quad \norm{y - \Pro_{L_{i}}(y ) } = \norm{y  - \Pro_{L_{1}} (y ) } \}.
 	\end{align*}
 	Set	$\mathcal{S}_{L}:= \{\Id,  \Pro_{L_{1}}, \ldots, \Pro_{L_{m}}  \}$. 	Moreover,  by \cref{lemma:affineTOlinear}\cref{lemma:affineTOlinear:Fix}, we know that $(\forall x \in C )$ $\CC{\mathcal{S}}x \in C $ if and only if $(\forall x \in  C-z)$ $\CC{\mathcal{S}_{L}}x \in C-z$. 
 	Therefore, without loss of generality, we assume that $(\forall i \in \I)$ $\eta_{i}=0$, that is, $(\forall i \in \I)$ $H_{i} $ is  a linear subspace  of $\mathcal{H}$ and $\Pro_{H_{i}}$ is linear.
 	
 	Let $x \in C$.	If there exists $j \in \I$ such that $x \in H_{j}$, then by definition of $C$, $x \in C$ implies that $x \in \cap_{i \in \I} H_{i}$ and so $ \CC{\mathcal{S}}(x)=x \in C$. Suppose that $(\forall i \in \I)$ $x \notin H_{i}$.
 	
 	According to  \cite[Corollary~5.2]{BOyW2018Proper}, $\CC{\mathcal{S}}(x) \in \mathcal{H}$ satisfies that $ \CC{\mathcal{S}}(x) \in \aff \{ x, \Pro_{H_{1}}x, \ldots, \Pro_{H_{m}}x \} $ and that $(\forall i \in \I)$ $\norm{ \CC{\mathcal{S}}(x)-x }= \norm{ \CC{\mathcal{S}}(x)-\Pro_{H_{i}}x }.$
 	Let $i \in \I$. Using \cref{lemma:normeq}\cref{lemma:normeq:norm}$\Leftrightarrow$\cref{lemma:normeq2}$\Leftrightarrow$\cref{lemma:normeq:0}, we have that 
 	\begin{subequations}\label{eq:prop:CCS:CP}
 		\begin{align}
 		&\norm{ \CC{\mathcal{S}}(x)-x }= \norm{ \CC{\mathcal{S}}(x)-\Pro_{H_{i}}x }\\
 		\Leftrightarrow & \Innp{ \CC{\mathcal{S}}(x) - \frac{x + \Pro_{H_{i}}x }{2} , \frac{x + \Pro_{H_{i}}x }{2}-\Pro_{H_{i}}x }=0 \label{eq:prop:CCS:CP:innp}\\
 		\Leftrightarrow & \Innp{ \CC{\mathcal{S}}(x) - \frac{x + \Pro_{H_{i}}x }{2} , x-\Pro_{H_{i}}x }=0. \label{eq:prop:CCS:CP:x}
 		\end{align}
 	\end{subequations}
 	Note that  $\eta_{i}=0$ implies    $H_{i}=\ker u_{i}$ and $H_{i}^{\perp}=\spn \{ u_{i} \}$.  Because  $x \notin H_{i}$,  by \cite[Corollary~3.24(v)]{BC2017}, we know that  $x-  \Pro_{H_{i}}x=\Pro_{H_{i}^{\perp}}x \in \spn\{ u_{i} \} \smallsetminus \{0\}$. Hence, due to \cref{eq:prop:CCS:CP:x}, $\CC{\mathcal{S}}(x) - \frac{x + \Pro_{H_{i}}x }{2}  \in  \left( \spn\{ x-\Pro_{H_{i}}x  \} \right)^{\perp} = \left( \spn\{ u_{i} \} \right)^{\perp} =H_{i}$, which, by \cite[Theorem~5.13(1)$\&$(2)]{D2012}, implies that
 	\begin{align}\label{eq:prop:CCS:CP:norm}
 	\CC{\mathcal{S}}(x) - \frac{x + \Pro_{H_{i}}x }{2}    =   \Pro_{H_{i}} \left(\CC{\mathcal{S}}(x) - \frac{x + \Pro_{H_{i}}x }{2} \right) 
 	=   \Pro_{H_{i}}  \CC{\mathcal{S}}(x) -\Pro_{H_{i}} x.
 	\end{align}
 	Using \cite[Proposition~2.10(iii)]{BOyW2018Proper} and \cref{eq:prop:CCS:CP:innp}, we obtain respectively  that
 	\begin{align*}
 	&\norm{ \CC{\mathcal{S}}(x) -\Pro_{H_{i}} x }^{2}=\norm{\CC{\mathcal{S}}(x)-\Pro_{H_{i}} \CC{\mathcal{S}}(x) }^{2} +\norm{ \Pro_{H_{i}}  \CC{\mathcal{S}}(x) -\Pro_{H_{i}} x  }^{2};\\
 	& \norm{ \CC{\mathcal{S}}(x) -\Pro_{H_{i}} x }^{2}=\Norm{\CC{\mathcal{S}}(x) - \frac{x + \Pro_{H_{i}}x }{2}}^{2} +\Norm{  \frac{x + \Pro_{H_{i}}x }{2}-\Pro_{H_{i}} x  }^{2}.
 	\end{align*}
 	Combine these two identities with \cref{eq:prop:CCS:CP:norm} to see that 
 	\begin{align*}
 	\norm{\CC{\mathcal{S}}(x)-\Pro_{H_{i}} \CC{\mathcal{S}}(x) }= \Norm{  \frac{x + \Pro_{H_{i}}x }{2}-\Pro_{H_{i}} x  }=\frac{1}{2} \norm{x - \Pro_{H_{i}}x  }.
 	\end{align*}
 	Therefore, $(\forall x \in C)$  $(\forall i \in \I)$ $\norm{\CC{\mathcal{S}}(x)-\Pro_{H_{i}} \CC{\mathcal{S}}(x) } =\frac{1}{2} \norm{x - \Pro_{H_{i}}x  } =\frac{1}{2} \norm{x - \Pro_{H_{1}}x  }=\norm{\CC{\mathcal{S}}(x)-\Pro_{H_{1}} \CC{\mathcal{S}}(x) }$, which verifies that $ \CC{\mathcal{S}}(x) \in C$. 
 \end{proof}

\section{One-step convergence of circumcentered methods} \label{sec:OneStepConvergence}
Assume that  the circumcenter mapping $\CC{\mathcal{S}}$   is  proper  over a nonempty subset  $\mathcal{D}$ of $\mathcal{H}$.
Recall that  the  related circumcentered method   generates the sequence $(\CC{\mathcal{S}}^{k} x)_{k \in \mathbb{N}}$ of iterations, and that when $\mathcal{D} \neq \mathcal{H}$, we   say the related circumcentered method is \emph{locally proper}.
In this section, we show one-step convergence of circumcentered methods
 in \cref{theorem:CCSHRT12}\cref{theorem:CCSHRT12:P}, \cref{cor:CCSHU} and \cref{prop:Tn:CCSP} below.

\subsection*{Circumcentered methods with initial points from a certain set}

 \cref{theorem:CCSHRT12}\cref{theorem:CCSHRT12:P}  is motivated by  and  a generalization of \cite[Lemma~2]{BBCS2020}.
In fact,   \cref{cor:CCSHU} below reduces to \cite[Lemma~2]{BBCS2020} when  $\mathcal{H}=\mathbb{R}^{n}$.

\begin{theorem} \label{theorem:CCSHRT12}
  Let $T_{1} :\mathcal{H} \to \mathcal{H}$ and $T_{2} :\mathcal{H} \to \mathcal{H}$  be isometries with $\Fix T_{1} \cap \Fix T_{2} \neq \varnothing$.  Set $\mathcal{S} := \{\Id, T_{1}, T_{2}T_{1}  \}$.
	Then the following statements hold:
	\begin{enumerate}
		\item \label{theorem:CCSHRT12:welldefined} $(\forall x \in \mathcal{H})$ $(\forall z \in \Fix T_{1} \cap \Fix T_{2} ) $ $ \CC{\mathcal{S}} x =\Pro_{\aff (\mathcal{S} (x))} (z) \in \mathcal{H}$ and $(\forall k \in \mathbb{N})$ $ \Pro_{ \Fix T_{1} \cap \Fix T_{2} }\CC{\mathcal{S}}^{k}x=\Pro_{ \Fix T_{1} \cap \Fix T_{2} }x$. Moreover, for every $x \in \Fix T_{2}$,
		\begin{align*}
		\CC{\mathcal{S}}(x) &=
		\begin{cases}
		x  \quad \text{if } x =T_{1}x;\\
		x+\alpha (T_{1}x-x) +\alpha(T_{2}T_{1}x  -x), ~\text{where~} \alpha:=\left( 4- \frac{\norm{T_{2}T_{1}x -T_{1}x}^{2} }{\norm{T_{1}x -x}^{2}}  \right)^{-1} \in \left[\frac{1}{4}, \infty \right[ \quad \text{otherwise}.
		\end{cases}
		\end{align*}
		\item  \label{theorem:CCSHRT12:CCS} Assume that $T_{1}(\Fix T_{2}) \subseteq \Fix T^{2}_{2}$  $($e.g., when $\Range T_{1} \subseteq \Fix T^{2}_{2}$ or $T^{2}_{2} =\Id$$)$. Then 
		\begin{align*}
		(\forall x  \in \Fix T_{2}) \quad \CC{\mathcal{S}}(x) \in \Fix T_{2}.
		\end{align*}

		\item \label{theorem:CCSHRT12:P} Assume that $T_{1}(\Fix T_{2}) \subseteq \Fix T^{2}_{2}$,  that
		$\Fix T_{2} \subseteq \Fix T^{2}_{1}$,
		 and that $\dim (\pa(\Fix T_{1})^{\perp} )=1$. Then 
		\begin{align*}
		(\forall x  \in \Fix T_{2}) \quad \CC{\mathcal{S}}(x) =\Pro_{ \Fix T_{1} \cap \Fix T_{2} }x.
		\end{align*}
	\end{enumerate}
\end{theorem}

\begin{proof}
	\cref{theorem:CCSHRT12:welldefined}: Note that compositions of isometries are   isometries and that $ \Fix T_{1} \cap \Fix T_{2} =\Fix T_{1} \cap \Fix T_{2}T_{1} $.  Then   \cref{fact:CCSRestrict}\cref{fact:CCSRestrict:P:CCS} and \cite[Theorem~5.7]{BOyW2019Isometry} imply that $(\forall x \in \mathcal{H})$ $(\forall z \in \Fix T_{1} \cap \Fix T_{2} ) $ $ \CC{\mathcal{S}} x =\Pro_{\aff (\mathcal{S} (x))} (z) \in \mathcal{H}$  and $(\forall k \in \mathbb{N})$ $ \Pro_{ \Fix T_{1} \cap \Fix T_{2} }\CC{\mathcal{S}}^{k}x=\Pro_{ \Fix T_{1} \cap \Fix T_{2} }x$.
	Combine this with \cref{lemma:pxyz} to see that 
	\begin{align*}
	\{ x \in \mathcal{H} ~:~  \norm{x -T_{1}x} = \norm{x -T_{2}T_{1}x}   \} \cap \{ x \in \mathcal{H} ~:~ x\neq T_{1}x \text{ and } x=\frac{T_{1}x +T_{2}T_{1}x}{2} \} =\varnothing.
	\end{align*}
	
	In addition, inasmuch as $T_{2}$ is isometric, 
	\begin{align*}
	(\forall x \in \mathcal{H}) \quad x \in \Fix T_{2} \Leftrightarrow x =T_{2}x \Rightarrow \norm{x -T_{2}T_{1}x} =\norm{T_{2}x -T_{2}T_{1}x} = \norm{x -T_{1}x},
	\end{align*}
	which yields  $	\Fix T_{2} \subseteq \{ x \in \mathcal{H} ~:~  \norm{x -T_{1}x} = \norm{x -T_{2}T_{1}x}   \}$.
Therefore, for every $x \in 	\Fix T_{2}$,	the required  formula  follows from   \cref{lemma:pxyz} with $x, y$ and  $z$ substituted by $x$, $T_{1}x$ and $T_{2}T_{1}x$,  respectively.

	\cref{theorem:CCSHRT12:CCS}: According to \cite[Proposition~3.4]{BOyW2019LinearConvergence}, every isometry must be affine. Then using \cref{lemma:affineTOlinear}\cref{lemma:affineTOlinear:F}$\&$\cref{lemma:affineTOlinear:Fix},  without loss of generality, we assume that $T_{1}$ and $T_{2}$ are linear isometries.
Let $x \in 	\Fix T_{2}$.	In view of \cref{theorem:CCSHRT12:welldefined} above, we have exactly the following two cases:
	
	\textbf{Case~1}: $x =T_{1}x$. Then $\CC{\mathcal{S}} (x) =x \in \Fix T_{2}$. 
	
	\textbf{Case~2}: $x \neq T_{1}x$. Then $\CC{\mathcal{S}} (x) =x+\alpha (T_{1}x-x) +\alpha (T_{2}T_{1}x  -x)$, where $\alpha:=\left( 4- \frac{\norm{T_{2}T_{1}x -T_{1}x}^{2} }{\norm{T_{1}x -x}^{2}}  \right)^{-1} \in \left[\frac{1}{4}, \infty \right[\,$. 
	Note that  $x\in \Fix T_{2}$ and $T_{1}(\Fix T_{2}) \subseteq \Fix T^{2}_{2}$ entail that 
	$T_{2}T_{2}T_{1}x=T_{1}x$. Combine this with the linearity of $T_{2}$ to observe that 
	\begin{align*}
	T_{2} \left( \CC{\mathcal{S}} (x)   \right) =T_{2} \left( x+\alpha (T_{1}x-x) +\alpha (T_{2}T_{1}x  -x) \right) = x+\alpha (T_{2}T_{1}x  -x)  +\alpha (T_{1}x-x)  =\CC{\mathcal{S}} (x),
	\end{align*}
	which leads that $ \CC{\mathcal{S}} (x) \in \Fix T_{2}$.

	\cref{theorem:CCSHRT12:P}: 
	 Recall that $ \Fix T_{1} \cap \Fix T_{2} =\Fix T_{1} \cap \Fix T_{2} T_{1} $. Then similarly with the proof of	\cref{theorem:CCSHRT12:CCS} above, using \cref{lemma:affineTOlinear}\cref{lemma:affineTOlinear:F}$\&$\cref{lemma:affineTOlinear:EQ}, we assume that $T_{1}$ and $T_{2}$ are linear isometries. Now, $\Fix T_{1} \cap \Fix T_{2}$ is a   closed linear subspace of $\mathcal{H}$.
		Let $ x \in \Fix T_{2}$.  
	
	 Note that if  we can prove that $\CC{\mathcal{S}}x \in \Fix T_{1}$, then, by \cref{theorem:CCSHRT12:CCS}, $\CC{\mathcal{S}}x \in \Fix T_{1} \cap \Fix T_{2}$, and so $ \CC{\mathcal{S}}x=  \Pro_{ \Fix T_{1} \cap \Fix T_{2} }\CC{\mathcal{S}}x \stackrel{\text{\cref{theorem:CCSHRT12:welldefined}}}{=} \Pro_{ \Fix T_{1} \cap \Fix T_{2} }x$.
	Therefore, it remains to show that $\CC{\mathcal{S}}x \in \Fix T_{1}$.
	
	Because  $T_{1}$ is linear, $x \in \Fix T_{2}$, and $\Fix T_{2} \subseteq \Fix T^{2}_{1}$, we have that $T_{1} (\tfrac{T_{1}x +x  }{2}) = \tfrac{x +T_{1}x    }{2}$, that is, $\tfrac{T_{1}x +x  }{2} \in \Fix T_{1}$. 
	If $\CC{\mathcal{S}}(x)-\tfrac{T_{1}x +x  }{2} =0$, then $ \CC{\mathcal{S}}(x)=\tfrac{T_{1}x +x  }{2}  \in \Fix T_{1}$. If $T_{1}x =x$, then by \cref{theorem:CCSHRT12:welldefined}, $  \CC{\mathcal{S}}(x) =x  \in \Fix T_{1} $. In the rest of the proof, we assume that $T_{1}x -x \neq 0$ and that $\CC{\mathcal{S}}(x)-\tfrac{T_{1}x +x  }{2} \neq 0$.
	Because $T_{1}$ is linear and isometric, by \cite[Lemma~2.29]{BOyW2019LinearConvergence}, $	T_{1}x -x \in \overline{\Range} (T_{1} -\Id)=\overline{\Range} (\Id -T_{1} ) =( \Fix T_{1})^{\perp}$.
	Combine this with the assumption  $\dim (\pa(\Fix T_{1})^{\perp} )=1$ and $T_{1}x -x \neq 0 $ to obtain that 
	\begin{align} \label{eq:prop:CCSHR:P:spn}
	( \Fix T_{1})^{\perp} =\spn \{ T_{1}x -x \}.
	\end{align}
	By \cref{defn:cir:map}, we know that $\norm{\CC{\mathcal{S}}(x) -x } = \norm{\CC{\mathcal{S}}(x) - T_{1}x } $. Because, by \cref{lemma:normeq}\cref{lemma:normeq:norm}$\Leftrightarrow$\cref{lemma:normeq:0}, $\norm{\CC{\mathcal{S}}(x) -x } = \norm{\CC{\mathcal{S}}(x) - T_{1}x } 
	\Leftrightarrow   \Innp{ \CC{\mathcal{S}}(x)-\tfrac{T_{1}x +x  }{2}, T_{1}x -x } =0$, 
we have that $	\CC{\mathcal{S}}(x)-\tfrac{T_{1}x +x  }{2}  \in  (\spn \{T_{1}x -x\})^{\perp}\stackrel{\cref{eq:prop:CCSHR:P:spn}}{=} ( \Fix T_{1})^{\perp\perp} =\Fix T_{1}$.
	Recall that  $\tfrac{T_{1}x +x  }{2} \in \Fix T_{1}$ and that $\Fix T_{1} $ is a linear subspace. We obtain that $	\CC{\mathcal{S}}(x) \in \tfrac{T_{1}x +x  }{2} + \Fix T_{1} =\Fix T_{1}.$
	Therefore, the proof is complete.
\end{proof}

\begin{corollary} \label{cor:CCSHU}
	Let $u \in \mathcal{H} \smallsetminus \{0\}$ and $\eta \in \mathbb{R}$. Let $H:=\{ x \in \mathcal{H} ~:~  \innp{x, u} =\eta  \}$, and let $U$ be an affine subspace of $\mathcal{H}$ such that $H \cap U \neq \varnothing$. Assume that 	$\mathcal{S} := \{\Id, \R_{H}, \R_{U}\R_{H}  \}$. Then $(\forall x  \in U)$  $\CC{\mathcal{S}}(x) =\Pro_{ H \cap U }x$.
\end{corollary}
\begin{proof}
Because $H$ and $U$ are affine subspaces, by \cite[Lemma~2.23(i)]{BOyW2019Isometry}, $\R_{H}$ and $\R_{U}$ are  isometries. Moreover,   the idempotent property of projectors and the definition of reflectors imply that $\R_{H}^{2}=\Id$ and $\R_{U}^{2}=\Id$. 
Hence, the desired result  follows from 	\cref{theorem:CCSHRT12}\cref{theorem:CCSHRT12:P}.
\end{proof}
The following  \cref{lemma:HW} is necessary in the proof of \cref{prop:CUC} below. In fact,
the result in the  following \cref{lemma:HW}  was directly used in the proofs of  \cite[Lemma~5]{BBCS2020} and   \cite[Proposition~5]{BBCS2020CRMbetter}. For completeness, we provide a clear analytical proof below.
\begin{lemma}\label{lemma:HW}
	Let $u \in \mathcal{H} \smallsetminus \{0\}$ and $\eta \in \mathbb{R}$. Let  $U$ be a closed affine subspace of $\mathcal{H}$. Set $H:=\{ x \in \mathcal{H} ~:~  \innp{x, u} =\eta  \}$ and $W:=\{ x \in \mathcal{H} ~:~  \innp{x, u} \leq \eta  \}$. Assume that $W \cap U \neq \varnothing$ and   $U \smallsetminus W \neq \varnothing$. Then $H \cap U \neq \varnothing$ and  $	(\forall x \in U \smallsetminus W)$ $\Pro_{ H \cap U }x = \Pro_{ W \cap U }x$.
\end{lemma}

\begin{proof}
	Let $x \in U \smallsetminus W$, and let $y \in W \cap U$. Define   $f: \mathbb{R} \to \mathbb{R}$ by $(\forall t \in \mathbb{R})$ $f(t) :=\innp{tx+(1-t)y,u}-\eta$. Because  $y \in W$  and $x \notin W$, we have   that $f(0)=\innp{y, u} -\eta \leq 0$ and $f(1)=\innp{x, u} -\eta >0$. So, there exists $\bar{t} \in \left[0,1\right[\,$ such that $f(\bar{t}) =\innp{\bar{t}x+(1-\bar{t})y,u}-\eta=0$, which implies that, $\bar{t}x+(1-\bar{t})y \in H$.  Because $x \in U$, $y \in U$, and $U$ is affine, we know that $ \bar{t}x+(1-\bar{t})y \in H \cap U \neq \varnothing$. Combine this with the result that $H \cap U$ is a  closed affine subspace, and \cite[Proposition~2.10(iii)]{BOyW2018Proper} to see that 
	\begin{align*}
\norm{ x - \Pro_{ H \cap U }x }^{2} +\norm{\Pro_{ H \cap U }x  - (\bar{t}x+(1-\bar{t})y) }^{2} =	\norm{x- (\bar{t}x+(1-\bar{t})y) }^{2} =(1-\bar{t})^{2} \norm{x-y}^{2} \leq \norm{x-y}^{2}.
	\end{align*}
	Note that $y \in W \cap U$ is arbitrary. Hence, we obtain that $(\forall y \in W \cap U)$ $\norm{ x - \Pro_{ H \cap U }x } \leq \norm{x-y}$,
	which, combining with the fact that $\Pro_{ H \cap U }x \in H\cap U \subseteq W\cap U$,  implies that $\Pro_{ H \cap U }x=\Pro_{ W \cap U }x$.
\end{proof}

The   result   $(\forall x \in U)$ $\CC{\mathcal{S}}(x) = \Pro_{H_{x} \cap U}x$ in the  following \cref{prop:CUC}\cref{prop:CUC:CCS} reduces to \cite[Lemma~3]{BBCS2020} when $\mathcal{H} =\mathbb{R}^{n}$. In fact, the idea of  proof of \cref{prop:CUC}\cref{prop:CUC:CCS} is   almost the same with that of \cite[Lemma~3]{BBCS2020}. The following \cref{prop:CUC}\cref{prop:CUC:ineq} illustrates that the inequality $(4)$ in \cite[Lemma~5]{BBCS2020} is actually an equation.
In the following \cref{prop:CUC}\cref{prop:CUC:conver}, we apply the demiclosedness principle for the circumcenter mapping proved in our \cite[Theorem~3.2]{BOyW2018Proper} to  generalize \cite[Theorem~1]{BBCS2020} from $\mathbb{R}^{n}$ to general Hilbert spaces and from the convergence to the weak convergence.

Note that the reflector associated with general convex set is generally not isometric. Hence, the following result is not a direct corollary of \cref{theorem:CCSHRT12}.

\begin{proposition} \label{prop:CUC}
	Let $C$ be a closed convex set of $\mathcal{H}$ and let $U$ be a closed affine subspace of $\mathcal{H}$ with $C \cap U \neq \varnothing$. Set $	\mathcal{S} := \{ \Id, \R_{C}, \R_{U}\R_{C}   \}$.
	For every $x \in \mathcal{H}$, define
				\begin{align*}
&	H_{x} := \begin{cases}
	\{  y \in \mathcal{H} ~:~ \innp{y, x-\Pro_{C}x} =\innp{\Pro_{C}x,  x-\Pro_{C}x}   \} \quad &\text{if } x \notin C;\\
	C  \quad &\text{if } x \in C,
	\end{cases} \\
&		W_{x} := \begin{cases}
	\{  y \in \mathcal{H} ~:~ \innp{y, x-\Pro_{C}x} \leq \innp{\Pro_{C}x,  x-\Pro_{C}x}   \} \quad &\text{if } x \notin C;\\
	C  \quad &\text{if } x \in C.
	\end{cases}  
	\end{align*}
Let $x \in U$.	Then the following statements hold:
	\begin{enumerate}
		\item  \label{prop:CUC:Hx}   $H_{x} \cap U \neq \varnothing$ and $ \Pro_{H_{x} \cap U}x= \Pro_{W_{x} \cap U}x$.
		\item \label{prop:CUC:CCS} Set $\tilde{\mathcal{S}} := \{  \Id, \R_{H_{x}}, R_{U}\R_{H_{x}}  \}$. Then $\mathcal{S}(x) =\tilde{\mathcal{S}} (x)$ and  $\CC{\mathcal{S}}(x)   = \CC{\tilde{\mathcal{S}}}(x) = \Pro_{H_{x} \cap U}x $.
	
		\item  \label{prop:CUC:formula}   $(\forall z \in C \cap U)$   $\CC{\mathcal{S}}x=\Pro_{\aff (\mathcal{S}(x))} z $. Moreover,
	\begin{align*}
	\CC{\mathcal{S}}(x) &=
	\begin{cases}
	x  \quad \text{if } x \in C;\\
	x+2\alpha ( \Pro_{U}\R_{C}(x) -x ), ~\text{~where~} \alpha=\left( 4- \frac{\norm{\R_{U}\R_{C}x -\R_{C}x}^{2} }{\norm{\R_{C}x -x}^{2}}  \right)^{-1} \in \left[\frac{1}{4}, +\infty \right[ \quad \text{otherwise}.
	\end{cases}
	\end{align*}

		\item \label{prop:CUC:ineq}  $(\forall z \in C\cap U )$ $ \norm{\CC{\mathcal{S}} (x) -x}^{2} +\norm{\CC{\mathcal{S}} (x) -z}^{2} = \norm{x-z}^{2}$.
Consequently, $\CC{\mathcal{S}} : U \to U$ is firmly quasinonexpansive. 

		\item \label{prop:CUC:conver} $(\forall k \in \mathbb{N})$ $\CC{\mathcal{S}}^{k}(x) \in U$. Moreover, there exists $\bar{x} \in K\cap U$ such that $ \CC{\mathcal{S}}^{k}(x) \weakly \bar{x}$.
	\end{enumerate}
\end{proposition}

\begin{proof}
	\cref{prop:CUC:Hx}: If $x \in C$, then by  definitions of $H_{x}$ and $W_{x}$, we have that $x \in C \cap U =H_{x} \cap U \neq \varnothing$ and $x=\Pro_{C\cap U}x=\Pro_{H_{x} \cap U}x= \Pro_{W_{x} \cap U}x$. 
	Assume that $x \notin C$. 	Then by the definition of $W_{x}$, $x \notin W_{x}$. So, $x \in U \smallsetminus W_{x}$. Note that  \cite[Theorem~3.16]{BC2017} yields $C \subseteq W_{x}$, which, 
combined with $C \cap U \neq \varnothing$, implies that $W_{x} \cap U \neq \varnothing$.
	Hence, applying \cref{lemma:HW} with $W$ and $H$ replaced by $W_{x}$ and $H_{x}$ respectively, we see that $H_{x} \cap U \neq \varnothing$ and $ \Pro_{H_{x} \cap U}x= \Pro_{W_{x} \cap U}x$.

	\cref{prop:CUC:CCS}:   If $x \in C$, then $H_{x} =C$. So,   $\mathcal{S}(x) =\tilde{\mathcal{S}} (x) =\{x\} $ and  $\CC{\mathcal{S}}(x)    =x = \Pro_{C \cap U}x = \Pro_{H_{x} \cap U}x$.
	
	Assume $x \notin C$. 	According to \cref{fact:Projec:Hyperplane}, 
	\begin{align*}
	\Pro_{H_{x}}x =x + \frac{\innp{\Pro_{C}x,  x-\Pro_{C}x}  - \innp{ x, x-\Pro_{C}x  }   }{\norm{x-\Pro_{C}x}^{2}}(x-\Pro_{C}x)=x-(x-\Pro_{C}x)=\Pro_{C}x,
	\end{align*}
	which implies that $ \R_{C}x =   \R_{H_{x}}x$ and $ \R_{U}\R_{C}x= \R_{U}\R_{H_{x}}x$. 
	Hence, $\mathcal{S}(x) =\tilde{\mathcal{S}} (x)$ and, by \cref{defn:cir:map},  $\CC{\mathcal{S}}(x)   = \CCO{  \left( \{ x, \R_{C}x, \R_{U}\R_{C}x   \} \right) }
	= \CCO{\left( \{ x, \R_{H_{x}}x,  \R_{U}\R_{H_{x}}x   \} \right) }= \CC{\tilde{\mathcal{S}}}(x) $. 
	
	 Because $x \in U$, and, by \cref{prop:CUC:Hx}, $H_{x} \cap U \neq \varnothing$, applying \cref{cor:CCSHU} with $H=H_{x}$, we know that  
	\begin{align*} 
\CC{\mathcal{S}}(x)  =	\CC{\tilde{\mathcal{S}}}(x)=	\CCO{\left( \{ x, \R_{H_{x}}x,  \R_{U}\R_{H_{x}}x   \} \right) }=\Pro_{H_{x} \cap U}x.
	\end{align*}

	\cref{prop:CUC:formula}: Let $z \in C \cap U$. If $x \in C$, as we did in \cref{prop:CUC:CCS}, it is easy to see that  $\aff \mathcal{S}(x) =\aff \{ x, \R_{C}x, \R_{U}\R_{C}x   \} =\{x\}$ and that $\CC{\mathcal{S}}(x)=x=\Pro_{\aff (\mathcal{S}(x))} z$.
	
	Assume $x \notin C$. Now $U$ is an affine subspace and $H_{x}$ is a hyperplane, which, by \cite[Lemma~2.23(i)]{BOyW2019Isometry}, implies that both $\R_{H_{x}}$ and $\R_{U} $ are isometries. 
	By definition of $H_{x}$ and \cref{fact:Projec:Hyperplane}, it is easy to see that  $\R_{C}x=\R_{H_{x}}x$, and that $x \notin C$ implies that $x \notin H_{x}$ and $x \neq \R_{H_{x}}x$. 
	Hence, applying \cref{theorem:CCSHRT12}\cref{theorem:CCSHRT12:welldefined} with $T_{1}$ and $T_{2}$ replaced by $\R_{H_{x}}$ and   $\R_{U}$ respectively, we obtain $\CC{\tilde{\mathcal{S}}}(x) =\Pro_{\aff (\tilde{\mathcal{S}}(x))} z =	x+2\alpha \left( \frac{ \R_{C}x +\R_{U}\R_{C}x }{2}  -x \right)$, where $\alpha:=\left( 4- \frac{\norm{\R_{U}\R_{C}x -\R_{C}x}^{2} }{\norm{\R_{C}x -x}^{2}}  \right)^{-1} \in \left[\frac{1}{4}, +\infty \right[ \,$.
	Combine these results with  the fact  $\frac{ \R_{C} +\R_{U}\R_{C} }{2} =\Pro_{U}\R_{C} $ and   \cref{prop:CUC:CCS} to obtain the required results.

\cref{prop:CUC:ineq}: Apply \cref{fact:CCSRestrict}\cref{fact:CCSRestrict:fix} with $T_{1}$, $T_{2}$,  and $\mathcal{S}$ replaced by  $\R_{C}$, $ \R_{U}$ and $  \{ \Id, \R_{C}, \R_{U}\R_{C}   \} $ respectively  to obtain that $\Fix  \CC{\mathcal{S}}=C \cap U$.

Let $z \in  C \cap U$. By \cref{prop:CUC:formula},   $\CC{\mathcal{S}}x=\Pro_{\aff (\mathcal{S}(x))} z $. 	Note that the definition of $\mathcal{S}(x)$ yields $x \in \mathcal{S}(x)$, and that $\aff(\mathcal{S}(x))$ is an affine subspace. Hence, by \cite[Proposition~2.10(iii)]{BOyW2018Proper}, we have that 
\begin{align*}
 \norm{z - \Pro_{\aff(\mathcal{S}(x))}z}^{2}+\norm{ \Pro_{\aff(\mathcal{S}(x))}z-x}^{2} =\norm{z-x}^{2} 
\Leftrightarrow \norm{z - \CC{\mathcal{S}}x }^{2}+\norm{ \CC{\mathcal{S}}(x) -x}^{2} =\norm{z-x}^{2},
\end{align*}
which, by \cref{defn:SQNE}, implies that $\CC{\mathcal{S}} : U \to U$ is firmly quasinonexpansive.

\cref{prop:CUC:conver}:  
Because, by 	\cref{prop:CUC:CCS},  $ \CC{\mathcal{S}} : U \to U$ is well-defined, and $\R_{C}$ and $\R_{U}\R_{C}$ are nonexpansive, using
\cref{fact:CCSRestrict}\cref{fact:CCSRestrict:fix}$\&$\cref{fact:CCSRestrict:Iddemiclosed},  we have that $\Fix \CC{\mathcal{S}} =C\cap U \neq \varnothing $ and that $\Id -\CC{\mathcal{S}} $ is  demiclosed at $0$. Moreover, \cref{prop:CUC:ineq} shows that $\CC{\mathcal{S}} : U \to U$ is firmly quasinonexpansive. Because, by \cite[Theorem~3.34]{BC2017}, $U$ is a nonempty closed convex set implies that $U$ is weakly sequentially closed, using
\cref{theorem:quasidemiweakconver}, we obtain that  $(\CC{\mathcal{S}}^{k}(x))_{k \in \mathbb{N}}$ weakly converges to one point in $C \cap U$.
\end{proof}

%%%%%%%%%%%%%%%%%%%%%%%%%%%%%%%%%%%%%%%%%%%%%%%%%%%%%%%%%%%%%%%%%%%%%%%%%%%%%%%%%%\section{One-step convergence of circumcentered isometry methods}%%%
%%%%%%%%%%%%%%%%%%%%%%%%%%%%%%%%%%%%%%%%%%%%%%%%%%%%%%%%%%%%%%%%%%%

\subsection*{One-step convergence of circumcentered isometry methods}

\begin{lemma} \label{lem:FixIsometryContinu}
	Let $F  : \mathcal{H} \to \mathcal{H}$ and $T : \mathcal{H} \to \mathcal{H}$ be linear and continuous. Let $x \in \mathcal{H}$ and let $z \in \mathcal{H}$. 
	\begin{enumerate}
		\item \label{lem:FixIsometryContinu:In} Suppose that $\mathcal{H} =\mathbb{R}^{n}$ and that $F$ is isometric.  Assume that $z \notin \Fix T$,  and that $Fx \in \Fix T$. Then 
		\begin{align*}
		(\forall t \in \mathbb{R} \smallsetminus \{0\} ) \quad  F (x +t F^{*}z) \notin \Fix T.
		\end{align*}
		\item \label{lem:FixIsometryContinu:NotIn}     Assume  that $Fx \notin \Fix T$. Then there exists $\bar{t} \in \mathbb{R}_{++}$ such that
		\begin{align*}
		(\forall t \in [-\bar{t}, \bar{t}]) \quad  F (x +t  z) \notin \Fix T.
		\end{align*}
	\end{enumerate}	
\end{lemma}	

\begin{proof}
	\cref{lem:FixIsometryContinu:In}: Because $\mathcal{H} =\mathbb{R}^{n}$ and $F$ is a linear isometry, by \cite[Lemma~3.9]{BOyW2019LinearConvergence}, $FF^{*}= \Id$. 
	Let $t \in \mathbb{R} \smallsetminus \{0\} $. Then $ F (x +t F^{*}z) = Fx +t FF^{*}z= Fx +t z $. Because $\Fix T$ is a linear subspace, by assumptions $Fx \in \Fix T$ and $ z \notin \Fix T$,  $F (x +t F^{*}z) =Fx +t z \notin \Fix T$.
	
	\cref{lem:FixIsometryContinu:NotIn}: Define $f :\mathbb{R} \to \mathbb{R}$ by $	(\forall t \in \mathbb{R})$ $f(t) := \dist (F (x +t z) , \Fix T )$. Note that
 $F$ is continuous, that $\Fix T$ is a nonempty closed  linear subspace of $\mathcal{H}$,  and  that the distance function $\dist(\cdot, \Fix T)$ is continuous. Hence, $f$ is continuous.  Because  $\Fix T$ is closed, $Fx \notin \Fix T$ yields $f(0) = \dist (F x , \Fix T ) >0$. Hence, there exists $\bar{t}>0$ such that $(\forall t \in [-\bar{t}, \bar{t}]) $ $ \dist (F (x +t z) , \Fix T )=f(t) >0$, which implies that $F (x +t z) \notin \Fix T$.
\end{proof}

\begin{lemma} \label{lemma:FixSpan}
 Let $(\forall i \in \I)$ $T_{i}$ is linear and nonexpansive, and let $x \in \mathcal{H}$. The following statements hold:
	\begin{enumerate}
		\item \label{lemma:FixSpan:subseteq}  
	$\spn \{ T_{1}x -x, T_{2}T_{1}x -T_{1}x ,  \dots, T_{m}\cdots T_{2}T_{1}x -T_{m-1}\cdots T_{2}T_{1}x \}   \subseteq \left( \cap^{m}_{i=1} \Fix T_{i} \right)^{\perp}$ and $\spn \{ T_{1}x -x, T_{2}x -x ,  \ldots, T_{m}x -x \}  \subseteq \left( \cap^{m}_{i=1} \Fix T_{i} \right)^{\perp}$.

		\item \label{lemma:FixSpan:=} Suppose that  $(\forall i \in \I)$ $\left( \Fix T_{i} \right)^{\perp}$ is one-dimensional. 
			\begin{enumerate}
			\item \label{lemma:FixSpan:=:Ti} Assume that $(\forall i \in \I)$ $x \notin \Fix T_{i} $. Then  $\spn \{ T_{1}x -x, T_{2}x -x ,  \ldots, T_{m}x -x \}  = \left( \cap^{m}_{i=1} \Fix T_{i} \right)^{\perp}$. 
			\item \label{lemma:FixSpan:=:TiT1} Assume that $(\forall i \in \I)$ $T_{i-1}\cdots T_{1} x \notin \Fix T_{i} $.\footnotemark
			\footnotetext{We use the empty product convention, so if $i=1$, $T_{i-1}\cdots T_{1} =\Id$.}
			Then  $\spn \{ T_{1}x -x, T_{2}T_{1}x -T_{1}x ,  \dots, T_{m}\cdots T_{2}T_{1}x -T_{m-1}\cdots T_{2}T_{1}x \}   = \left( \cap^{m}_{i=1} \Fix T_{i} \right)^{\perp}$.
		\end{enumerate}
	\end{enumerate}
\end{lemma}

\begin{proof}
	Because $(\forall i \in \I)$ $T_{i}$ is linear and nonexpansive, due to \cite[Lemma~2.29]{BOyW2019LinearConvergence}, 
	\begin{align} \label{eq:lem:Rn:CCSP:Contain}
	(\forall i \in \I) \quad  \overline{\Range} (\Id -T_{i}) =( \Fix T_{i} )^{\perp}.
	\end{align}
Moreover, due to \cite[Theorems~4.6(5)]{D2012}, 
		\begin{align} \label{eq:FixTiPerp}
		\left( \cap_{i \in \I} \Fix T_{i} \right)^{\perp}
	 = \overline{ \sum_{i \in \I}  (\Fix T_{i} )^{\perp}  }
	 \stackrel{\cref{eq:lem:Rn:CCSP:Contain}}{=}	\overline{ \sum_{i \in \I}  \overline{\Range} (\Id -T_{i}) }.		
		\end{align}
		
	\cref{lemma:FixSpan:subseteq}:  Recall that $(\forall i \in \I)$ $T_{i} -\Id$ is linear. So, $\spn \{ T_{1}x -x, T_{2}T_{1}x -T_{1}x ,  \dots, T_{m}\cdots T_{2}T_{1}x -T_{m-1}\cdots T_{2}T_{1}x \}   
= \sum_{i \in \I} \spn \{ T_{i}T_{i-1}\cdots T_{1}x -T_{i-1}\cdots T_{1}x\} \subseteq \sum_{i \in \I} \Range (\Id -T_{i}) \stackrel{\cref{eq:FixTiPerp}}{\subseteq}\left( \cap^{m}_{i=1} \Fix T_{i} \right)^{\perp}$, and $\spn \{ T_{1}x -x, T_{2}x -x ,  \ldots, T_{m}x -x \}  =\sum_{i \in \I} \spn \{T_{i}x-x\}  \subseteq \sum_{i \in \I} \Range (\Id -T_{i}) \stackrel{\cref{eq:FixTiPerp}}{\subseteq}\left( \cap^{m}_{i=1} \Fix T_{i} \right)^{\perp}$.

	\cref{lemma:FixSpan:=}: Because $(\forall i \in \I)$ $\left( \Fix T_{i} \right)^{\perp}$ is one-dimensional,   the finite-dimensional linear space $\sum_{i \in \I}  (\Fix T_{i} )^{\perp} $ is closed. Combine this with
	\cite[Theorems~4.6(5)]{D2012} to see that
	\begin{align} \label{eq:FixTiPerp:Fix}
	\left( \cap_{i \in \I} \Fix T_{i} \right)^{\perp}
	= \overline{ \sum_{i \in \I}  (\Fix T_{i} )^{\perp}  }
	= \sum_{i \in \I}  (\Fix T_{i} )^{\perp}.  
	\end{align}
	
	\cref{lemma:FixSpan:=:Ti}: Because $(\forall i \in \I)$ $x \notin \Fix T_{i} $ and $\left( \Fix T_{i} \right)^{\perp}$ is one-dimensional,  we have that $(\forall  i \in \I)$ $ \spn \{T_{i}x-x\} =\overline{\Range} (\Id -T_{i}) \stackrel{\cref{eq:lem:Rn:CCSP:Contain}}{=} \left( \Fix T_{i} \right)^{\perp}$.  Therefore, $\spn \{ T_{1}x -x, T_{2}x -x ,  \ldots, T_{m}x -x \}  =\sum_{i \in \I} \spn \{T_{i}x-x\}  =\sum_{i \in \I} \left( \Fix T_{i} \right)^{\perp} \stackrel{\cref{eq:FixTiPerp:Fix}}{=} \left( \cap_{i \in \I} \Fix T_{i} \right)^{\perp}$.

	\cref{lemma:FixSpan:=:TiT1}: Because $(\forall i \in \I)$ $T_{i-1}\cdots T_{1} x \notin \Fix T_{i} $ and $\left( \Fix T_{i} \right)^{\perp}$ is one-dimensional,  we have that $(\forall  i \in \I)$ $ \spn \{T_{i}T_{i-1}\cdots T_{1} x-T_{i-1}\cdots T_{1} x\} =\overline{\Range} (\Id -T_{i})  \stackrel{\cref{eq:lem:Rn:CCSP:Contain}}{=} \left( \Fix T_{i} \right)^{\perp}$.  
	Hence, 
$
	\spn \{ T_{1}x -x, T_{2}T_{1}x -T_{1}x ,  \dots, T_{m}\cdots T_{2}T_{1}x -T_{m-1}\cdots T_{2}T_{1}x \}  =\sum_{i \in \I} \spn \{T_{i}T_{i-1}\cdots T_{1} x-T_{i-1}\cdots T_{1} x\}  =\sum_{i \in \I} \left( \Fix T_{i} \right)^{\perp} \stackrel{\cref{eq:FixTiPerp:Fix}}{=} \left( \cap_{i \in \I} \Fix T_{i} \right)^{\perp}.
$	 
\end{proof}

The following result provides sufficient conditions for the circumcentered isometry methods   finding the best approximation point in one step.  In fact,  \cref{prop:Tn:CCSP} under the assumption  \cref{prop:Tn:CCSP:T12mTm} below and \cref{prop:Rn:CCSP:Condition} reduce to  \cite[Lemma~3]{BCS2019}  and \cite[Lemma~4]{BCS2019}, respectively, when $\mathcal{H} =\mathbb{R}^{n}$ and $(\forall i \in \{1, \ldots, m\})$ $T_{i} =\R_{H_{i}}$ with $H_{i}$ being a hyperplane. 
\begin{theorem} \label{prop:Tn:CCSP}
	Let $T_{1}, \ldots, T_{m}$ be  isometries from $\mathcal{H} $ to $\mathcal{H}$. Assume that $z \in \cap^{m}_{i =1} \Fix T_{i} \neq \varnothing$ and that $(\forall i \in \I)$ $(\Fix T_{i} -z)^{\perp}$ is one-dimensional linear subspace of $\mathcal{H}$.  Let $x \in \mathcal{H}$. 
	Assume that  one of the following items hold.
	\begin{enumerate}
		\item \label{prop:Tn:CCSP:T1Tm}  $\mathcal{S} :=\{ \Id,  T_{1},T_{2},  \ldots, T_{m} \}$ 
		and  $(\forall i \in \I)$ $x \notin \Fix T_{i}$. 
		\item  \label{prop:Tn:CCSP:T12mTm}    $\mathcal{S} := \{ \Id,  T_{1}, T_{2}T_{1},  \ldots, T_{m}\cdots T_{2}T_{1} \} $ and  $(\forall i \in \I)$ $T_{i-1}\cdots T_{1}x \notin \Fix T_{i}$. 
	\end{enumerate}	
	Then $\CC{\mathcal{S}}x =  \Pro_{\cap^{m}_{i=1} \Fix T_{i}}x$.
\end{theorem}	

\begin{proof}
	Because $T_{1}, \ldots, T_{m}$ are   isometries with $z\in \cap_{i \in \I} \Fix T_{i} \neq \varnothing$ and compositions of isometries are affine  isometries, applying \cref{lemma:affineTOlinear}\cref{lemma:affineTOlinear:F}$\&$\cref{lemma:affineTOlinear:EQ}, without loss of generality, we assume that $(\forall i \in \I)$ $T_{i}$ is linear and that $z=0$.

	Assume that \cref{prop:Tn:CCSP:T1Tm} or \cref{prop:Tn:CCSP:T12mTm} holds. Denote by $\widehat{\aff} \mathcal{S}(x) := \pa \aff \mathcal{S}(x)$. Note that $\Id \in \mathcal{S}$ implies that $x \in  \aff \mathcal{S}(x)$. Then $ \widehat{\aff} \mathcal{S}(x)= \aff \mathcal{S}(x)  -x$. 
In view of  assumptions and \cref{lemma:FixSpan}\cref{lemma:FixSpan:=}, we have that $\widehat{\aff} \mathcal{S}(x)= \left( \cap^{m}_{i=1} \Fix T_{i} \right)^{\perp}$.  This combined with \cite[Theorems~5.8(6)]{D2012} implies that 
	\begin{align}\label{eq:prop:Tn:CCSP}
	(\widehat{\aff} \mathcal{S}(x) )^{\perp}= \left( \cap^{m}_{i=1} \Fix T_{i} \right)^{\perp \perp} =  \cap^{m}_{i=1} \Fix T_{i}.
	\end{align}
	Using  \cite[Proposition~4.2(iii)]{BOyW2019Isometry} and    \cite[Theorems~5.8(2)]{D2012}, we have that 
	\begin{align} \label{lem:Rn:CCSP:onePart}
	\Pro_{\cap^{m}_{i=1} \Fix T_{i}}x -  \CC{\mathcal{S}}x  = \Pro_{\cap^{m}_{i=1} \Fix T_{i}}\CC{\mathcal{S}}x - \CC{\mathcal{S}}x =  -  \Pro_{\left(\cap^{m}_{i=1} \Fix T_{i}\right)^{\perp} }(\CC{\mathcal{S}}(x) ) \in \left( \cap^{m}_{i=1} \Fix T_{i} \right)^{\perp}.
	\end{align} 
On the other hand, by \cref{fact:CCSRestrict}\cref{fact:CCSRestrict:P:CCS}, \cref{fac:SetChangeProje} and  \cite[Theorems~5.8(2)]{D2012}, we obtain that
		\begin{align*}
	\Pro_{\cap^{m}_{i=1} \Fix T_{i}}x -  \CC{\mathcal{S}}x    &=  \Pro_{\cap^{m}_{i=1} \Fix T_{i}}x -  \Pro_{ \aff \mathcal{S}(x)} \Pro_{\cap^{m}_{i=1} \Fix T_{i}}x  
		 = \Pro_{\cap^{m}_{i=1} \Fix T_{i}}x - \left( x + \Pro_{ \widehat{\aff} \mathcal{S}(x)} (\Pro_{\cap^{m}_{i=1} \Fix T_{i}}(x)  -x) \right) \\
&	= \Pro_{ ( \widehat{\aff} \mathcal{S}(x))^{\perp}} (\Pro_{\cap^{m}_{i=1} \Fix T_{i}}(x)  -x) \in  ( \widehat{\aff} \mathcal{S}(x))^{\perp} \stackrel{\cref{eq:prop:Tn:CCSP}}{=} \cap^{m}_{i=1} \Fix T_{i},
	\end{align*}
	which, combined with \cref{lem:Rn:CCSP:onePart}, forces that 
 $\Pro_{\cap^{m}_{i=1} \Fix T_{i}}x -  \CC{\mathcal{S}}x \in \left( \cap^{m}_{i=1} \Fix T_{i} \right)^{\perp} \cap \left( \cap^{m}_{i=1} \Fix T_{i}  \right)= \{0\}$, that is,  $\Pro_{\cap^{m}_{i=1} \Fix T_{i}}x =  \CC{\mathcal{S}}x $. 
\end{proof}

In the following \Cref{prop:Rn:CCSTmT1:Condition,prop:Rn:CCSP:Condition}, we affirm that the sets of points $x$ satisfying the conditions presented in \cref{prop:Tn:CCSP} are actually dense in $\mathcal{H}$. Note that in the proof of \cref{prop:Rn:CCSP:Condition}, $\mathcal{H} =\mathbb{R}^{n}$ is required, while it is not necessary in \cref{prop:Rn:CCSTmT1:Condition}.

\begin{proposition} \label{prop:Rn:CCSTmT1:Condition}
	Let $(\forall i \in \I)$ $T_{i} : \mathcal{H} \to \mathcal{H}$ be isometric such that $z \in \cap^{m}_{i=1} \Fix T_{i} \neq \varnothing $ and $(\forall i \in \I)$ $(\Fix T_{i} -z)^{\perp} \neq \{0\}$. Set $ \mathcal{S} := \{ \Id,  T_{1}, T_{2},  \ldots, T_{m} \}  $.  
	Let $x \in \mathcal{H}$.  Then  for every $\epsilon \in \mathbb{R}_{++}$, there exists $y_{x} \in \mathbf{B}[x;\epsilon]$ such that $ \Pro_{\cap^{m}_{i=1} \Fix T_{i}}x  = \Pro_{\cap^{m}_{i=1} \Fix T_{i}} y_{x}$ and 
	$(\forall i \in \I)$ $y_{x} \notin \Fix T_{i}$. 	
	Consequently, there exists a sequence $(x_{k} )_{k \in \mathbb{N}}$ in $\mathcal{H}$ such that $ \lim_{k \to \infty} x_{k} =x$, $ (\forall k \in \mathbb{N})$ $ \Pro_{\cap^{m}_{i=1} \Fix T_{i}}x  = \Pro_{\cap^{m}_{i=1} \Fix T_{i}} x_{k}$ and $(\forall i \in \{1, \ldots, m\})$ $ x_{k} \notin \Fix T_{i}$. 
\end{proposition}

\begin{proof}
	Let  $\epsilon \in \mathbb{R}_{++}$.		Because $(\forall i \in \I)$ $T_{i}: \mathcal{H}  \to \mathcal{H} $ is    isometric, using \cite[Proposition~3.4]{BOyW2019LinearConvergence}, we know that  $(\forall i \in \I)$ $T_{i}: \mathcal{H}  \to \mathcal{H} $  is affine.
		
  Define $(\forall i \in \I)$ $(\forall x \in \mathcal{H})$  $F_{i}x := T_{i} (x+z) -z$. Then by \cref{lemma:affineTOlinear}\cref{lemma:affineTOlinear:F}, $(\forall i \in \I)$  $F_{i}$ is linear, $\Fix F_{i} =\Fix T_{i} -z$, and 	$\cap^{m}_{i=1} \Fix F_{i} =\cap^{m}_{i=1}  \Fix T_{i}   -z$. 
 Moreover, bearing \cref{fac:SetChangeProje} in mind, we notice that there exists $y_{x} \in \mathbf{B}[x;\epsilon]$ such that $ \Pro_{\cap^{m}_{i=1} \Fix T_{i}}x  = \Pro_{\cap^{m}_{i=1} \Fix T_{i}} y_{x}$ and 
	 $(\forall i \in \I )$ $y_{x} \notin \Fix T_{i}$ if and only if   there exists $z_{x}:=y_{x}-z \in \mathbf{B}[x-z;\epsilon]$ such that $ \Pro_{\cap^{m}_{i=1} \Fix F_{i}}(x -z) = \Pro_{\cap^{m}_{i=1} \Fix F_{i}} (z_{x})$ and 
	 $(\forall i \in \I)$ $z_{x} \notin \Fix F_{i}$.	 
Hence, without loss of generality, we assume that $(\forall i \in \I)$ $T_{i}: \mathcal{H}  \to \mathcal{H} $ is linearly isometric and $z =0$.

	Suppose $m=1$. If $x \notin \Fix T_{1}$,  then  $y_{x}  =x$. If $x \in \Fix T_{1}$, then take $w \in (\Fix T_{1})^{\perp} \smallsetminus \{0\}$, $y_{x}=x+\frac{\epsilon}{\norm{w}}w$. Because $w \in (\Fix T_{1})^{\perp} $,  $\Fix T_{1}$ is   linear, and $\Pro_{  \Fix T_{1}}$ is   linear, we know that in both cases,  $y_{x} \in \mathbf{B}[x;\epsilon]$, $ \Pro_{  \Fix T_{1}}x  = \Pro_{  \Fix T_{1}} y_{x}$ and 
	  $y_{x} \notin \Fix T_{1}$.

	In the rest of the proof, we assume that $m \geq 2$.	
	If $(\forall i \in \I)$ $ x \notin \Fix T_{i}$, then the proof is done with $y_{x} =x$. Otherwise, assume that $j $ is the smallest index in $ \I:=\{1, \ldots, m\} $ such that $x \in  \Fix T_{j} $.  
	
	\textbf{Step~1:} Take $z_{j} \in (\Fix T_{j})^{\perp} \smallsetminus \{0\}$.
	Because $x \in \Fix T_{j} $,  $z_{j} \notin \Fix T_{j} $ and $\Fix T_{j}$ is a linear subspace, we know  that $(\forall t \in \mathbb{R} \smallsetminus \{0\})$ $x + t z_{j} \notin \Fix T_{j} $. 
	
	In addition, for every $i \in \{1, \ldots, j-1\}$, because $x \notin \Fix T_{i}$,   by applying \cref{lem:FixIsometryContinu}\cref{lem:FixIsometryContinu:NotIn} with $F=\Id, T=T_{i}$, and $z=z_{j}$, we have that there exists $\overline{t}_{i} \in \mathbb{R}_{++}$ such that  $(\forall t \in  [-\overline{t}_{i}, \overline{t}_{i}])$ $x + t z_{j} \notin \Fix T_{i}$. Set $\alpha_{j} := \min_{1 \leq i \leq j-1} \overline{t}_{i}$  and take $y_{j} := x + \min \{ \frac{\epsilon}{\norm{z_{j}}m }, \alpha_{j}   \}z_{j}$.  Then we obtain that
	$(\forall i \in \{1, \ldots, j\})$   $y_{j} \notin \Fix T_{i}$,  and $\norm{y_{j} -x} \leq \frac{\epsilon}{m}$. 
	
	Furthermore, using the linearity of $\Pro_{\cap^{m}_{i=1} \Fix T_{i}}$, $z_{j} \in (\Fix T_{j})^{\perp} \subseteq (\cap^{m}_{i=1} \Fix T_{i} )^{\perp}$, and   \cite[Theorems~5.8(4)]{D2012}, we have that  $ \Pro_{\cap^{m}_{i=1} \Fix T_{i}} y_{j} =  \Pro_{\cap^{m}_{i=1} \Fix T_{i}} x +  \min \{ \frac{\epsilon}{\norm{z_{j}}m }, \alpha_{j}   \} \Pro_{\cap^{m}_{i=1} \Fix T_{i}}  z_{j} =  \Pro_{\cap^{m}_{i=1} \Fix T_{i} } x $ .
	
	\textbf{Recursive Steps:} If $j=m$, we go to Step~$m-j +2$. Otherwise, for $k =j, \ldots, m-1  $ successively,  if $y_{k} \notin \Fix T_{k+1}$, then take $y_{k+1} =y_{k} $. Otherwise, we repeat Step~1 by substituting $x =y_{k}$, $j=k+1$ to obtain $y_{k+1}$ such that 	$(\forall i \in \{1, \ldots, k+1\})$   $y_{k+1} \notin \Fix T_{i}$,    $\norm{y_{k+1} -y_{k}} \leq \frac{\epsilon}{m}$, and $ \Pro_{\cap^{m}_{i=1} \Fix T_{i}} y_{k+1} =   \Pro_{\cap^{m}_{i=1} \Fix T_{i}} y_{k} =  \Pro_{\cap^{m}_{i=1} \Fix T_{i} } x $.

\textbf{Step~$\mathbf{m-j +2}$:} Note that from our Step~$1$ and Recursive Steps,  we got that  $(\forall k \in \{ j, \ldots, m\})$ $y_{k}$ satisfying  $\norm{y_{k} -y_{k-1}} \leq \frac{\epsilon}{m}$, where $y_{j-1}=x$. Moreover, in our Recursive Steps with $k=m-1$, we obtained $y_{m}$ such that $(\forall i \in \I)$   $y_{m} \notin \Fix T_{i}$,   and $ \Pro_{\cap^{m}_{i=1} \Fix T_{i}} y_{m} =     \Pro_{\cap^{m}_{i=1} \Fix T_{i} } x $.

Now take $y_{x}  =y_{m}$. Then $  \Pro_{\cap^{m}_{i=1} \Fix T_{i}} y_{x}=\Pro_{\cap^{m}_{i=1} \Fix T_{i}}x  $ and 
$(\forall i \in \I)$ $y_{x} \notin \Fix T_{i}$. Moreover, 
\begin{align*}
\norm{y_{x}-x} =\norm{y_{m}-x} \leq \norm{y_{m} -y_{j}} +\norm{y_{j}-x} \leq \norm{y_{j}-x}  + \sum^{m-1}_{k=j} \norm{y_{k+1}-y_{k}} \leq  \frac{\epsilon}{m} + (m-j) \frac{\epsilon}{m} \leq  \epsilon.
 \end{align*} 
Altogether, the proof is complete. 
\end{proof}

\begin{proposition} \label{prop:Rn:CCSP:Condition}
	Suppose that $\mathcal{H} =\mathbb{R}^{n}$.	Let $(\forall i \in \I)$ $T_{i} : \mathcal{H} \to \mathcal{H}$ be isometric such that $z \in \cap^{m}_{i=1} \Fix T_{i} \neq \varnothing $ and $(\forall i \in \I)$ $(\Fix T_{i} -z)^{\perp} \neq \{0\}$. 	Set $ \mathcal{S} := \{ \Id,  T_{1}, T_{2}T_{1},  \ldots, T_{m}\cdots T_{2}T_{1} \}  $.
	Let $x \in \mathcal{H}$. Then  for every $\epsilon \in \mathbb{R}_{++}$,  there exists $y_{x} \in \mathbf{B}[x;\epsilon]$ such that $ \Pro_{\cap^{m}_{i=1} \Fix T_{i}}x  = \Pro_{\cap^{m}_{i=1} \Fix T_{i}} y_{x}$ and 
	$(\forall i \in \I)$ $T_{i-1}\cdots T_{1} y_{x} \notin \Fix T_{i}$.
	Consequently, there exists a sequence $(x_{k} )_{k \in \mathbb{N}}$ in $\mathcal{H}$ such that $ \lim_{k \to \infty} x_{k} =x$, $ (\forall k \in \mathbb{N})$ $ \Pro_{\cap^{m}_{i=1} \Fix T_{i}}x  = \Pro_{\cap^{m}_{i=1} \Fix T_{i}} x_{k}$, and $(\forall i \in \I)$ $T_{i-1}\cdots T_{1} x_{k} \notin \Fix T_{i}$.  
\end{proposition}	

\begin{proof}
	Similarly with what we explained in the proof of \cref{prop:Rn:CCSTmT1:Condition}, without loss of generality, we assume that $(\forall i \in \I)$ $T_{i}: \mathcal{H}  \to \mathcal{H} $ is  a linear isometry and $z=0$.

	If $m=1$, then $\mathcal{S} =\{\Id, T_{1} \}$. Hence, the required result was proved in   \cref{prop:Rn:CCSTmT1:Condition}.
	
	Suppose that $m \geq 2$.
	If $(\forall i \in\I)$ $T_{i-1}\cdots T_{1} x \notin \Fix T_{i}$, then the proof is done with $y_{x} =x$. Otherwise, assume that $j $ is the smallest index in $ \I:=\{1, \ldots, m\} $ such that $T_{j-1}\cdots T_{1} x \in  \Fix T_{j} $.  
	
	\textbf{Step~1:} Take $z_{j} \in (\Fix T_{j})^{\perp} \smallsetminus \{0\}$.
	Because $\mathcal{H} =\mathbb{R}^{n}$, and $T_{j-1}\cdots T_{1} x \in \Fix T_{j} $, then applying \cref{lem:FixIsometryContinu}\cref{lem:FixIsometryContinu:In} with  $F=T_{j-1}\cdots T_{1}$, $T=T_{j}$ and $z=z_{j}$,  we have that $(\forall t \in \mathbb{R} \smallsetminus \{0\})$ $ T_{j-1}\cdots T_{1}(x + t T^{*}_{1}T^{*}_{2} \cdots T^{*}_{j-1} z_{j}) =T_{j-1}\cdots T_{1}x +t z_{j} \notin \Fix T_{j} $. 
	
	In addition, since $(\forall i \in \{1, \ldots, j-1\})$ $T_{i-1}\cdots T_{1} x \notin \Fix T_{i}$, thus  applying \cref{lem:FixIsometryContinu}\cref{lem:FixIsometryContinu:NotIn} with $F=T_{i-1}\cdots T_{1} , T=T_{i}$,  and $z=T^{*}_{1}T^{*}_{2} \cdots T^{*}_{j-1} z_{j}$, we have that there exists $\overline{t}_{i} \in \mathbb{R}_{++}$ such that  $(\forall t \in  [-\overline{t}_{i}, \overline{t}_{i}])$ $T_{i-1}\cdots T_{1}(x + t T^{*}_{1}T^{*}_{2} \cdots T^{*}_{j-1} z_{j}) \notin \Fix T_{i}$. 
	
	Set $\alpha_{j} := \min_{1 \leq i \leq j-1} \overline{t}_{i}$  and take $y_{j} := x + \min \{ \frac{\epsilon}{ m \norm{  z_{j}}}, \alpha_{j}   \}T^{*}_{1}T^{*}_{2} \cdots T^{*}_{j-1} z_{j}$. Then we obtain that
	$(\forall i \in \{1, \ldots, j\})$    $T_{i-1}\cdots T_{1}y_{j} \notin \Fix T_{i}$, and 
	$\norm{ y_{j} - x } \leq \frac{\epsilon}{m}$, since $(\forall i \in \I)$ $T_{i}$ is isometric implies that $(\forall i \in \I)$ $\norm{T^{*}_{i}}=\norm{T_{i}} \leq 1 $. 
Furthermore, because $(\forall i \in \I)$ $T_{i}$ is linearly isometric and $\cap^{m}_{k=1} \Fix T^{*}_{k} \subseteq \Fix T^{*}_{i} $, by  \cite[Lemma~2.1]{BDHP2003}  and  \cite[Lemma~3.7]{BOyW2019LinearConvergence}, 
	\begin{align} \label{eq:lem:Rn:CCSP:Condition}
	\Pro_{\cap^{m}_{t=1} \Fix T_{t}}  T^{*}_{1}T^{*}_{2} \cdots T^{*}_{j-1} =\Pro_{\cap^{m}_{t=1} \Fix T^{*}_{t}}  T^{*}_{1}T^{*}_{2} \cdots T^{*}_{j-1} =\Pro_{\cap^{m}_{t=1} \Fix T^{*}_{t}}=\Pro_{\cap^{m}_{t=1} \Fix T_{t}}.
	\end{align}
Employing $\Pro_{\cap^{m}_{i=1} \Fix T_{i}}$,  $z_{j} \in (\Fix T_{j})^{\perp} \subseteq (\cap^{m}_{i=1} \Fix T_{i} )^{\perp}$, and   \cite[Theorems~5.8(4)]{D2012},  we establish that 
	\begin{align*}
	\Pro_{\cap^{m}_{i=1} \Fix T_{i}} y_{j} &~ =~  \Pro_{\cap^{m}_{i=1} \Fix T_{i}} x +  \min \{ \frac{\epsilon}{ m \norm{  z_{j}}}, \alpha_{j}   \} \Pro_{\cap^{m}_{i=1} \Fix T_{i}}  T^{*}_{1}T^{*}_{2} \cdots T^{*}_{j-1} z_{j} \\
	& \stackrel{\cref{eq:lem:Rn:CCSP:Condition}}{=}  \Pro_{\cap^{m}_{i=1} \Fix T_{i}} x +  \min \{ \frac{\epsilon}{ m \norm{   z_{j}}}, \alpha_{j}   \} \Pro_{\cap^{m}_{i=1} \Fix T_{i}} z_{j} =\Pro_{\cap^{m}_{i=1} \Fix T_{i}} x.
	\end{align*}  
	
		\textbf{Recursive Steps:} If $j=m$, we go to Step~$m-j +2$. Otherwise, for $k =j, \ldots, m-1  $ successively,  if $T_{k}\cdots T_{1}y_{k} \notin \Fix T_{k+1}$, then take $y_{k+1} =y_{k} $. Otherwise, we repeat Step~1 by replacing $x =y_{k}$ and  $j=k+1$ to obtain $y_{k+1}$ such that 	$(\forall i \in \{1, \ldots, k+1\})$   $T_{i-1}\cdots T_{1}y_{k+1} \notin \Fix T_{i}$,    $\norm{y_{k+1} -y_{k}} \leq \frac{\epsilon}{m}$, and $ \Pro_{\cap^{m}_{i=1} \Fix T_{i}} y_{k+1} =   \Pro_{\cap^{m}_{i=1} \Fix T_{i}} y_{k} =  \Pro_{\cap^{m}_{i=1} \Fix T_{i} } x $.

	\textbf{Step~$\mathbf{m-j+2}$:}  This is almost the same  as the Step~$ m-j+2$ in the proof of \cref{prop:Rn:CCSTmT1:Condition}.
Altogether, the proof is complete.
\end{proof}

%%%%%%%%%%%%%%%%%%%%%%%%%%%%%%%%%%%%%%%%%%%%%%%%%%%%%%%%%%%%%%%%%%%%%%%%%%%%%%%%%%\section{CRMs associated with two hyperplanes}%%%%%%%%%%%%%%%%%%%%
%%%%%%%%%%%%%%%%%%%%%%%%%%%%%%%%%%%%%%%%%%%%%%%%%%%%%%%%%%%%%%%%%%%
\section{Finite convergence of CRMs associated with  hyperplanes and halfspaces} \label{sec:CRM:hyperplanes}
Throughout this section,  
for every $ i \in \I := \{1, \ldots, m \}$,  let $u_{i}$  be in $\mathcal{H}$,  let $\eta_{i}$ be in $\mathbb{R}$, and set 
\begin{empheq}[box=\mybluebox]{equation} \label{eq:WH}
W_{i} := \{x \in \mathcal{H} ~:~ \innp{x,u_{i}} \leq \eta_{i} \} \quad \text{and} \quad H_{i} := \{x \in \mathcal{H} ~:~ \innp{x,u_{i}} = \eta_{i} \}.
\end{empheq}

In this section, we investigate the finite convergence of CRMs induced by sets of reflectors associated with halfspaces and hyperplanes for solving the related best approximation or feasibility problems.

According to \cite[Theorems~4.17 and 6.5]{Oy2020ProjectionHH},
if $u_{1}$ and $u_{2}$ are linearly dependent  or  orthogonal, then $\Pro_{ W_{2} }\Pro_{ W_{1} }=\Pro_{ W_{1} \cap W_{2}  }$ and  $\Pro_{ W_{2} }\Pro_{ H_{1} }=\Pro_{ H_{1} }\Pro_{ W_{2} }=\Pro_{ H_{1} \cap W_{2}  }$; moreover, if $u_{1}$ and $u_{2}$ are linearly independent, we can always find corresponding sequences of iterations of compositions of projections onto halfspaces and hyperplanes with initial points $x \in \mathcal{H}$ converging linearly to $\Pro_{ W_{1} \cap W_{2}  }x$ or $\Pro_{ H_{1} \cap W_{2}  }x$.

For completeness, we shall explore the  performance of CRMs associated with  hyperplanes and halfspaces  in all cases.

\subsection*{CRMs associated with  hyperplanes}
Note that the following \cref{cor:Tn:CCSP:R}\cref{item:cor:Tn:CCSP:R:S2} is first proved  in \cite[Lemma~3]{BCS2019} when $\mathcal{H} =\mathbb{R}^{n}$.
\begin{corollary} \label{cor:Tn:CCSP:R}
	Assume that $ \cap^{m}_{i =1} H_{i} \neq \varnothing$.  Let $x \in \mathcal{H}$.  Then the following statements hold. 
	\begin{enumerate}
		\item \label{item:cor:Tn:CCSP:R:S1} Set $\mathcal{S}:=\{\Id, \R_{H_{1}}, \ldots, \R_{H_{m}} \}$. If $x \notin \cup^{m}_{i=1} H_{i}$, then $\CC{\mathcal{S}}x =  \Pro_{\cap^{m}_{i=1}  H_{i}}x$.
		\item \label{item:cor:Tn:CCSP:R:S2}   Set $\mathcal{S}:=\{\Id, \R_{H_{1}}, \R_{H_{2}}\R_{H_{1}}, \ldots, \R_{H_{m}}\cdots\R_{H_{2}}\R_{H_{1}} \}$. If $(\forall i \in \I)$ $\R_{H_{i-1}}\cdots\R_{H_{1}} x \notin H_{i} $, then $\CC{\mathcal{S}}x = \Pro_{\cap^{m}_{i=1}  H_{i}}x$.
	\end{enumerate}
\end{corollary}

\begin{proof}
	By \cite[Lemmas~2.37]{BOyW2019Isometry}, $(\forall i \in \I)$  $\R_{H_{i}}$ is an affine isometry with $\Fix \R_{H_{i}} =H_{i} $. 
Let  $z \in \cap_{i \in \I} H_{i}$. Note that 
	$(\forall i \in \I) $ $ \dim \left( (\Fix \R_{H_{i}} )^{\perp} -z \right) =\dim \left( ( H_{i} -z)^{\perp}  \right) =\dim \left( ( \ker u_{i})^{\perp}  \right) =\dim \left( \spn \{ u_{i} \} \right) = 1$. 
	
	Therefore, the required results follow from  \cref{prop:Tn:CCSP}
with $T_{1} =\R_{H_{1}}, T_{2} =\R_{H_{2}}, \ldots, T_{m} =\R_{H_{m}}$.
\end{proof}

\cref{them:Tn:CCSP:R:S1} below illustrates that  the CRM induced by $\mathcal{S} :=\{\Id, \R_{H_{1}}, \R_{H_{2}}\}$  generally can not find $\Pro_{H_{1} \cap H_{2}}x$ in finite steps. In particular, if  $x \in (H_{1} \cap H_{2}) \cup ( H^{c}_{1} \cap H^{c}_{2})$, then  the CRM converges in one step; otherwise, the classical method of alternating projections becomes a subsequence of  the CRM.

\begin{theorem} \label{them:Tn:CCSP:R:S1}
Assume	 that $H_{1} \cap H_{2} \neq \varnothing$.
	Set $\mathcal{S} :=\{\Id, \R_{H_{1}}, \R_{H_{2}}\}$. Let $x \in \mathcal{H}$. Then exactly one of the following cases occurs:
	\begin{enumerate}
		\item \label{item:them:Tn:CCSP:R:S1:eq} $x \in (H_{1} \cap H_{2}) \cup ( H^{c}_{1} \cap H^{c}_{2})$. Then $\CC{\mathcal{S}}x = \Pro_{H_{1} \cap H_{2}}x$.
		\item   \label{item:them:Tn:CCSP:R:S1:PH1H2} $x \in H_{1} \cap H^{c}_{2} $. Then $(\forall k \in \mathbb{N})$   $\CC{\mathcal{S}}^{2k}x =(\Pro_{H_{1}} \Pro_{H_{2}})^{k} x$ and $\CC{\mathcal{S}}^{2k+1}x =\Pro_{H_{2}}(\Pro_{H_{1}} \Pro_{H_{2}})^{k} x$.
		\item  \label{item:them:Tn:CCSP:R:S1:PH2H1} $x \in  H^{c}_{1} \cap H_{2} $. Then $(\forall k \in \mathbb{N})$   $\CC{\mathcal{S}}^{2k}x =(\Pro_{H_{2}} \Pro_{H_{1}})^{k} x$ and $\CC{\mathcal{S}}^{2k+1}x =\Pro_{H_{1}} (\Pro_{H_{2}} \Pro_{H_{1}})^{k} x$. 
	\end{enumerate}
\end{theorem}	

\begin{proof}
	\cref{item:them:Tn:CCSP:R:S1:eq}: 
	If $x \in H_{1} \cap H_{2} $, then $\R_{H_{1}}x=x$, $\R_{H_{2}}x=x$,  $\Pro_{H_{1} \cap H_{2}}x=x$, and 
	$\mathcal{S}(x) =\{ x, \R_{H_{1}}x, \R_{H_{2}}x\} =\{x\}$. Hence, by \cref{thm:SymForm}\cref{thm:SymForm:1}, $\CC{\mathcal{S}} x =\CCO{(\mathcal{S}(x) ) } =\CCO{(\{x\})}=x = \Pro_{H_{1} \cap H_{2}}x$. 
	
	If $x \in H^{c}_{1} \cap H^{c}_{2}$, that is, $x \notin H_{1} \cup H_{2}$, then by  \cref{cor:Tn:CCSP:R}\cref{item:cor:Tn:CCSP:R:S1}, $\CC{\mathcal{S}} x =\Pro_{H_{1} \cap H_{2}}x$.
	
	\cref{item:them:Tn:CCSP:R:S1:PH1H2}: Assume that $x \in H_{1} \cap H^{c}_{2} $. We prove 	\cref{item:them:Tn:CCSP:R:S1:PH1H2} by induction on $k$.

	Clearly $x \in H_{1} \cap H^{c}_{2} $ yields  $\mathcal{S}(x) =\{ x, \R_{H_{1}}x, \R_{H_{2}}x\} =\{x, \R_{H_{2}}x\}$. By  \cref{thm:SymForm}\cref{thm:SymForm:2},   $\CC{\mathcal{S}} x =\CCO{(\mathcal{S}(x) ) }  =\frac{x+ \R_{H_{2}}x}{2} =\Pro_{H_{2}}x$.
	Hence, \cref{item:them:Tn:CCSP:R:S1:PH1H2} holds for $k=0$.  Assume that \cref{item:them:Tn:CCSP:R:S1:PH1H2} is true for some  $k \in \mathbb{N}$.
	
	Denote by $(\forall t \in \mathbb{N})$ $ x^{(t)} := \CC{\mathcal{S}}^{t}x$.
	By induction hypothesis, $x^{(2k+1)} = \CC{\mathcal{S}}^{2k+1}x =\Pro_{H_{2}}(\Pro_{H_{1}} \Pro_{H_{2}})^{k} x \in H_{2}$, so $\R_{H_{2}}  (x^{(2k+1)} ) = x^{(2k+1)} $ and $\mathcal{S} (x^{(2k+1)} )  =\{ x^{(2k+1)}, \R_{H_{1}} (x^{(2k+1)} )\}$.
	Then by \cref{thm:SymForm}\cref{thm:SymForm:2} and by induction hypothesis again,
	\begin{align} \label{eq:them:Tn:CCSP:R:S1}
	\CC{\mathcal{S}}^{2(k+1)} x =\CCO{(\mathcal{S}( x^{(2k+1)} )   ) } =\frac{x^{(2k+1)} +\R_{H_{1}} (x^{(2k+1)} )}{2}  =\Pro_{H_{1}} (x^{(2k+1)} ) =(\Pro_{H_{1}} \Pro_{H_{2}})^{k+1} x \in H_{1},
	\end{align}
	which implies that   $\R_{H_{1}} (x^{(2(k+1))} )= x^{(2(k+1))}  $  and $	\mathcal{S} (x^{(2(k+1))} ) =\{ x^{(2(k+1))}, \R_{H_{2}} (x^{(2(k+1))} )\}$.
	Furthermore, using \cref{thm:SymForm}\cref{thm:SymForm:2}, we obtain that
	\begin{align*}
	\CC{\mathcal{S}}^{2(k+1) +1} x=\CCO{(\mathcal{S}( x^{(2(k+1))} )   ) } 
	=\frac{x^{(2(k+1))} +\R_{H_{2}} (x^{(2(k+1))} )}{2}  =\Pro_{H_{2}} (x^{(2(k+1))} )\stackrel{\cref{eq:them:Tn:CCSP:R:S1}}{=} \Pro_{H_{2}}(\Pro_{H_{1}} \Pro_{H_{2}})^{k+1} x.
	\end{align*}
	Hence, the desired results hold for $k+1$ and so \cref{item:them:Tn:CCSP:R:S1:PH1H2} is true by induction.
	
	\cref{item:them:Tn:CCSP:R:S1:PH2H1}: Switch $H_{1}$ and $H_{2}$ in 	\cref{item:them:Tn:CCSP:R:S1:PH1H2} to obtain the required results. 
\end{proof}

In Example 1 of the first arXiv version of \cite{BCS2019}, the authors showed the idea that the CRM induced by $\mathcal{S} :=\{\Id, \R_{H_{1}}, \R_{H_{2}}\R_{H_{1}} \}$ needs at most three steps to find a point in $H_{1 } \cap H_{2}$  in $\mathbb{R}^{n}$.  This motivates the following \cref{them:Tn:CCSP:R:S2}.
The  main idea of the following proof comes from that example as well. Note that the proof of that Example didn't consider the Case~3 below, and that taking advantage of our  \cite[Proposition~4.4]{BOyW2019Isometry}, we actually discover that the CRM  solves the best approximation problem in at most $3$ steps.

\begin{theorem} \label{them:Tn:CCSP:R:S2} 
	Assume that $H_{1} \cap H_{2} \neq \varnothing$.
	Set $\mathcal{S} :=\{ \Id, \R_{H_{1}}, \R_{H_{2}}\R_{H_{1}} \}$. Then for every $x \in \mathcal{H}$, there exists $k \in \{1,2,3\}$ such that $\CC{\mathcal{S}}^{k}x = \Pro_{H_{1} \cap H_{2}}x$, that is, the circumcentered reflection method induced by $\mathcal{S}$ needs at most $3$ steps to find the best approximation point $\Pro_{H_{1} \cap H_{2}}x$.
\end{theorem}	

\begin{proof}
	Applying  \cite[Proposition~4.4]{BOyW2019Isometry} and \cite[Proposition~4.2(iii)]{BOyW2019Isometry}  with substituting $m=3$, $T_{1}=\Id$, $T_{2} =\R_{H_{1}}$,  $T_{3} =\R_{H_{2}}\R_{H_{1}}$ and $W =H_{1} \cap H_{2} $, we   obtain respectively  that
	\begin{subequations}
			\begin{align} 
	&	(\forall x \in \mathcal{H})  \quad \CC{\mathcal{S}} x = \Pro_{H_{1} \cap H_{2}}x \Leftrightarrow \CC{\mathcal{S}} x \in H_{1} \cap H_{2}; \label{eq:prop:Tn:CCSP:R:S2:equivalence}\\
	&	(\forall x \in \mathcal{H})  (\forall k \in \mathbb{N}) \quad 	 	\Pro_{H_{1} \cap H_{2}} \CC{\mathcal{S}}^{k} x =\Pro_{H_{1} \cap H_{2}} x.\label{eq:prop:Tn:CCSP:R:S2:PH1H2CCSk}
		\end{align}
	\end{subequations} 
	Let $x \in \mathcal{H}$. According to \cref{defn:cir:map},
	\begin{align} \label{eq:them:Tn:CCSP:R:S2:CCS} 
	\CC{\mathcal{S}} x =\CCO{ (\mathcal{S}(x) )}= \CCO{ (\{ x, \R_{H_{1}}x, \R_{H_{2}}\R_{H_{1}} x \}) }.
	\end{align}
	 We have exactly the following four cases.
	
	\textbf{Case~1:} $\R_{H_{1}}x -x =0 $ and $\R_{H_{2}}\R_{H_{1}}x- \R_{H_{1}}x=0$.  Then, in view of  \cref{eq:them:Tn:CCSP:R:S2:CCS}  and \cref{thm:SymForm}\cref{thm:SymForm:1}, $\CC{\mathcal{S}} x =\CCO{(\{x\})}=x = \Pro_{H_{1} \cap H_{2}}x$.

	\textbf{Case~2:} $\R_{H_{1}}x -x \neq 0 $ and $\R_{H_{2}}\R_{H_{1}}x- \R_{H_{1}}x \neq 0$, that is, $x \notin \Fix \R_{H_{1}} =H_{1}$ and $\R_{H_{1}}x \notin \Fix \R_{H_{2}} =H_{2}$. Then using \cref{cor:Tn:CCSP:R}\cref{item:cor:Tn:CCSP:R:S2}, we obtain that $\CC{\mathcal{S}} x = \Pro_{H_{1} \cap H_{2}}x$.
	
	\textbf{Case~3:} $\R_{H_{1}}x -x =0 $ and $\R_{H_{2}}\R_{H_{1}}x- \R_{H_{1}}x \neq 0$. Then $	\mathcal{S} (x)   =\{ \R_{H_{1}}x, \R_{H_{2}}\R_{H_{1}}x\}$,
	which, combining with \cref{eq:them:Tn:CCSP:R:S2:CCS}  and \cref{thm:SymForm}\cref{thm:SymForm:2}, implies that 
	\begin{align}  \label{eq:prop:Tn:CCSP:R:S2:Case3}
	x^{(1)} :=\CC{\mathcal{S}} x = \frac{\R_{H_{1}}x + \R_{H_{2}}\R_{H_{1}}x }{2} =  \Pro_{H_{2}}\R_{H_{1}}x \in H_{2}.
	\end{align}
	
	Case~3.1: Assume $x^{(1)} \in H_{1} $. Then by \cref{eq:prop:Tn:CCSP:R:S2:Case3},  $\CC{\mathcal{S}} x \in H_{1} \cap H_{2}$. Hence, by \cref{eq:prop:Tn:CCSP:R:S2:equivalence},  $\CC{\mathcal{S}} x = \Pro_{H_{1} \cap H_{2} }x$.
	
	Case~3.2: Assume $x^{(1)} \notin H_{1}$. If $\R_{H_{1}} (x^{(1)} ) \notin H_{2}$, then 
 apply \cref{cor:Tn:CCSP:R}\cref{item:cor:Tn:CCSP:R:S2} with $x = x^{(1)} $ to obtain that $\CC{\mathcal{S}}^{2} x =\CC{\mathcal{S}} (x^{(1)} )= \Pro_{H_{1} \cap H_{2}} (x^{(1)})= \Pro_{H_{1} \cap H_{2}} \CC{\mathcal{S}} x \stackrel{\cref{eq:prop:Tn:CCSP:R:S2:PH1H2CCSk}}{=} \Pro_{H_{1} \cap H_{2}} x$.
	
 Assume that $\R_{H_{1}} (x^{(1)} ) \in H_{2}$.
	Then $\mathcal{S} (x^{(1)}  )   =\{ x^{(1)}, \R_{H_{1}} (x^{(1)} )  \}$, and, by
  \cref{defn:cir:map} and  \cref{thm:SymForm}\cref{thm:SymForm:2},  
	\begin{align}  \label{eq:prop:Tn:CCSP:R:S2:Case3.2.2:H1}
	x^{(2)} :=\CC{\mathcal{S}}  (x^{(1)}  ) = \frac{  x^{(1)}    + \R_{H_{1}} (x^{(1)} )   }{2} =  \Pro_{H_{1}} (x^{(1)}  )  \in H_{1}.
	\end{align}
	Moreover, 
	because $H_{2}$ is an affine subspace,  $ x^{(1)}  \in H_{2}$ and $\R_{H_{1}} (x^{(1)} ) \in H_{2} $, we know that $	x^{(2)} =\CC{\mathcal{S}}  (x^{(1)}  ) = \frac{  x^{(1)}    + \R_{H_{1}} (x^{(1)} )   }{2} \in H_{2}$,
which, combining with \cref{eq:prop:Tn:CCSP:R:S2:Case3.2.2:H1}, yields  that $ \CC{\mathcal{S}} (x^{(1)} ) =\CC{\mathcal{S}}^{2} x = 	x^{(2)} \in H_{1} \cap H_{2} $. Hence,
	$ \CC{\mathcal{S}}^{2} x = \CC{\mathcal{S}} (x^{(1)} ) \stackrel{ \cref{eq:prop:Tn:CCSP:R:S2:equivalence}}{=} \Pro_{H_{1} \cap H_{2}} (x^{(1)})= \Pro_{H_{1} \cap H_{2}} \CC{\mathcal{S}} x \stackrel{\cref{eq:prop:Tn:CCSP:R:S2:PH1H2CCSk}}{=} \Pro_{H_{1} \cap H_{2}} x$.
	
	\textbf{Case~4:} $\R_{H_{1}}x -x \neq 0 $ and $\R_{H_{2}}\R_{H_{1}}x- \R_{H_{1}}x=0$.  Then  $\mathcal{S} (x)   =\{ x,  \R_{H_{1}}x\}$,
	and, by \cref{thm:SymForm}\cref{thm:SymForm:2},   $x^{(1)} :=\CC{\mathcal{S}} x = \frac{x + \R_{H_{1}}x }{2} =  \Pro_{H_{1}} x \in H_{1}$.
Then 
 $\R_{H_{1}} (	x^{(1)}   )  = 	x^{(1)} $ and 
$\mathcal{S}   (	x^{(1)} )  
	=  \{  \R_{H_{1}} (	x^{(1)}   ) ,   \R_{H_{2}}\R_{H_{1}} (	x^{(1)}   )  \}.$

	If $ \R_{H_{1}} (	x^{(1)}   ) = \R_{H_{2}}\R_{H_{1}} (	x^{(1)}   ) $, then applying case 1 with $x=x^{(1)}$, we derive that $\CC{\mathcal{S}}^{2} x =\CC{\mathcal{S}} (x^{(1)} ) = \Pro_{H_{1} \cap H_{2}} (x^{(1)} )= \Pro_{H_{1} \cap H_{2}} (\CC{\mathcal{S}}  x) \stackrel{\cref{eq:prop:Tn:CCSP:R:S2:PH1H2CCSk}}{=} \Pro_{H_{1} \cap H_{2}} x$;
	If $ \R_{H_{1}} (	x^{(1)}   ) \neq  \R_{H_{2}}\R_{H_{1}} (	x^{(1)}   ) $, then employing case 3 with
	$x=x^{(1)}$, we obtain that either $\CC{\mathcal{S}}^{2} x =\CC{\mathcal{S}} (x^{(1)} ) = \Pro_{H_{1} \cap H_{2}} (x^{(1)} ) \stackrel{\cref{eq:prop:Tn:CCSP:R:S2:PH1H2CCSk}}{=} \Pro_{H_{1} \cap H_{2}} x$ or $\CC{\mathcal{S}}^{3} x =\CC{\mathcal{S}}^{2} (x^{(1)} ) = \Pro_{H_{1} \cap H_{2}} (x^{(1)})  \stackrel{\cref{eq:prop:Tn:CCSP:R:S2:PH1H2CCSk}}{=} \Pro_{H_{1} \cap H_{2}} x$.
	
	Altogether, the proof is complete.
\end{proof}

%%%%%%%%%%%%%%%%%%%%%%%%%%%%%%%%%%%%%%%%%%%%%%%%%%%%%%%%%%%%%%%%%%%%%%%%%%%%%%%%%%%\section{Finite convergence of CRM on halfspaces and hyperplanes}%%%%%%%%%%%%%%%%%%%
%%%%%%%%%%%%%%%%%%%%%%%%%%%%%%%%%%%%%%%%%%%%%%%%%%%%%%%%%%%%%%%%%%%%

\subsection*{CRMs associated with halfspaces and hyperplanes } \label{sec:FiniteConverHH}
\begin{proposition} \label{prop:CCS1:feasible:half-space}
	Suppose that $u_{1} \neq 0$ and $u_{2}\neq 0$, 
	and  that $\innp{u_{1} , u_{2}} \geq 0$. Set $\mathcal{S} :=\{ \Id, \R_{W_{1}}, \R_{W_{2}} \}$. 
	Then $(\forall x \in W_{1} \cup W_{2})$, $\CC{\mathcal{S}} x =\Pro_{W_{1} \cap W_{2}}x$.	
\end{proposition}

\begin{proof}
	Let $x \in W_{1} \cup W_{2}$. Then we have exactly the following three cases:
	
	\textbf{Case~1:} $x \in W_{1} \cap W_{2}$. Then  by \cref{thm:SymForm}\cref{thm:SymForm:1}, clearly $ \CC{\mathcal{S}} x =\CCO{ (\mathcal{S} (x) ) }=x=\Pro_{W_{1} \cap W_{2}}x$.

	\textbf{Case~2:} $x \in W_{1} \cap W^{c}_{2}$.   Then $\R_{W_{1}}x =x$ and, by \cref{fact:FactWFactH}, $\Pro_{W_{2}}x =\Pro_{H_{2}}x$ and $\R_{W_{2}}x = \R_{H_{2}}x $. Hence, $	\mathcal{S} (x)=\{ x, \R_{W_{1}}x, \R_{W_{2}}x \}= \{ x,  \R_{H_{2}}x \}$, and
 by \cref{thm:SymForm}\cref{thm:SymForm:2},  $	\CC{\mathcal{S}} x=  \CCO{(\mathcal{S}( x )   ) } = \CCO{(\{ x,  \R_{W_{2}}x \}  )}=\frac{x +  \R_{W_{2}}x}{2} =\Pro_{W_{2}} x$.
	Using definitions of $W_{1}, W_{2}$ in \cref{eq:WH} and   $x \in W_{1} \cap W^{c}_{2}$, we know that $\innp{x, u_{1}} \leq \eta_{1}$ and $ \innp{x, u_{2}} > \eta_{2}$.
	Combine this with the assumption $\innp{u_{1}, u_{2}} \geq 0$ to obtain that
	\begin{align} \label{eq:prop:CCS1:feasible:half-space:leq:eta1}
	\innp{x, u_{1}} + \frac{\eta_{2} - \innp{x,u_{2}} }{\norm{u_{2}}^{2}} \innp{u_{2}, u_{1}} \leq \eta_{1}.
	\end{align}
In view of \cref{fact:Projec:Hyperplane},
	\begin{align} \label{eq:prop:CCS1:feasible:half-space:IN:H1}
	\innp{ \Pro_{H_{2}} x, u_{1} } = \Innp{x + \frac{\eta_{2} - \innp{x,u_{2}} }{\norm{u_{2}}^{2}}u_{2}, u_{1}} = \innp{x, u_{1}} + \frac{\eta_{2} - \innp{x,u_{2}} }{\norm{u_{2}}^{2}} \innp{u_{2}, u_{1}}  \stackrel{\cref{eq:prop:CCS1:feasible:half-space:leq:eta1}}{\leq} \eta_{1},
	\end{align}
	which implies that $\Pro_{H_{2}} x \in W_{1} $. Hence, $\Pro_{W_{2}}x =\Pro_{H_{2}} x \in W_{1} \cap H_{2}  \subseteq W_{1} \cap W_{2}$.
	Apply \cref{fact:AsubseteqB:Projection} with $A=W_{1} \cap W_{2}$ and $B=W_{2}$ to obtain that $\CC{\mathcal{S}} x =\Pro_{W_{2}}x =\Pro_{W_{1} \cap W_{2}}x $.
	
	\textbf{Case~3:} $x \in W^{c}_{1} \cap W_{2}$. Swap $W_{1}$ and $W_{2}$ in Case~2 above to obtain the required result.
\end{proof}

\subsubsection*{$u_{1}$ and $u_{2}$ are linearly dependent}
\begin{lemma} \label{lem:CCSu10u20}
	Assume that $u_{1} =0$ or $u_{2} =0$. The following statements hold:
	\begin{enumerate}
		\item \label{lem:CCS:R:u1=0ORu2=0}	Assume that $H_{1} \cap H_{2} \neq \varnothing$. Set $\mathcal{S}_{1} :=\{\Id, \R_{H_{1}}, \R_{H_{2}}\}$ and $\mathcal{S}_{2} :=\{\Id, \R_{H_{1}}, \R_{H_{2}}\R_{H_{1}} \}$. Then
		\begin{align*}
		(\forall x \in \mathcal{H}) \quad \CC{\mathcal{S}_{1}}x= \CC{\mathcal{S}_{2}}x= \Pro_{H_{1} \cap H_{2}}x.
		\end{align*}
		\item \label{lem:CCS:u1eq0:or:u2eq0}
		Assume that $W_{1} \cap W_{2} \neq \varnothing$. Set $\mathcal{S}_{1} :=\{\Id, \R_{W_{1}}, \R_{W_{2}} \}$ and $\mathcal{S}_{2} :=\{\Id, \R_{W_{1}}, \R_{W_{2}}\R_{W_{1}} \}$. Then
		\begin{align*}
		(\forall x \in \mathcal{H}) \quad \CC{\mathcal{S}_{1}}x = \CC{\mathcal{S}_{2}}x = \Pro_{W_{1} \cap W_{2}}x.
		\end{align*}
		\item \label{lem:CCSu10u20:HW} Assume that $H_{1} \cap W_{2} \neq \varnothing$.  Set $\mathcal{S}_{1} :=\{\Id, \R_{H_{1}}, \R_{W_{2}} \}$, $\mathcal{S}_{2} :=\{\Id, \R_{H_{1}}, \R_{W_{2}}\R_{H_{1}} \}$ and $\mathcal{S}_{3} :=\{\Id, \R_{W_{2}}, \R_{H_{1}} \R_{W_{2}}\}$.  Then  
		\begin{align*}
		(\forall x \in \mathcal{H}) \quad \CC{\mathcal{S}_{1}}x=\CC{\mathcal{S}_{2}}x=\CC{\mathcal{S}_{3}}x=\Pro_{H_{1} \cap W_{2}}x.
		\end{align*}
	\end{enumerate} 
	
\end{lemma}

\begin{proof}
	\cref{lem:CCS:R:u1=0ORu2=0}: Note that  $H_{1} \cap H_{2} \neq \varnothing$ implies that for every $i \in \{1,2\}$, if $u_{i}=0$, then $H_{i} =\mathcal{H}$ and $\R_{H_{i}} =\Id$. Without loss of generality, assume $u_{1}=0$. Then  $H_{1} =\mathcal{H}$ and, by \cref{thm:SymForm}\cref{thm:SymForm:2}, $(\forall x \in \mathcal{H})$ $\CC{\mathcal{S}_{1}}x= \CC{\mathcal{S}_{2}}x= \CCO{ (\{  x, \R_{H_{2}}x \} )} = \frac{x+\R_{H_{2}}x}{2 }=\Pro_{H_{2}}x= \Pro_{H_{1} \cap H_{2}}x$.
	Hence, \cref{lem:CCS:R:u1=0ORu2=0} holds.
	
	The proofs for	\cref{lem:CCS:u1eq0:or:u2eq0} and \cref{lem:CCSu10u20:HW} are similar to the proof of 	\cref{lem:CCS:R:u1=0ORu2=0} and are omitted. 
\end{proof}

The following results are necessary to proofs of \Cref{prop:u1u2LD:S1:neq0,prop:u1u2LD:S2:<0}.  The long and mechanical proof of \cref{lem:u1u2LD} is left in the  Appendix.
\begin{lemma}\label{lem:u1u2LD}
	Assume  that $u_{1} \neq 0$ and $u_{2} \neq 0$, and that  $u_{1}$ and $u_{2}$ are linearly dependent.
	The following statements hold.
	\begin{enumerate}
		\item \label{item:lem:u1u2LD:u12neq0:innpneq0} $\innp{u_{1}, u_{2}} \neq 0$. Moreover, if  $\innp{u_{1}, u_{2}} > 0$, then   $u_{2} =  \frac{\norm{u_{2}}}{\norm{u_{1}}} u_{1}$ and $\innp{u_{1}, u_{2}} =  \norm{u_{1}}\norm{u_{2}}$; if $\innp{u_{1}, u_{2}} < 0$, then $u_{2} = -\frac{\norm{u_{2}}}{\norm{u_{1}}} u_{1}$ and $\innp{u_{1}, u_{2}} =- \norm{u_{1}}\norm{u_{2}}$.
		\item \label{item:lem:u1u2LD:>0} Assume that $\innp{u_{1}, u_{2}} > 0$. Then $\frac{\eta_{1}}{\norm{u_{1}}} =  \frac{\eta_{2}}{\norm{u_{2}}} \Leftrightarrow W_{1} =W_{2} \Leftrightarrow H_{1} =H_{2}$. Moreover, we have that $\frac{\eta_{1}}{\norm{u_{1}}} \neq \frac{\eta_{2}}{\norm{u_{2}}}  \Leftrightarrow  H_{1} \cap H_{2} =\varnothing$,   that 
		$\frac{\eta_{1}}{\norm{u_{1}}} < \frac{\eta_{2}}{\norm{u_{2}}} \Leftrightarrow W_{1} \subsetneqq W_{2}$, and that $\frac{\eta_{1}}{\norm{u_{1}}} > \frac{\eta_{2}}{\norm{u_{2}}} \Leftrightarrow W_{2} \subsetneqq W_{1}$.
			\item \label{item:lem:u1u2LD:<0} Assume that $\innp{u_{1}, u_{2}} < 0$. Then $\frac{\eta_{1}}{\norm{u_{1}}} =  -\frac{\eta_{2}}{\norm{u_{2}}} \Leftrightarrow  H_{1} =H_{2}$. Furthermore  $\frac{\eta_{1}}{\norm{u_{1}}} \neq -\frac{\eta_{2}}{\norm{u_{2}}}  \Leftrightarrow  H_{1} \cap H_{2} =\varnothing$. Moreover,  $W_{1} \cap W_{2} \neq \varnothing \Leftrightarrow  \eta_{1}\norm{u_{2}} +\eta_{2} \norm{u_{1}} \geq 0$.  In addition, if $W_{1} \cap W_{2} \neq \varnothing$, then	$\inte W_{1}\cup W_{2} = W_{1} \cup \inte W_{2}=\mathcal{H}$ and $H_{1} \subseteq W_{2}$ and $H_{2} \subseteq W_{1}$.

		\item  \label{item:lem:u1u2LD:u12neq0:innp>0:a} Assume that $\innp{u_{1}, u_{2}} > 0$ and   $\frac{\eta_{1}}{\norm{u_{1}}} \geq  \frac{\eta_{2}}{\norm{u_{2}}}$. Then $(\forall  x \in \mathcal{H} )$ $\norm{\Pro_{H_{1}}x - \Pro_{H_{2}}x} = \frac{\eta_{1}}{\norm{u_{1}}} -\frac{\eta_{2}}{\norm{u_{2}}}  $. Moreover,   $(x \in \mathcal{H} \smallsetminus W_{1})$   $\R_{W_{1} }x\in W_{2} \Leftrightarrow \norm{x - \Pro_{W_{1}}x} \geq \norm{\Pro_{H_{1}}x-\Pro_{H_{2}}x } $.

	\item 	\label{item:lem:u1u2LD} Assume that $\innp{u_{1}, u_{2}} < 0$ and   $W_{1} \cap W_{2} \neq \varnothing$.  Then  $(\forall  x \in \mathcal{H} )$ $\norm{\Pro_{H_{1}}x - \Pro_{H_{2}}x} = \frac{\eta_{1}}{\norm{u_{1}}} +\frac{\eta_{2}}{\norm{u_{2}}}  $. Moreover,   $( \forall x \in \mathcal{H} \smallsetminus W_{1})$  $\R_{W_{1}}x \notin W_{2} \Leftrightarrow  \norm{x - \Pro_{W_{1}}x} > \norm{\Pro_{H_{1}}x-\Pro_{H_{2}}x } $.
		
			\item  \label{item:lem:u1u2LD:u12neq0:innp<0:W12} Assume that  $\innp{u_{1}, u_{2}} < 0$, $W_{1} \cap W_{2} \neq \varnothing$ and  $x \in W_{1} \cap W^{c}_{2}$. Then $\Pro_{W_{2}}x  =\Pro_{W_{1} \cap H_{2}  }x =\Pro_{W_{1} \cap W_{2}} x$.
			\item  \label{item:lem:u1u2LD:u12neq0:innp<0:d} Assume that  $\innp{u_{1}, u_{2}} < 0$ and  $W_{1} \cap W_{2} \neq \varnothing$. Let $x \in W_{2}$ with $\norm{x - \Pro_{H_{1}}x} > \norm{\Pro_{H_{1}}x-\Pro_{H_{2}}x }$.  Then 
			\begin{enumerate}
				\item \label{lemma:u1u2Ld:Less0:xinW2:xPh1Larger:xnotinW1} $x \notin W_{1}$.
				\item \label{lemma:u1u2Ld:Less0:xinW2:xPh1Larger:RW1notinW2} $\R_{W_{1}}x = \R_{H_{1}}x \notin W_{2}$.
				\item \label{lemma:u1u2Ld:Less0:xinW2:xPh1Larger:NormLarger} $\norm{\R_{W_{2}}\R_{W_{1}}x - \R_{W_{1}}x } =2 ( \norm{x - \Pro_{H_{1}}x} - \norm{ \Pro_{H_{1}}x- \Pro_{H_{2}}x}) 
				\leq 2\norm{x - \Pro_{H_{1}}x}  = \norm{ x -\R_{W_{1}}x }$.
			\end{enumerate}

		\item  \label{item:lem:u1u2LD:LD}    Let $x \in \mathcal{H}$.   Then  $x, \R_{H_{1}}x, \R_{H_{2}}x$ are affinely  dependent;  $x, \R_{H_{1}}x, \R_{H_{2}}\R_{H_{1}}x$ are affinely  dependent; $x, \R_{W_{1}}x, \R_{W_{2}}x$ are affinely  dependent; and $x, \R_{W_{1}}x, \R_{W_{2}}\R_{W_{1}}x$ are affinely  dependent.
	\end{enumerate}
\end{lemma}

\begin{proposition} \label{prop:u1u2LD:S1:neq0}
	Assume that $u_{1} \neq 0$ and $u_{2} \neq 0$ and that $u_{1}$ and $u_{2}$ are linearly dependent. Assume that $W_{1} \cap W_{2} \neq \varnothing$.  Set $\mathcal{S} :=\{ \Id, \R_{W_{1}}, \R_{W_{2}} \}$.  
	Then  exactly one of the following cases occurs:	
	
	  \emph{Case~1}: $\innp{u_{1}, u_{2}} > 0$. Then if $\frac{\eta_{1}}{\norm{u_{1}}} =  \frac{\eta_{2}}{\norm{u_{2}}}$, then $(\forall x \in \mathcal{H})$ $\CC{\mathcal{S}}x =\Pro_{W_{1} \cap W_{2}}x$; otherwise, $(\forall x \in W_{1} \cup W_{2})$ $ \CC{\mathcal{S}}x =\Pro_{W_{1} \cap W_{2}}x$, and
		$(\forall x \in \mathcal{H} \smallsetminus (W_{1} \cup W_{2}))$   $ \CC{\mathcal{S}}x = \varnothing$.

	  \emph{Case~2}:   $\innp{u_{1}, u_{2}} < 0$.  Then $W_{1} \cup W_{2} =\mathcal{H} $ and $(\forall x \in \mathcal{H})$ $\CC{\mathcal{S}}x =\Pro_{W_{1} \cap W_{2}}x$.
 
Consequently,  $(\forall x \in W_{1} \cup W_{2})$ $\CC{\mathcal{S}}x =\Pro_{W_{1} \cap W_{2}}x$.
\end{proposition}

\begin{proof}
By assumptions and   \cref{lem:u1u2LD}\cref{item:lem:u1u2LD:u12neq0:innpneq0}, we have either $\innp{u_{1}, u_{2}} > 0$ or $\innp{u_{1}, u_{2}} < 0$.
	
 \emph{Case~1}:  Assume that $\innp{u_{1}, u_{2}} > 0$.
If $\frac{\eta_{1}}{\norm{u_{1}}} =  \frac{\eta_{2}}{\norm{u_{2}}}$, then by  \cref{lem:u1u2LD}\cref{item:lem:u1u2LD:>0}, $W_{1} =W_{2}=W_{1} \cap W_{2} $, and $\mathcal{S}=\{ \Id, \R_{W_{1}}  \}$. Hence, by  \cref{thm:SymForm}\cref{thm:SymForm:2},  $	\CC{\mathcal{S}}x =\CCO{ (\{ x, \R_{W_{1}} x  \})} =\frac{ x+ \R_{W_{1}} x }{2} =\Pro_{W_{1}}x =  \Pro_{W_{1} \cap W_{2}}x$.

	Assume that $\frac{\eta_{1}}{\norm{u_{1}}} \neq  \frac{\eta_{2}}{\norm{u_{2}}}$. 
If  $x \in W_{1} \cup W_{2}$, then  \cref{prop:CCS1:feasible:half-space} yields    $\CC{\mathcal{S}}x = \Pro_{W_{1} \cap W_{2}}x$.
Suppose that $x \in \mathcal{H} \smallsetminus (W_{1} \cup W_{2})$. 
	By \cref{lem:WH:mathcalH:RWinW}, $\R_{W_{1}}x =\R_{H_{1}}x \in \inte W_{1}$ and $\R_{W_{2}}x =\R_{H_{2}}x \in \inte W_{2}$. Note that, by \cref{lem:u1u2LD}\cref{item:lem:u1u2LD:>0}, $H_{1} \cap H_{2} =\varnothing$. Then clearly $\Pro_{H_{1}}x \neq \Pro_{H_{2}}x$. Now we have 
	\begin{align*}
	x \neq \R_{W_{1}}x, x \neq \R_{W_{2}}x, ~\text{and}~ \R_{W_{1}}x =2 \Pro_{H_{1}}x -x  \neq 2 \Pro_{H_{2}}x -x =\R_{W_{2}}x,
	\end{align*}
	which, combining with \cref{lem:u1u2LD}\cref{item:lem:u1u2LD:LD} and \cref{thm:SymForm}\cref{thm:SymForm:3:b}, imply that $ \CC{\mathcal{S}}x = \varnothing$.

 \emph{Case~2}:  Assume that $\innp{u_{1}, u_{2}} < 0$.  Then  \cref{lem:u1u2LD}\cref{item:lem:u1u2LD:<0} implies that $\mathcal{H}=W_{1} \cup W_{2} $. Hence,  we have exactly the following three cases:
	
	\emph{Case~2.1}: $x \in W_{1} \cap W_{2}$. Then $\mathcal{S}(x)=\{x\}$ and by \cref{thm:SymForm}\cref{thm:SymForm:1}, $\CC{\mathcal{S}}x =x= \Pro_{W_{1} \cap W_{2}}x$.

\emph{Case~2.2}: 
	$x \in W_{1} \cap W^{c}_{2}$. Then $\mathcal{S}(x)=\{ x, \R_{W_{2}}x \}$, and by \cref{thm:SymForm}\cref{thm:SymForm:2} and \cref{lem:u1u2LD}\cref{item:lem:u1u2LD:u12neq0:innp<0:W12}, $	\CC{\mathcal{S}}x =\CCO{ (\{ x, \R_{W_{2}}x \} )}  =\Pro_{W_{2}}x = \Pro_{W_{1} \cap W_{2}}x$.

\emph{Case~2.3}:
	$x \in W_{2} \cap W^{c}_{1}$. This proof is similar to the proof of Case 2. 
	
	Altogether, the proof is complete.
\end{proof}

 \Cref{prop:u1u2LD:S1:neq0,prop:u1u2LD:S2:<0}  suggest that if $u_{1}$ and $u_{2}$ are linearly dependent, then the CRM induced by $\mathcal{S} :=\{ \Id, \R_{W_{1}}, \R_{W_{2}}\R_{W_{1}} \}$ performs worse than the CRM induced by $\mathcal{S} :=\{ \Id, \R_{W_{1}}, \R_{W_{2}} \}$.
\begin{proposition} \label{prop:u1u2LD:S2:<0}
	Assume that $u_{1} \neq 0$ and $u_{2} \neq 0$, and that $u_{1}$ and $u_{2}$ are linearly dependent. Set $\mathcal{S} :=\{ \Id, \R_{W_{1}}, \R_{W_{2}}\R_{W_{1}} \}$.  	Assume that $W_{1} \cap W_{2} \neq \varnothing$.   Then exactly one of the following cases occurs:

			 \emph{Case~1}:   $\innp{u_{1}, u_{2}} > 0$ and $\frac{\eta_{1}}{\norm{u_{1}}} \leq \frac{\eta_{2}}{\norm{u_{2}}} $. Then   $(\forall x \in \mathcal{H})$ $ \CC{\mathcal{S}}x=\Pro_{W_{1}}x =\Pro_{W_{1}\cap W_{2}}x $. 

			 \emph{Case~2}:    $\innp{u_{1}, u_{2}} > 0$ and $\frac{\eta_{1}}{\norm{u_{1}}} >\frac{\eta_{2}}{\norm{u_{2}}} $.   Then if $x \in W_{1}$, then $ \CC{\mathcal{S}}x=\Pro_{W_{2}}x =\Pro_{W_{1}\cap W_{2}}x $; if  $x \notin W_{1}$ and $\norm{x - \Pro_{W_{1}}x} \geq \norm{\Pro_{H_{1}}x-\Pro_{H_{2}}x }$, then  $ \CC{\mathcal{S}}x=\Pro_{W_{1}}x=\Pro_{H_{1}}x \notin  W_{1}\cap W_{2}$; and if $x \notin W_{1}$ and $\norm{x - \Pro_{W_{1}}x} < \norm{\Pro_{H_{1}}x-\Pro_{H_{2}}x }$, then $ \CC{\mathcal{S}}x =\varnothing$.

 \emph{Case~3}:   $\innp{u_{1}, u_{2}} < 0$.  Then if $x \in W_{1}  $, then $\CC{\mathcal{S}} x  =\Pro_{W_{1}\cap W_{2}}x$; if $x \notin W_{1}  $ and  $\norm{x - \Pro_{W_{1}}x} \leq  \norm{\Pro_{H_{1}}x-\Pro_{H_{2}}x }$, then $\CC{\mathcal{S}} x =\Pro_{W_{1}}x=\Pro_{W_{1}\cap W_{2}}x$; 
	 if $x \notin W_{1}  $,  $\norm{x - \Pro_{W_{1}}x} > \norm{\Pro_{H_{1}}x-\Pro_{H_{2}}x }$, and $\frac{\eta_{1}}{\norm{u_{1}}} = -\frac{\eta_{2}}{\norm{u_{2}}}$, then   $\CC{\mathcal{S}} x = \Pro_{W_{1}\cap W_{2}}x$; if $x \notin W_{1}  $,  $\norm{x - \Pro_{W_{1}}x} > \norm{\Pro_{H_{1}}x-\Pro_{H_{2}}x }$, and $\frac{\eta_{1}}{\norm{u_{1}}} \neq -\frac{\eta_{2}}{\norm{u_{2}}}$, then $\CC{\mathcal{S}} x =\varnothing$.

Consequently, $(\forall x \in W_{1})$  $ \CC{\mathcal{S}} x =\Pro_{W_{1}\cap W_{2}}x $. 
\end{proposition}

\begin{proof} 
 \emph{Case~1}:    Assume that $\innp{u_{1}, u_{2}} > 0$ and $\frac{\eta_{1}}{\norm{u_{1}}} \leq  \frac{\eta_{2}}{\norm{u_{2}}}$. Then by \cref{lem:u1u2LD}\cref{item:lem:u1u2LD:>0} and \cref{lem:WH:mathcalH:RWinW}, $W_{1} \subseteq W_{2}$ and $(\forall x\in \mathcal{H})$ $  \R_{W_{2}}\R_{W_{1}}x= \R_{W_{1}}x$.  Hence, by  \cref{thm:SymForm}\cref{thm:SymForm:2}, $(\forall x \in \mathcal{H})$ $\CC{\mathcal{S}}x =\frac{ x+ \R_{W_{1}} x }{2} =\Pro_{W_{1}}x =  \Pro_{W_{1} \cap W_{2}}x$.

 \emph{Case~2}:   Assume that $\innp{u_{1}, u_{2}} > 0$ and $\frac{\eta_{1}}{\norm{u_{1}}} >  \frac{\eta_{2}}{\norm{u_{2}}}$. Then by \cref{lem:u1u2LD}\cref{item:lem:u1u2LD:>0}, $W_{2} \subsetneqq W_{1}$ and so $W_{2}=W_{1} \cap W_{2}$. If $x \in W_{1}$, then $x =\R_{W_{1}}x$ and $\mathcal{S}(x)=\{ x, \R_{W_{2}}  x\}$. Hence, by \cref{thm:SymForm}\cref{thm:SymForm:2}, $\CC{\mathcal{S}}x =\CCO{ (\{ x, \R_{W_{2}} x  \})} =\frac{ x+ \R_{W_{2}} x }{2} =\Pro_{W_{2}}x =  \Pro_{W_{1} \cap W_{2}}x$.

Suppose that $x \notin W_{1}$. Then $x \neq \R_{W_{1}}x$.  Using 
\cref{lem:u1u2LD}\cref{item:lem:u1u2LD:u12neq0:innp>0:a}, we know that
\begin{align} \label{eq:item:them:u1u2LD:S2:>0:>:cases}
\norm{x - \Pro_{W_{1}}x} \geq \norm{\Pro_{H_{1}}x-\Pro_{H_{2}}x }  \Leftrightarrow \R_{W_{1}}x \in W_{2}.
\end{align}
If  $\norm{x - \Pro_{W_{1}}x} \geq \norm{\Pro_{H_{1}}x-\Pro_{H_{2}}x }$, then  $\R_{W_{1}}x \in W_{2}$ and $ \R_{W_{2}}\R_{W_{1}}x =\R_{W_{1}}x $.  So $\mathcal{S}=\{ \Id, \R_{W_{1}}  \}$,  which, by \cref{thm:SymForm}\cref{thm:SymForm:2}, implies that $\CC{\mathcal{S}}x =\CCO{ (\{ x, \R_{W_{1}} x  \} )} =\frac{ x+ \R_{W_{1}} x }{2} =\Pro_{W_{1}}x  =\Pro_{H_{1}}x \notin \Pro_{W_{1} \cap W_{2}}x$,
where the last inequality is because $\frac{\eta_{1}}{\norm{u_{1}}} >  \frac{\eta_{2}}{\norm{u_{2}}}$ yields $H_{1} \cap  W_{2} =\varnothing$.
Suppose $\norm{x - \Pro_{W_{1}}x} < \norm{\Pro_{H_{1}}x-\Pro_{H_{2}}x }$. Then by \cref{eq:item:them:u1u2LD:S2:>0:>:cases} $\R_{W_{1}}x \notin W_{2}$. 
Hence, apply \cref{lem:WH:mathcalH:RWinW} with $x= \R_{W_{1}}x$ and $W=W_{2}$ to obtain that $\R_{W_{2}}\R_{W_{1}}x =\R_{H_{2}}\R_{W_{1}}x \in \inte W_{2} $. Combine this with $x \in W^{c}_{1} \subseteq W^{c}_{2} $ to know that $x \neq \R_{W_{2}}\R_{W_{1}}x$. Recall that $x \neq \R_{W_{1}}x$ and $\R_{W_{1}}x \neq \R_{W_{2}}\R_{W_{1}}x$. So, $\card \{x, \R_{W_{1}}x, \R_{W_{2}}\R_{W_{1}}x \} =3$.  Moreover, by \cref{lem:u1u2LD}\cref{item:lem:u1u2LD:LD}, $x, \R_{W_{1}}x, \R_{W_{2}}\R_{W_{1}}x $ are affinely dependent. Therefore, by \cref{thm:SymForm}\cref{thm:SymForm:3:b},   $ \CC{\mathcal{S}}x = \varnothing$.

 \emph{Case~3}:  
If  $x \in W_{1} $, then   $\mathcal{S}(x) =\{x, \R_{W_{2}}x\}$ and so 
by \cref{thm:SymForm}\cref{thm:SymForm:2} and   \cref{lem:u1u2LD}\cref{item:lem:u1u2LD:u12neq0:innp<0:W12},  $\CC{\mathcal{S}} x =\CCO{( \{  x, \R_{W_{2}}x\})}=\frac{x+\R_{W_{2}}x}{2} =\Pro_{W_{2}}x=\Pro_{W_{1}\cap W_{2}}x$.

Assume that $x \notin W_{1}$.  Then, by \cref{lem:u1u2LD}\cref{item:lem:u1u2LD}, 
\begin{align} \label{eq:them:u1u2LD:S2:eq}
\norm{x - \Pro_{W_{1}}x} > \norm{\Pro_{H_{1}}x-\Pro_{H_{2}}x } \Leftrightarrow \R_{W_{1}}x  \notin W_{2}.
\end{align}
If  $\norm{x - \Pro_{W_{1}}x} \leq  \norm{\Pro_{H_{1}}x-\Pro_{H_{2}}x }$, then by \cref{eq:them:u1u2LD:S2:eq},  $\R_{W_{1}}x \in W_{2}$ and $\mathcal{S}(x) =\{x, \R_{W_{1}}x\}$. Moreover,
by \cref{thm:SymForm}\cref{thm:SymForm:2} and by \cref{lem:u1u2LD}\cref{item:lem:u1u2LD:u12neq0:innp<0:W12} with swapping $W_{1}$ and $W_{2}$, $\CC{\mathcal{S}} x =\CCO{ (\{  x, \R_{W_{1}}x\}) }=\frac{x+\R_{W_{1}}x}{2} =\Pro_{W_{1}}x=\Pro_{W_{1}\cap W_{2}}x$.
Suppose that  $\norm{x - \Pro_{W_{1}}x} > \norm{\Pro_{H_{1}}x-\Pro_{H_{2}}x }$. 
 If $\frac{\eta_{1}}{\norm{u_{1}}} =-\frac{\eta_{2}}{\norm{u_{2}}}$, then by \cref{lem:u1u2LD}\cref{item:lem:u1u2LD:<0},  $H_{1}=H_{2}$. Note that $\Pro_{W_{1}}x =\Pro_{H_{1}}x \in H_{1} \cap H_{2} \subseteq W_{1} \cap W_{2}$. Then apply \cref{fact:AsubseteqB:Projection} with $A= W_{1} \subseteq W_{2}$ and $B= W_{1}$ to obtain that $\Pro_{W_{1}}x =\Pro_{W_{1} \cap W_{2}}x$. 
 Moreover, by \cref{fact:FactWFactH} and \cref{eq:them:u1u2LD:S2:eq}, $\R_{W_{1}}x=\R_{H_{1}}x \notin W_{2}$ and $\R_{W_{2}}\R_{W_{1}}x = \R_{H_{2}}\R_{H_{1}}x =  \R_{H_{1}}\R_{H_{1}}x =x$. Hence, $\mathcal{S}(x) =\{x, \R_{W_{1}}x\}$ and $\CC{\mathcal{S}} x =\CCO{( \{  x, \R_{W_{1}}x\})} =\Pro_{W_{1}}x =\Pro_{W_{1} \cap W_{2}}x$.
Suppose $\frac{\eta_{1}}{\norm{u_{1}}} \neq -\frac{\eta_{2}}{\norm{u_{2}}}$. Then by \cref{lem:u1u2LD}\cref{item:lem:u1u2LD:<0}, $H_{1} \cap H_{2} =\varnothing$ and so $\Pro_{H_{1}}x- \Pro_{H_{2}}x \neq 0$. Using \cref{lem:u1u2LD}\cref{lemma:u1u2Ld:Less0:xinW2:xPh1Larger:NormLarger}, we know that    $\norm{\R_{W_{2}}\R_{W_{1}}x - \R_{W_{1}}x } < \norm{ x -\R_{W_{1}}x }$ and that $x \neq \R_{W_{2}}\R_{W_{1}}x $. Combine this with $x \notin W_{1}$ and \cref{eq:them:u1u2LD:S2:eq} to see that  $\card \{ x, \R_{W_{1}}x, \R_{W_{2}}\R_{W_{1}}x \} =3$. Moreover, by \cref{lem:u1u2LD}\cref{item:lem:u1u2LD:LD}, $x, \R_{W_{1}}x, \R_{W_{2}}\R_{W_{1}}x $ are affinely dependent. Hence, via \cref{thm:SymForm}\cref{thm:SymForm:3:b}, we obtain that $ \CC{\mathcal{S}}x = \varnothing$.

Altogether, the proof is complete.
\end{proof}

According to \cref{prop:H1W2:u1u2LD:P} below, if $u_{1}$ and $u_{2}$ are linearly dependent, then we don't need any projection-based algorithms  to find  $\Pro_{H_{1} \cap W_{2}}x$.
\begin{proposition} \label{prop:H1W2:u1u2LD:P}
	Assume that $u_{1} \neq 0$ and $u_{2} \neq 0$, that $u_{1}$ and $u_{2}$ are linearly dependent, and  that $H_{1} \cap W_{2} \neq \varnothing$. Then 
$H_{1} \subseteq W_{2}$. Consequently, $(\forall x \in \mathcal{H})$ $ \Pro_{H_{1}}x=\Pro_{H_{1} \cap W_{2}}x$.
\end{proposition}

\begin{proof}
	This is clear from \cref{lem:u1u2LD}\cref{item:lem:u1u2LD:u12neq0:innpneq0} and definitions of $H_{1}$ and $W_{2}$ in \cref{eq:WH}.
\end{proof}

\subsubsection*{$u_{1}$ and $u_{2}$ are linearly independent}
Although the following \cref{theorem:CCS1:half-spaces} implies that   $\CC{\mathcal{S}}$ with $\mathcal{S} :=\{ \Id, \R_{W_{1}}, \R_{W_{2}} \}$ is proper, it is easy to find $W_{1}$ and $W_{2}$ in $\mathbb{R}^{2}$ such that there exists $x \in (W_{1} \cap W^{c}_{2}) \cup (W^{c}_{1} \cap W_{2})$ satisfying that $(\forall k \in \mathbb{N})$  $\CC{\mathcal{S}}^{k} x \notin W_{1} \cap W_{2}$.
\begin{proposition} \label{theorem:CCS1:half-spaces}
	Assume that $u_{1}$ and $u_{2}$ are linearly independent.   	Set $\mathcal{S} :=\{ \Id, \R_{W_{1}}, \R_{W_{2}} \}$.  Let $x \in \mathcal{H}$.
	Then exactly one of the following cases occurs:
 
	\emph{Case~1}: $x \in W_{1} \cup W_{2}$. Then $(\forall x \in W_{1} \cap W_{2})$ $\CC{\mathcal{S}} x =\Pro_{W_{1} \cap W_{2}}x$;
		$(\forall x \in W_{1} \cap W^{c}_{2})$ $\CC{\mathcal{S}} x =\Pro_{H_{2}}x$; and
		$(\forall x \in W^{c}_{1} \cap W_{2})$ $\CC{\mathcal{S}} x =\Pro_{H_{1} }x$.
			
	\emph{Case~2}:  $x \in \mathcal{H} \smallsetminus (W_{1} \cup W_{2})$. Then $\CC{\mathcal{S}} x =\Pro_{H_{1} \cap H_{2}}x$. 
 
\end{proposition}

\begin{proof}
\emph{Case~1}:  Suppose that $x \in W_{1} \cup W_{2}$.   
If  $ x \in W_{1} \cap W_{2}$, then clearly $\CC{\mathcal{S}} x =x=\Pro_{W_{1} \cap W_{2}}x$.
If  $x \in W_{1} \cap W^{c}_{2}$, then $x =\R_{W_{1}}x$, $\Pro_{W_{2}}x=\Pro_{H_{2}}x$, $\R_{W_{2}}x=\R_{H_{2}}x $, and $ \mathcal{S}(x) =\{ x, \R_{W_{1}}x, \R_{W_{2}}x \}=\{ x, \R_{W_{2}}x \} $. Hence, $\CC{\mathcal{S}} x =\CCO{ (\{ x,  \R_{W_{2}}x \}) } = \frac{ x + \R_{W_{2}}x }{2}=\Pro_{W_{2}}x=\Pro_{H_{2}}x$.  Similarly, if $x \in W^{c}_{1} \cap W_{2}$, then $\CC{\mathcal{S}} x =\Pro_{H_{1} }x$.

\emph{Case~2}:   Suppose that  $x \in \mathcal{H} \smallsetminus (W_{1} \cup W_{2})$. Then $ \mathcal{S}(x) =\{ x, \R_{W_{1}}x, \R_{W_{2}}x \}= \{ x, \R_{H_{1}}x, \R_{H_{2}}x \}$.
	Because  $H_{1} \subseteq W_{1}$ and $H_{2} \subseteq W_{2}$, we know that $x \notin H_{1} \cup H_{2}$. Note that, in view of \cite[Lemma~2.3]{Oy2020ProjectionHH}, if $u_{1}$ and $u_{2}$ are linearly independent, then $H_{1} \cap H_{2} \neq \varnothing$. Hence, by \cref{them:Tn:CCSP:R:S1}\cref{item:them:Tn:CCSP:R:S1:eq}, we obtain that $ \CC{\mathcal{S}} x =\CCO{ ( \{ x, \R_{H_{1}}x, \R_{H_{2}}x \} )}=
	\Pro_{H_{1} \cap H_{2}}x$. 	
\end{proof}

The following results are necessary to the proof of \cref{prop:CCS2:half-spaces}.  We leave the long and mechanical proof  in the  Appendix.
\begin{lemma} \label{lem:innpu1u2LID}
	Assume that $u_{1}$ and $u_{2}$ are linearly independent. Let $x \in \mathcal{H}$.
	\begin{enumerate}
		\item \label{lem:innpu1u2LID:notinW1} Assume that $x \notin W_{1}$. Then $\R_{W_{1}} x \in W_{2}  \Leftrightarrow (\innp{x,u_{2}} -\eta_{2})\norm{u_{1}}^{2} -2( \innp{x,u_{1}}-\eta_{1} )\innp{u_{1},u_{2}} \leq 0$.
		
			\item \label{lem:innpu1u2leq0:xW1CW2} Assume that $x \in W^{c}_{1} \cap W_{2}$ and that $\R_{W_{1}} x \in W_{2}$. Then $\Pro_{W_{1}}x  \in W_{1} \cap W_{2}$.
	
		\item 	\label{lem:innpu1u2geq0} 	Assume that	$\innp{u_{1},u_{2}} \geq  0$ and that $x \in W_{1} \cap W^{c}_{2}$.   Then $\Pro_{W_{2}}x  \in W_{1} \cap W_{2}$ and $\R_{W_{2}}x  \in W_{1} \cap \inte W_{2}$.
		
		\item  \label{lem:innpu1u2leq0:xW1CW2C} Assume that $\innp{u_{1},u_{2}} \leq 0$ and that  $x \in W^{c}_{1} \cap W^{c}_{2}$. Then $\R_{W_{1}}x \notin W_{2}$.
		\item \label{lemma:PHiPHji} Assume that  $\innp{u_{1},u_{2}}>0$ and that $ \Pro_{H_{1}}x \notin W_{2}$. Then $\Pro_{W_{2}}\Pro_{H_{1}}x=\Pro_{H_{2}}\Pro_{H_{1}}x \in W_{1} \cap W_{2}$.
		\item \label{lemma:PH1H2notinW2} Assume that $\innp{u_{1},u_{2}} <0$, that  $x \in  H_{2}$  and that $x \notin W_{1}$.  	 Then $ \R_{H_{1}}x  \notin W_{2}$.
	\end{enumerate}

\end{lemma}

The following result elaborates that the CRM induced by  $ \mathcal{S} :=\{ \Id, \R_{W_{1}}, \R_{W_{2}}\R_{W_{1}} \}$ needs at most two steps to find a point in $W_{1}\cap W_{2}$. 
\begin{proposition}\label{prop:CCS2:half-spaces}
	Assume that $u_{1} $ and $u_{2}$ are linearly independent.   Set $\mathcal{S} :=\{ \Id, \R_{W_{1}}, \R_{W_{2}}\R_{W_{1}} \}$.  Let $x \in \mathcal{H}$. 
	Then exactly one of the following cases occurs:
	\begin{enumerate}
		\item \label{theorem:CCS2:half-spaces:W1capW2} $x \in W_{1}\cap W_{2} $. Then $\CC{\mathcal{S}}x= x= \Pro_{W_{1} \cap W_{2}}x$.
		\item \label{theorem:CCS2:half-spaces:W1capW2C}  $x \in W_{1} \cap W^{c}_{2} $. Then $\CC{\mathcal{S}}x= \Pro_{W_{2} } x=\Pro_{H_{2} } x$ and there exists $k \in \{1,2\}$ such that $\CC{\mathcal{S}}^{k}x \in W_{1} \cap W_{2}$.  

		\item \label{theorem:CCS2:half-spaces:W1CcapW2:leq} $x \in W^{c}_{1} \cap W_{2} $ and $(\innp{x,u_{2}} -\eta_{2})\norm{u_{1}}^{2} -2( \innp{x,u_{1}}-\eta_{1} )\innp{u_{1},u_{2}} \leq 0$. Then $\CC{\mathcal{S}}x= \Pro_{W_{1} } x=\Pro_{H_{1} } x = \Pro_{W_{1} \cap W_{2}}x$.  
		\item \label{theorem:CCS2:half-spaces:W1CcapW2:>}  $x \in W^{c}_{1} \cap W_{2} $ and $(\innp{x,u_{2}} -\eta_{2})\norm{u_{1}}^{2} -2( \innp{x,u_{1}}-\eta_{1} )\innp{u_{1},u_{2}} > 0$ $($this case  won't occur when $\innp{u_{1},u_{2}} \geq0 $$)$. Then $\CC{\mathcal{S}}x= \Pro_{H_{1} \cap H_{2}}x$. 
		\item \label{theorem:CCS2:half-spaces:W1CcapW2C:leq}  $x \in W^{c}_{1} \cap W^{c}_{2}  $ and $(\innp{x,u_{2}} -\eta_{2})\norm{u_{1}}^{2} -2( \innp{x,u_{1}}-\eta_{1} )\innp{u_{1},u_{2}} \leq 0$ $($this case  won't occur when $\innp{u_{1},u_{2}} \leq  0 $$)$. Then $\CC{\mathcal{S}}x= \Pro_{W_{1} } x=\Pro_{H_{1} } x$ and there exists $k \in \{1,2\}$ such that $\CC{\mathcal{S}}^{k}x \in W_{1} \cap W_{2}$. 
		\item \label{theorem:CCS2:half-spaces:W1CcapW2C:>} $x \in W^{c}_{1} \cap W^{c}_{2}  $ and $(\innp{x,u_{2}} -\eta_{2})\norm{u_{1}}^{2} -2( \innp{x,u_{1}}-\eta_{1} )\innp{u_{1},u_{2}} > 0$. Then $\CC{\mathcal{S}}x= \Pro_{H_{1} \cap H_{2}}x$.    
	\end{enumerate}
\end{proposition}

\begin{proof}

	\cref{theorem:CCS2:half-spaces:W1capW2}: Assume that $x \in W_{1}\cap W_{2} $. Then $ \mathcal{S}(x) =\{x\}$. Hence, clearly $\CC{\mathcal{S}}x= x= \Pro_{W_{1} \cap W_{2}}x$.
	
	\cref{theorem:CCS2:half-spaces:W1capW2C}: Assume that $x \in W_{1} \cap W^{c}_{2} $. Then  $x =\R_{W_{1}}x$ and $ \mathcal{S}(x) =\{ x, \R_{W_{1}}x, \R_{W_{2}}\R_{W_{1}}x \} = \{ x, \R_{W_{2}}x\}$. Hence, by \cref{thm:SymForm}\cref{thm:SymForm:2} and \cref{fact:FactWFactH}, $	\CC{\mathcal{S}}x =\CCO{ (\{ x, \R_{W_{2}} x  \}) } =\frac{ x+ \R_{W_{2}} x }{2} =\Pro_{W_{2}}x =  \Pro_{H_{2}}x$.

	Assume that $\innp{u_{1}, u_{2}} \geq 0$. Then by \cref{lem:innpu1u2LID}\cref{lem:innpu1u2geq0},  $\Pro_{W_{2}}x \in W_{1} \cap W_{2}$. Hence, by \cref{fact:AsubseteqB:Projection},   $\CC{\mathcal{S}}x= \Pro_{W_{2}}x= \Pro_{W_{1} \cap W_{2}}x$.
	
	Suppose $\innp{u_{1}, u_{2}} < 0$. We shall prove that $\CC{\mathcal{S}}^{2}x \in W_{1} \cap W_{2}$. 	
	If $ \CC{\mathcal{S}}x=  \Pro_{H_{2}}x \in W_{1}$, then we are done. 	
	Assume $\Pro_{H_{2}}x \notin W_{1}$. Then apply \cref{lem:innpu1u2LID}\cref{lemma:PH1H2notinW2} with $x= \Pro_{H_{2}}x$ to obtain that $\R_{H_{1}}\Pro_{H_{2}}x \notin W_{2}$. 
	Now $ \Pro_{H_{2}}x \notin  W_{1}$ and  $\R_{H_{1}}\Pro_{H_{2}}x \notin W_{2}$ imply that  $ \Pro_{H_{2}}x \notin  H_{1}  $ and $\R_{H_{1}}\Pro_{H_{2}}x \notin H_{2} $, and that $\CCO{( \{\Pro_{H_{2}}x, \R_{W_{1}}\Pro_{H_{2}}x,   \R_{W_{2}}\R_{W_{1}}\Pro_{H_{2}}x\})} =\CCO{( \{\Pro_{H_{2}}x, \R_{H_{1}}\Pro_{H_{2}}x,   \R_{H_{2}}\R_{H_{1}}\Pro_{H_{2}}x\})}$. Using these results and applying  \cref{cor:Tn:CCSP:R}\cref{item:cor:Tn:CCSP:R:S2} with $x=\Pro_{H_{2}}x$,  we obtain that  $\CCO{( \{\Pro_{H_{2}}x, \R_{H_{1}}\Pro_{H_{2}}x,   \R_{H_{2}}\R_{H_{1}}\Pro_{H_{2}}x\})} =\Pro_{H_{1} \cap H_{2}}\Pro_{H_{2}}x $.
	Moreover, using \cite[Lemma~2.3]{BOyW2020BAM}, we know that $\Pro_{H_{1} \cap H_{2}}\Pro_{H_{2}}x = \Pro_{H_{1} \cap H_{2}}x$. Altogether,   
$\CC{\mathcal{S}}^{2}x= \CC{\mathcal{S}} (\CC{\mathcal{S}}x) = \CC{\mathcal{S}} (\Pro_{H_{2}}x) = \Pro_{H_{1} \cap H_{2}}x \in W_{1} \cap W_{2}$.
	
	\cref{theorem:CCS2:half-spaces:W1CcapW2:leq}: Assume that $x \in W^{c}_{1} \cap W_{2} $ and $(\innp{x,u_{2}} -\eta_{2})\norm{u_{1}}^{2} -2( \innp{x,u_{1}}-\eta_{1} )\innp{u_{1},u_{2}} \leq 0$.
Note that \cref{lem:innpu1u2LID}\cref{lem:innpu1u2LID:notinW1} yields    $ \R_{W_{1}}x \in W_{2}$.  Then  $\R_{W_{2}}\R_{W_{1}}x=\R_{W_{1}}x$, and $ \mathcal{S}(x) =\{ x, \R_{W_{1}}x, \R_{W_{2}}\R_{W_{1}}x \} = \{ x, \R_{W_{1}}x\}$. Hence, by \cref{thm:SymForm}\cref{thm:SymForm:2} and \cref{fact:FactWFactH}, $\CC{\mathcal{S}}x =\CCO{ (\{ x, \R_{W_{1}} x  \}) } =\frac{ x+ \R_{W_{1}} x }{2} =\Pro_{W_{1}}x =  \Pro_{H_{1}}x$.
On the other hand, $x \in W^{c}_{1} \cap W_{2} $ and $ \R_{W_{1}}x \in W_{2}$. Moreover, via \cref{lem:innpu1u2LID}\cref{lem:innpu1u2leq0:xW1CW2}, $\Pro_{W_{1}}x  \in W_{1} \cap W_{2}$. Combine these results with \cref{fact:AsubseteqB:Projection} to obtain that $\CC{\mathcal{S}}x= \Pro_{W_{1}}x= \Pro_{W_{1} \cap W_{2}}x$.
	
	\cref{theorem:CCS2:half-spaces:W1CcapW2:>}: Assume that $x \in W^{c}_{1} \cap W_{2} $ and $(\innp{x,u_{2}} -\eta_{2})\norm{u_{1}}^{2} -2( \innp{x,u_{1}}-\eta_{1} )\innp{u_{1},u_{2}} > 0$.
	By \cref{lem:innpu1u2LID}\cref{lem:innpu1u2LID:notinW1}, we have that  $ \R_{W_{1}}x \notin W_{2}$. Apply \cref{lem:innpu1u2LID}\cref{lem:innpu1u2geq0} with swapping $W_{1}  $ and $W_{2} $ to know that this case won't happen when $\innp{u_{1},u_{2}} \geq 0$. 
		Moreover, because $x \in W^{c}_{1} \cap W_{2} $ and $ \R_{W_{1}}x \notin W_{2}$,  by \cref{fact:FactWFactH}, $ \mathcal{S}(x) =\{ x, \R_{W_{1}}x, \R_{W_{2}}\R_{W_{1}}x \} = \{ x, \R_{H_{1}}x, \R_{H_{2}}\R_{H_{1}}x\}$. Hence, by \cref{cor:Tn:CCSP:R}\cref{item:cor:Tn:CCSP:R:S2}, $\CC{\mathcal{S}}x= \Pro_{H_{1} \cap H_{2}}x$. 
	
	\cref{theorem:CCS2:half-spaces:W1CcapW2C:leq}: Assume that $x \in W^{c}_{1} \cap W^{c}_{2}  $ and $(\innp{x,u_{2}} -\eta_{2})\norm{u_{1}}^{2} -2( \innp{x,u_{1}}-\eta_{1} )\innp{u_{1},u_{2}} \leq 0$.  In view of \cref{lem:innpu1u2LID}\cref{lem:innpu1u2LID:notinW1},  $ \R_{W_{1}}x \in W_{2}$. By \cref{lem:innpu1u2LID}\cref{lem:innpu1u2leq0:xW1CW2C}, this case won't happen when  $\innp{u_{1},u_{2}} \leq 0$. Hence, we have $\innp{u_{1},u_{2}} > 0$.
	
	Because $x \in W^{c}_{1} \cap W^{c}_{2} $ and $ \R_{W_{1}}x \in W_{2}$,  we have, by \cref{fact:FactWFactH}, that $\R_{W_{1}}x=\R_{H_{1}}x$, $\R_{W_{2}}\R_{W_{1}}x=\R_{W_{1}}x $, and, 
	by \cref{thm:SymForm}\cref{thm:SymForm:2}, $	\CC{\mathcal{S}}x =\CCO{ (\{ x, \R_{W_{1}} x  \})} =\frac{ x+ \R_{W_{1}} x }{2} =\Pro_{W_{1}}x =  \Pro_{H_{1}}x$.

	If $\Pro_{H_{1}}x \in W_{2} $, then we are done. Assume that $\Pro_{H_{1}}x \notin W_{2} $.
	Then, 
	$
	\CC{\mathcal{S}}^{2}x=\CC{\mathcal{S}} ( \CC{\mathcal{S}}x ) =\CC{\mathcal{S}} ( \Pro_{H_{1}}x)  =\CCO{ (\{ \Pro_{H_{1}}x,  \R_{W_{1}}\Pro_{H_{1}}x, \R_{W_{2}}\R_{W_{1}}\Pro_{H_{1}}x\})} = \CCO{( \{ \Pro_{H_{1}}x,   \R_{H_{2}}\Pro_{H_{1}}x\})}  =\tfrac{\Pro_{H_{1}}x+   \R_{H_{2}}\Pro_{H_{1}}x}{2}=\Pro_{H_{2}}\Pro_{H_{1}}x.
	$
	Hence, by \cref{lem:innpu1u2LID}\cref{lemma:PHiPHji}, $\CC{\mathcal{S}}^{2}x =\Pro_{H_{2}}\Pro_{H_{1}}x \in W_{1} \cap W_{2}$. Altogether, 
	\cref{theorem:CCS2:half-spaces:W1CcapW2C:leq} holds.
	
	\cref{theorem:CCS2:half-spaces:W1CcapW2C:>}: Assume that $x \in W^{c}_{1} \cap W^{c}_{2}  $ and $(\innp{x,u_{2}} -\eta_{2})\norm{u_{1}}^{2} -2( \innp{x,u_{1}}-\eta_{1} )\innp{u_{1},u_{2}} > 0$.
By \cref{lem:innpu1u2LID}\cref{lem:innpu1u2LID:notinW1},   $ \R_{W_{1}}x \notin W_{2}$. Hence, via \cref{cor:Tn:CCSP:R}\cref{item:cor:Tn:CCSP:R:S2}, $\CC{\mathcal{S}}x =\CCO{ (\{ \Pro_{H_{1}}x,  \R_{H_{1}}x, \R_{H_{2}}\R_{H_{1}}x\})}= \Pro_{H_{1} \cap H_{2}}x$.

	In conclusion, the proof is complete.
\end{proof}
In the following result, we reduce the CRM associated with  one hyperplane and one  halfspace to the CRM associated with   two hyperplanes which was considered in the first subsection of 
\cref{sec:CRM:hyperplanes}. 
\begin{lemma} \label{lem:u1u2LD:H1W2}
	 Assume that $u_{1} $ and $u_{2}$ are linearly independent. 
	Set $\mathcal{S}_{1} :=\{ \Id, \R_{H_{1}}, \R_{W_{2}} \}, \mathcal{S}_{2} :=\{ \Id, \R_{H_{1}}, \R_{W_{2}}\R_{H_{1}}  \}$ and $\mathcal{S}_{3} :=\{ \Id, \R_{W_{2}}, \R_{H_{1}} \R_{W_{2}} \}$.  Let $x \in \mathcal{H}$. 
	Then the following statements hold.
	\begin{enumerate}
		\item \label{lem:u1u2LD:H1W2:CCS1} If $x \in W_{2}$, then $\CC{\mathcal{S}_{1}} x =\Pro_{H_{1}}x$; otherwise, $\CC{\mathcal{S}_{1}} x =\CCO{ (\{ x, \R_{H_{1}}x, \R_{H_{2}} x\}) }$.
		\item \label{lem:u1u2LD:H1W2:CCS2}  If $\R_{H_{1}}x \in W_{2}$, then $\CC{\mathcal{S}_{2}} x =\Pro_{H_{1}}x$; otherwise, $\CC{\mathcal{S}_{2}} x =\CCO{ (\{ x, \R_{H_{1}}x, \R_{H_{2}} \R_{H_{1}}x\}) }$. 

		\item \label{lem:u1u2LD:H1W2:CCS3} If $x \in W_{2}$, then $\CC{\mathcal{S}_{3}} x =\Pro_{H_{1}}x$; otherwise, $\CC{\mathcal{S}_{3}} x =\CCO{ (\{ x, \R_{H_{2}}x, \R_{H_{1}} \R_{H_{2}}x\})  }$.  
	\end{enumerate}
\end{lemma}

\begin{proof}
Note that $x \in W_{2} \Leftrightarrow x = \R_{W_{2}} x$ and that $\R_{H_{1}}x \in W_{2} \Leftrightarrow \R_{H_{1}}x = \R_{W_{2}} \R_{H_{1}}x$.
	Therefore, the required results follow easily from \cref{thm:SymForm}\cref{thm:SymForm:2} and \cref{defn:cir:map}.		
\end{proof}

To end this section, below we summarize results on the finite convergence of CRMs that we obtained in this section. 
\begin{theorem}\label{theorem:ConclusionWH}
For every $ i \in \{1,2\}$,   $u_{i}$  is in $\mathcal{H}$,   $\eta_{i}$ is in $\mathbb{R}$, and set 
\begin{align*}
W_{i} := \{x \in \mathcal{H} ~:~ \innp{x,u_{i}} \leq \eta_{i} \} \quad \text{and} \quad H_{i} := \{x \in \mathcal{H} ~:~ \innp{x,u_{i}} = \eta_{i} \}.
\end{align*}
	\begin{enumerate}
		\item 	Assume that $H_{1} \cap H_{2} \neq \varnothing$. Then the CRM induced by $\mathcal{S} :=\{ \Id, \R_{H_{1}}, \R_{H_{2}}\R_{H_{1}} \}$ needs at most $3$ steps to find the best approximation point $\Pro_{H_{1} \cap H_{2}}x$ (see \cref{them:Tn:CCSP:R:S2}).
		\item Assume that  $u_{1}$ and $u_{2}$ are linearly dependent and that $W_{1} \cap W_{2} \neq \varnothing$.   Then the CRM induced by  $\mathcal{S} :=\{ \Id, \R_{W_{1}}, \R_{W_{2}} \}$ with initial point $x \in W_{1} \cup W_{2}$ finds the  best approximation point $\Pro_{W_{1} \cap W_{2}}x$ in one step   (see \cref{lem:CCSu10u20} and \cref{prop:u1u2LD:S1:neq0}).
		\item  Assume that  $u_{1}$ and $u_{2}$ are linearly dependent and that $W_{1} \cap W_{2} \neq \varnothing$.  Then the CRM induced by  $\mathcal{S} :=\{ \Id, \R_{W_{1}}, \R_{W_{2}}\R_{W_{1}} \}$ with initial point $x \in W_{1} $ finds the  best approximation point $\Pro_{W_{1} \cap W_{2}}x$ in one step   (see \cref{lem:CCSu10u20} and  \cref{prop:u1u2LD:S2:<0}). 
		\item 	Assume that $u_{1} $ and $ u_{2}$ are linearly independent.   Then the CRM induced by  $\mathcal{S} :=\{ \Id, \R_{W_{1}}, \R_{W_{2}}\R_{W_{1}} \}$ with initial point $x \in \mathcal{H} $ finds a feasibility point in $ W_{1} \cap W_{2}$ in at most two steps   (see \cref{prop:CCS2:half-spaces}).
	\item 	Assume that $H_{1} \cap W_{2} \neq \varnothing$.  If $u_{1} = u_{2}=0$, then $(\forall x \in \mathcal{H})$ $\CC{\mathcal{S}} x =\Pro_{H_{1} \cap W_{2}}x$ with  $\mathcal{S}$ being $\{\Id, \R_{H_{1}}, \R_{W_{2}} \}$, $\{\Id, \R_{H_{1}}, \R_{W_{2}}\R_{H_{1}} \}$, or $ \{\Id, \R_{W_{2}}, \R_{H_{1}} \R_{W_{2}}\}$;
  If $u_{1}$ and $u_{2}$ are nonzero and linearly dependent, then $(\forall x \in \mathcal{H})$ $ \Pro_{H_{1}}x=\Pro_{H_{1} \cap W_{2}}x$; otherwise,  the CRM induced by $\mathcal{S} :=\{ \Id, \R_{H_{1}}, \R_{H_{2}}\R_{H_{1}} \}$ finds the feasibility point $\Pro_{H_{1} \cap H_{2}}x \in H_{1} \cap W_{2}$   in  $3$ steps 
  (see  \cref{them:Tn:CCSP:R:S2},
		\cref{prop:H1W2:u1u2LD:P} and \cref{lem:CCSu10u20}).
	\end{enumerate}
\end{theorem}

\section*{Acknowledgements}
The author is grateful to the anonymous referee and the editor for their helpful comments which improved the presentation of this work.

\addcontentsline{toc}{section}{References}

\bibliographystyle{abbrv}

\appendix
\section*{Appendix} \label{sec:Appendix}
\addcontentsline{toc}{section}{Appendix}
Below, we provide proofs of \Cref{lem:u1u2LD,lem:innpu1u2LID}.

\subsection*{Proof of \cref{lem:u1u2LD}}
 
\begin{proof}
	\cref{item:lem:u1u2LD:u12neq0:innpneq0}: Because $u_{1}  $ and $u_{2}  $ are nonzero and linearly dependent, due to \cite[Fact~2.1]{Oy2020ProjectionHH},    either $	  u_{2} = \frac{\norm{u_{2}}}{\norm{u_{1}}} u_{1} $ or $u_{2} = -\frac{\norm{u_{2}}}{\norm{u_{1}}} u_{1}$, which implies required results in \cref{item:lem:u1u2LD:u12neq0:innpneq0}.

	\cref{item:lem:u1u2LD:>0}: These results follow easily from \cref{item:lem:u1u2LD:u12neq0:innpneq0} and \cref{eq:WH}.
	
	\cref{item:lem:u1u2LD:<0}: Due to \cref{eq:WH} and \cref{item:lem:u1u2LD:u12neq0:innpneq0}, clearly  $\frac{\eta_{1}}{\norm{u_{1}}} =  -\frac{\eta_{2}}{\norm{u_{2}}} \Leftrightarrow  H_{1} =H_{2}$; and    $\frac{\eta_{1}}{\norm{u_{1}}} \neq -\frac{\eta_{2}}{\norm{u_{2}}}  \Leftrightarrow  H_{1} \cap H_{2} =\varnothing$. 
	
	In view of \cref{item:lem:u1u2LD:u12neq0:innpneq0}, $\innp{u_{1}, u_{2}} <0$ implies that $u_{2} = -\frac{\norm{u_{2}}}{\norm{u_{1}}} u_{1}$.
	This and  \cref{eq:WH} yield that 
	\begin{align} \label{item:lem:u1u2LD:u12neq0:innp<0:W2}\tag{1}
	W_{2} =\{ x \in \mathcal{H} ~:~ \innp{x, u_{2}} \leq \eta_{2} \} =   \left\{ x \in \mathcal{H} ~:~ \innp{x, u_{1}} \geq -\frac{\norm{u_{1}}}{\norm{u_{2}}} \eta_{2} \right\}.
	\end{align}
Employing \cref{eq:WH} again, we yield $W_{1} = \{ x \in \mathcal{H} ~:~ \innp{x, u_{1}} \leq \eta_{1} \}$. Hence, 
	\begin{align} \label{eq:item:lem:u1u2LD:u12neq0:innp<0}  \tag{2}
	W_{1}\cap W_{2} \neq \varnothing \Leftrightarrow -\frac{\norm{u_{1}}}{\norm{u_{2}}} \eta_{2}  \leq \eta_{1} \Leftrightarrow \eta_{1}\norm{u_{2}} +\eta_{2} \norm{u_{1}} \geq 0.
	\end{align} 
	Assume that $W_{1}\cap W_{2} \neq \varnothing$. Notice that $	x \notin W_{2} \stackrel{\cref{item:lem:u1u2LD:u12neq0:innp<0:W2}}{\Leftrightarrow} \innp{x, u_{1}} < -\frac{\norm{u_{1}}}{\norm{u_{2}}} \eta_{2}  \stackrel{\cref{eq:item:lem:u1u2LD:u12neq0:innp<0} }{\Rightarrow} \innp{x, u_{1}} < \eta_{1} \Leftrightarrow x \in \inte W_{1}$ and that $x \notin W_{1}  \Leftrightarrow \innp{x, u_{1}} > \eta_{1}  \stackrel{\cref{eq:item:lem:u1u2LD:u12neq0:innp<0} }{\Rightarrow} \innp{x, u_{1}} > -\frac{\norm{u_{1}}}{\norm{u_{2}}} \eta_{2} \stackrel{\cref{item:lem:u1u2LD:u12neq0:innp<0:W2}}{\Leftrightarrow} x \in \inte W_{2}$.
	Hence, $\mathcal{H}=\inte W_{1}\cup W_{2} = W_{1} \cup \inte W_{2}$.

	Because $u_{2} = -\frac{\norm{u_{2}}}{\norm{u_{1}}} u_{1}$, $(\forall x \in H_{1})$  $	\innp{x, u_{2}} -\eta_{2}= -\frac{\norm{u_{2}}}{\norm{u_{1}}}  \innp{x, u_{1}} -\eta_{2}  \stackrel{\cref{eq:WH}}{=}  -\frac{\norm{u_{2}}}{\norm{u_{1}}}  \eta_{1} -\eta_{2}  \stackrel{\cref{eq:item:lem:u1u2LD:u12neq0:innp<0}}{\leq} 0$,
	which implies that  $H_{1} \subseteq W_{2}$. The proof of $H_{2} \subseteq W_{1}$ is similar.

	\cref{item:lem:u1u2LD:u12neq0:innp>0:a}: 
In consideration of \cref{item:lem:u1u2LD:u12neq0:innpneq0}, $u_{2} =  \frac{\norm{u_{2}}}{\norm{u_{1}}} u_{1}$.  Combine this with   \cref{fact:Projec:Hyperplane}, the assumption $ \frac{\eta_{1}}{\norm{u_{1}}} \geq \frac{\eta_{2}}{\norm{u_{2}}} $, and some easy algebra, to  observe  that  
		\begin{align}\label{eq:item:lem:u1u2LD:u12neq0:innp>0:a} \tag{3}
		(\forall x \in \mathcal{H}) \quad 	\norm{\Pro_{H_{1}}x - \Pro_{H_{2}}x}  = \Norm{ x + \frac{\eta_{1} - \innp{x,u_{1}}}{\norm{u_{1}}^{2} }  u_{1}  -\Big( x + \frac{\eta_{2} - \innp{x,u_{2}}}{\norm{u_{2}}^{2} }  u_{2}  \Big) } 
		 = \frac{\eta_{1}}{\norm{u_{1}}} -\frac{\eta_{2}}{\norm{u_{2}}}. 
		\end{align}

	Let $x \in \mathcal{H} \smallsetminus W_{1}$. Then via \cref{eq:WH}   and    \cref{fact:FactWFactH}, $\innp{x, u_{1}} >\eta_{1}$ and  $ \Pro_{W_{1}}x = \Pro_{H_{1}}x = x + \frac{\eta_{1} - \innp{x,u_{1}}}{\norm{u_{1}}^{2} }  u_{1} $. Combining these results with  \cref{eq:item:lem:u1u2LD:u12neq0:innp>0:a}, we derive that
	\begin{align} \label{item:lemma:u1u2Ld:larger0:PH1PH2:notinW2:x-PW1} \tag{4}
	\norm{x - \Pro_{W_{1}}x} \geq  \norm{\Pro_{H_{1}}x-\Pro_{H_{2}}x } \Leftrightarrow  \frac{ \innp{x,u_{1}} -\eta_{1} }{\norm{u_{1}} }   \geq  \frac{\eta_{1}}{\norm{u_{1}}} -\frac{\eta_{2}}{\norm{u_{2}}}.
	\end{align}
	On the other hand, 	
	\begin{subequations}  
		\begin{align*}
		\R_{W_{1}} x\in W_{2} 
		& \Leftrightarrow	\innp{\R_{W_{1}}x, u_{2}} -\eta_{2} \leq 0  \quad (\text{by \cref{eq:WH}})\\
		& \Leftrightarrow \Innp{2\left(x + \frac{\eta_{1} - \innp{x,u_{1}}}{\norm{u_{1}}^{2} }  u_{1} \right)-x, u_{2}} -\eta_{2} \leq 0  \quad (\text{by  \cref{fact:FactWFactH}})\\
		&\Leftrightarrow \Innp{x+2 \frac{\eta_{1} - \innp{x,u_{1}}}{\norm{u_{1}}^{2} }  u_{1} ,  \frac{\norm{u_{2}}}{\norm{u_{1}}}u_{1}} -\eta_{2} \leq 0 \quad (\text{by $u_{2} =  \frac{\norm{u_{2}}}{\norm{u_{1}}} u_{1}$})\\
		& \Leftrightarrow  \norm{u_{2}}  \left( \innp{x, u_{1}} +2 (\eta_{1} - \innp{x,u_{1}})\right)  -\norm{u_{1}} \eta_{2} \leq 0\\
		& \Leftrightarrow    \frac{ \innp{x,u_{1}} -\eta_{1} }{\norm{u_{1}} }   \geq  \frac{\eta_{1}}{\norm{u_{1}}} -\frac{\eta_{2}}{\norm{u_{2}}},
		\end{align*}
	\end{subequations}
	which, combining with \cref{item:lemma:u1u2Ld:larger0:PH1PH2:notinW2:x-PW1}, leads that $(x \in \mathcal{H} \smallsetminus W_{1})$ $ \norm{x - \Pro_{W_{1}}x} \geq \norm{\Pro_{H_{1}}x-\Pro_{H_{2}}x } $ if and only if $\R_{W_{1} }x\in W_{2}$.

	\cref{item:lem:u1u2LD}: The proof is similar to that  of \cref{item:lem:u1u2LD:u12neq0:innp>0:a} and is omitted.

	\cref{item:lem:u1u2LD:u12neq0:innp<0:W12}:  Because $x \in W_{1} \cap W^{c}_{2}$, invoking \cref{fact:FactWFactH} and \cref{item:lem:u1u2LD:<0},
	we know that $\Pro_{W_{2}}x=\Pro_{H_{2}}x \in H_{2}  \subseteq W_{1}$. Hence, $\Pro_{W_{2}}x =\Pro_{H_{2}}x  \in W_{1} \cap H_{2} \subseteq W_{1} \cap W_{2}$.
	Apply \cref{fact:AsubseteqB:Projection} with $A=W_{1} \cap H_{2}$ and $B=H_{2}$ to obtain that $ \Pro_{H_{2}}x = \Pro_{W_{1} \cap H_{2}  }x $. Moreover, employ \cref{fact:AsubseteqB:Projection} with $A=W_{1} \cap W_{2}$ and $B=W_{2}$ to entail that $\Pro_{W_{2}}x = \Pro_{W_{1} \cap W_{2}  } x$.  So, 	\cref{item:lem:u1u2LD:u12neq0:innp<0:W12} holds.

	\cref{item:lem:u1u2LD:u12neq0:innp<0:d}: Inasmuch as  \cref{item:lem:u1u2LD:u12neq0:innpneq0}, $\innp{u_{1},u_{2}} <0 $ implies that $u_{2}=- \frac{\norm{u_{2}}}{\norm{u_{1}}}u_{1}$.   Then via  \cref{fact:Projec:Hyperplane} and  \cref{eq:WH},  
	\begin{subequations}
		\begin{align} 
		&\norm{x -\Pro_{H_{1}}x} = \Norm{ x -\Big( x + \frac{\eta_{1} - \innp{x,u_{1}}}{\norm{u_{1}}^{2} }  u_{1} \Big) } =\Big| \frac{\eta_{1} - \innp{x,u_{1}}}{\norm{u_{1}} }  \Big|; \label{eq:lemma:u1u2Ld:Less0:xinW2:xPh1Larger:xPH1x} \tag{5a}\\
		&x \in W_{2} \Leftrightarrow \innp{x, u_{2}} \leq \eta_{2} \Leftrightarrow \Innp{x, - \frac{\norm{u_{2}}u_{1}}{\norm{u_{1}}}}  \leq \eta_{2} \Leftrightarrow \Innp{x, - \frac{u_{1}}{\norm{u_{1}}}} \leq \frac{\eta_{2}}{\norm{u_{2}}}.\label{eq:lemma:u1u2Ld:Less0:xinW2:xPh1Larger:xinW2} \tag{5b}
		\end{align}
	\end{subequations}

	\cref{lemma:u1u2Ld:Less0:xinW2:xPh1Larger:xnotinW1}: Assume to the contrary that  $x \in W_{1}$.  Then $\innp{x, u_{1}} \leq \eta_{1}$. 
	Hence, $	\norm{x -\Pro_{H_{1}}x}   \stackrel{\cref{eq:lemma:u1u2Ld:Less0:xinW2:xPh1Larger:xPH1x}}{=} \Big| \frac{\eta_{1} - \innp{x,u_{1}}}{\norm{u_{1}}}  \Big|  
	= \frac{\eta_{1} - \innp{x,u_{1}}}{\norm{u_{1}} }  
	\stackrel{\cref{eq:lemma:u1u2Ld:Less0:xinW2:xPh1Larger:xinW2}}{\leq} \frac{\eta_{1}}{\norm{u_{1}}}  + \frac{\eta_{2}}{\norm{u_{2}}} \stackrel{\text{\cref{item:lem:u1u2LD}}}{=} \norm{\Pro_{H_{1}}x - \Pro_{H_{2}}x}$,
	which   contradicts  the assumption, $\norm{x - \Pro_{H_{1}}x} > \norm{\Pro_{H_{1}}x-\Pro_{H_{2}}x }$.  
	
	\cref{lemma:u1u2Ld:Less0:xinW2:xPh1Larger:RW1notinW2}: Due to \cref{lemma:u1u2Ld:Less0:xinW2:xPh1Larger:xnotinW1} and \cref{eq:WH},  $\innp{x,u_{1}} > \eta_{1}$. So,   \cref{eq:lemma:u1u2Ld:Less0:xinW2:xPh1Larger:xPH1x} and  \cref{item:lem:u1u2LD} lead to
		\begin{align} \label{lemma:u1u2Ld:Less0:xinW2:d:>} \tag{6}
		\norm{x - \Pro_{H_{1}}x} > \norm{\Pro_{H_{1}}x-\Pro_{H_{2}}x }  \Leftrightarrow \frac{ \innp{x,u_{1}} -\eta_{1} }{\norm{u_{1}} } > \frac{\eta_{1}}{\norm{u_{1}}} +\frac{\eta_{2}}{\norm{u_{2}}} 
	  \Leftrightarrow -\frac{\norm{u_{2}}}{\norm{u_{1}}}  (2\eta_{1} - \innp{x,u_{1}} )> \eta_{2}. 
		\end{align}
	
	On the other hand, adopting $u_{2}=- \tfrac{\norm{u_{2}}}{\norm{u_{1}}}u_{1}$ and $\R_{W_{1}}x = x+2\frac{\eta_{1} - \innp{x,u_{1}}}{\norm{u_{1}}^{2} } u_{1}$, we observe that $	\innp{\R_{W_{1}}x ,u_{2}} 
	= \Innp{\R_{W_{1}}x, - \frac{\norm{u_{2}}u_{1}}{\norm{u_{1}}} } 
	= - \frac{\norm{u_{2}}}{\norm{u_{1}}}   \Innp{ x+2\frac{\eta_{1} - \innp{x,u_{1}}}{\norm{u_{1}}^{2} } u_{1}, u_{1}  }  
	= - \frac{\norm{u_{2}}}{\norm{u_{1}}}   (2 \eta_{1} - \innp{x,u_{1}}  ) 
	\stackrel{\cref{lemma:u1u2Ld:Less0:xinW2:d:>}}{>} \eta_{2}$,
	which implies that $\R_{W_{1}}x  \notin W_{2} $.

	\cref{lemma:u1u2Ld:Less0:xinW2:xPh1Larger:NormLarger}:  
Notice that
	\begin{align*}
	\norm{\R_{W_{2}}\R_{W_{1}}x - \R_{W_{1}}x } &=\norm{\R_{H_{2}}\R_{H_{1}}x - \R_{H_{1}}x } \quad (\text{by \cref{lemma:u1u2Ld:Less0:xinW2:xPh1Larger:xnotinW1}, \cref{lemma:u1u2Ld:Less0:xinW2:xPh1Larger:RW1notinW2}, and \cref{fact:FactWFactH}}) \\
	&= 2 \norm{ 2\Pro_{H_{2}} \Pro_{H_{1}}x -\Pro_{H_{2}} x - 2\Pro_{H_{1}}x+x}\\
	&= 2\left| \frac{\eta_{2} -\innp{x,u_{2}}}{\norm{u_{2}}}  +2 \frac{\eta_{1} -\innp{x,u_{1}}}{\norm{u_{1}}}  \right|(\text{\cref{fact:Projec:Hyperplane}, \cite[Lemma~2.11(i)]{Oy2020ProjectionHH}, and $u_{2}=- \tfrac{\norm{u_{2}}}{\norm{u_{1}}}u_{1}$}) \\
	&=  2\left|  \frac{\eta_{2}}{\norm{u_{2}}} + 2\frac{\eta_{1}}{\norm{u_{1}}}- \Innp{x, \frac{u_{1}}{\norm{u_{1}}} } \right| \quad (\text{by $u_{2}=- \tfrac{\norm{u_{2}}}{\norm{u_{1}}}u_{1}$}) \\
	&= 2 \left| \frac{\eta_{2}}{\norm{u_{2}}} + \frac{\eta_{1}}{\norm{u_{1}}}+ \frac{\eta_{1} -\innp{x,u_{1}}}{\norm{u_{1}}}  \right| \\
	&= 2\frac{\innp{x,u_{1}} - \eta_{1} }{\norm{u_{1}}} -2\Big( \frac{\eta_{1}}{\norm{u_{1}}} +\frac{\eta_{2}}{\norm{u_{2}}}  \Big) \quad (\text{by \cref{lemma:u1u2Ld:Less0:xinW2:d:>}})\\
	&=2\norm{x - \Pro_{H_{1}}x} - 2\norm{ \Pro_{H_{1}}x- \Pro_{H_{2}}x} \quad  (\text{by \cref{lemma:u1u2Ld:Less0:xinW2:xPh1Larger:xnotinW1}, \cref{eq:lemma:u1u2Ld:Less0:xinW2:xPh1Larger:xPH1x} and \cref{item:lem:u1u2LD}})\\
	& \leq 2\norm{x - \Pro_{H_{1}}x} 
	= \norm{ x -\R_{W_{1}}x }.
	\end{align*}
	Therefore, \cref{lemma:u1u2Ld:Less0:xinW2:xPh1Larger:NormLarger} holds.

	\cref{item:lem:u1u2LD:LD}:    Let $x \in \mathbb{N}$ and let $(\forall i \in \{1,2\})$ $y_{i} \in H_{i}$.    Bearing  \cref{fac:SetChangeProje} in mind, we have that
	\begin{subequations}
		\begin{align}
		& \R_{H_{i}}x - x   =2(\Pro_{y_{i} +\pa H_{i}} (x)-x)  =-2 \Pro_{(\pa H_{i})^{\perp} }(x-y_{i} ) \in (\pa H_{i} )^{\perp}=(\ker u_{i})^{\perp} =\spn \{u_{i}\}; \label{eq:item:lem:u1u2LD:LD:R} \tag{7a}\\
		& \R_{H_{2}}\R_{H_{1}}x - \R_{H_{1}}x  = 2(\Pro_{y_{2} +\pa H_{2}} (\R_{H_{1}}x)-\R_{H_{1}}x) =-2 \Pro_{(\pa H_{2})^{\perp} }(\R_{H_{1}}x-y_{2}) \in \spn \{ u_{2}\}.\label{eq:item:lem:u1u2LD:LD:RR}\tag{7b}
		\end{align}	
	\end{subequations}
	Recall that $u_{1}$ and $u_{2}$ are linear dependent. Hence, $x, \R_{H_{1}}x, \R_{H_{2}}x$ are affinely  dependent, and so do $x, \R_{H_{1}}x, \R_{H_{2}}\R_{H_{1}}x$. 
	
	In addition,    if $x \in W_{1} \cup W_{2}$, then $\R_{W_{1}}x =x$ or $\R_{W_{2}}x =x$, which clearly forces that  $x, \R_{W_{1}}x, \R_{W_{2}}x$ are affinely dependent. If $x \notin W_{1} \cup W_{2}$, then, by \cref{fact:FactWFactH}, $(\forall i \in \{1,2\})$ $ \R_{W_{i}}x =\R_{H_{i}}x$. Hence, by \cref{eq:item:lem:u1u2LD:LD:R},  $x, \R_{W_{1}}x, \R_{W_{2}}x$ are affinely  dependent.
	
	Similarly, in view of \cref{fact:FactWFactH} and \cref{eq:item:lem:u1u2LD:LD:RR},   $x, \R_{W_{1}}x, \R_{W_{2}}\R_{W_{1}}x$ are affinely dependent. 
\end{proof}

\subsection*{Proof of \cref{lem:innpu1u2LID}}

\begin{proof}
	\cref{lem:innpu1u2LID:notinW1}: Because  $x \notin W_{1}$,  by \cref{fact:FactWFactH}, we derive that  $\R_{W_{1}}x= 2\Pro_{W_{1}}x-x=  x + 2\frac{\eta_{1} - \innp{x, u_{1}}}{\norm{u_{1}}^{2}} u_{1} $. Then
	\begin{align*}
	\R_{W_{1}} x \in W_{2} 
\Leftrightarrow \Innp{ x + 2\frac{\eta_{1} - \innp{x, u_{1}}}{\norm{u_{1}}^{2}} u_{1} , u_{2}} \leq \eta_{2} 
\Leftrightarrow (\innp{x,u_{2}} -\eta_{2})\norm{u_{1}}^{2} -2( \innp{x,u_{1}}-\eta_{1} )\innp{u_{1},u_{2}} \leq 0.
	\end{align*}

	\cref{lem:innpu1u2leq0:xW1CW2}: 
As a consequence of  \cref{eq:WH} and $x \in W_{2}$,  $\innp{ x,u_{2} } -\eta_{2} \leq 0$. Moreover, 
	using  $\R_{W_{1}}x \in W_{2}$  and	\cref{eq:WH}, we see that   $	\innp{\R_{W_{1}}x,u_{2} } -\eta_{2} \leq 0   \Leftrightarrow 
	\innp{2 \Pro_{W_{1}}x -x,u_{2} } -\eta_{2} \leq 0 
	\Rightarrow 	\innp{ \Pro_{W_{1}}x ,u_{2} } -\eta_{2} \leq \tfrac{1}{2}( \innp{ x,u_{2} } -\eta_{2}) \leq 0$,
	which implies that $\Pro_{W_{1}}x \in W_{2} $. Clearly, $\Pro_{W_{1}}x  \in W_{1}$.  Hence,  $\Pro_{W_{1}}x  \in W_{1} \cap W_{2}$.

	\cref{lem:innpu1u2geq0}:
	Because $x \in W_{1} \cap W^{c}_{2}$, by \cref{eq:WH}, we know that $	\innp{x,u_{1}} \leq \eta_{1}$ and   $\innp{x,u_{2}} > \eta_{2}$. Combine these inequalities with
	$\innp{u_{1},u_{2}} \geq  0$ to see that 
		\begin{align}\label{eq:lem:innpu1u2geq0:2:inW2} \tag{8}
		 \innp{x, u_{1}} -\eta_{1} + \frac{\eta_{2} - \innp{x, u_{2}}}{\norm{u_{2}}^{2}} \innp{u_{1}, u_{2}} \leq 0 \text{ and }
		 \innp{x, u_{1}} -\eta_{1} + 2\frac{\eta_{2} - \innp{x, u_{2}}}{\norm{u_{2}}^{2}} \innp{u_{1}, u_{2}} \leq 0. 
		\end{align}
	According to \cref{fact:FactWFactH}, $\innp{\Pro_{W_{2}}x, u_{1}} -\eta_{1} 
	= \Innp{ x + \frac{\eta_{2} - \innp{x, u_{2}}}{\norm{u_{2}}^{2}} u_{2}, u_{1}} -\eta_{1} =
	\innp{x, u_{1}} -\eta_{1} + \frac{\eta_{2} - \innp{x, u_{2}}}{\norm{u_{2}}^{2}} \innp{u_{1}, u_{2}}
	\stackrel{\cref{eq:lem:innpu1u2geq0:2:inW2}}{\leq} 0$ and $\innp{\R_{W_{2}}x, u_{1}} -\eta_{1}  =  \Innp{x +2 \frac{\eta_{2} - \innp{x, u_{2}}}{\norm{u_{2}}^{2}} u_{2} , u_{1}} -\eta_{1}  
	=\innp{x, u_{1}} -\eta_{1} + 2\frac{\eta_{2} - \innp{x, u_{2}}}{\norm{u_{2}}^{2}} \innp{u_{1}, u_{2}} 
	\stackrel{\cref{eq:lem:innpu1u2geq0:2:inW2}}{\leq} 0$,
	which imply that $\Pro_{W_{2}}x  =\Pro_{H_{2}}x \in W_{1}$ and    $\R_{W_{2}}x   \in W_{1}$, respectively. In addition, because $x \notin W_{2}$, by  \cref{lem:WH:mathcalH:RWinW},  $\R_{W_{2}}x   \in \inte W_{2}$.

	\cref{lem:innpu1u2leq0:xW1CW2C}: Note that, by \cref{eq:WH},  $x \in W^{c}_{1} \cap W^{c}_{2}$ means that $\innp{ x, u_{1}} -\eta_{1} >0 $ and $\innp{ x, u_{2}} -\eta_{2} >0$. Combine these two inequalities with  $\innp{u_{1}, u_{2} } \leq 0$ to see that
	\begin{align} \label{eq:lem:innpu1u2leq0:xW1CW2C:>0} \tag{9}
	\innp{ x, u_{2}} -\eta_{2} +  2\frac{\eta_{1} - \innp{x, u_{1}}}{\norm{u_{1}}^{2}} \innp{u_{1}, u_{2}} >0.
	\end{align}
	Because $x \notin W_{1}$,  by \cref{fact:FactWFactH},  $\R_{W_{1}}x= 2\Pro_{H_{1}}x-x=  x + 2\frac{\eta_{1} - \innp{x, u_{1}}}{\norm{u_{1}}^{2}} u_{1} $. Now
	\begin{align*}
	\innp{\R_{W_{1}}x, u_{2}} -\eta_{2}  =  \Innp{x + 2\frac{\eta_{1} - \innp{x, u_{1}}}{\norm{u_{1}}^{2}} u_{1} , u_{2}} -\eta_{2}  
	= \innp{ x, u_{2}} -\eta_{2} +  2\frac{\eta_{1} - \innp{x, u_{1}}}{\norm{u_{1}}^{2}} \innp{u_{1}, u_{2}} 
	\stackrel{\cref{eq:lem:innpu1u2leq0:xW1CW2C:>0}}{>}0,
	\end{align*}
	which entails that $ \R_{W_{1}}x \notin W_{2}$.
	
	\cref{lemma:PHiPHji}: On account of the assumption $ \Pro_{H_{1}}x \notin W_{2}$ and  \cref{fact:Projec:Hyperplane}, 	
	\begin{align} \label{eq:lemma:PHiPHji}  \tag{10}
	\innp{ \Pro_{H_{1}}x , u_{2}} >\eta_{2} \Leftrightarrow (\eta_{2} -\innp{x,u_{2}}) \norm{u_{1}}^{2}  -(\eta_{1} -\innp{x,u_{1}}) \innp{u_{1},u_{2}} <0.
	\end{align}
	According to the formula of $ \Pro_{H_{2}}\Pro_{H_{1}}x$ in \cite[Lemma~2.11(i)]{Oy2020ProjectionHH},
	\begin{align*}
	& \innp{ \Pro_{H_{2}}\Pro_{H_{1}}x, u_{1} } -\eta_{1} \\
	=& \innp{x,u_{1}} -\eta_{1} + \frac{\eta_{1} -\innp{x,u_{1}}}{\norm{u_{1}}^{2}} \innp{u_{1}, u_{1}} +	\frac{1}{ \norm{u_{1}} \norm{u_{2}} } \left( (\eta_{2} -\innp{x,u_{2}}) \norm{u_{1}}^{2}  -(\eta_{1} -\innp{x,u_{1}}) \innp{u_{1},u_{2}}\right) \innp{u_{1},u_{2}} \\
	=& 	\frac{1}{ \norm{u_{1}} \norm{u_{2}} } \left( (\eta_{2} -\innp{x,u_{2}}) \norm{u_{1}}^{2}  -(\eta_{1} -\innp{x,u_{1}}) \innp{u_{1},u_{2}}\right) \innp{u_{1},u_{2}} \\
	<& 0, \quad (\text{by \cref{eq:lemma:PHiPHji}  and } \innp{u_{1},u_{2}}>0 )
	\end{align*}
	which, by \cref{eq:WH} and \cref{fact:FactWFactH}, guarantees that $\Pro_{W_{2}}\Pro_{H_{1}}x=\Pro_{H_{2}}\Pro_{H_{1}}x \in W_{1}$. Clearly, $\Pro_{H_{2}}\Pro_{H_{1}}x \in H_{2} \subseteq W_{2}$.	 
	Hence, $\Pro_{H_{2}}\Pro_{H_{1}}x \in W_{1} \cap W_{2}$. 
	
	\cref{lemma:PH1H2notinW2}: Note that $x \notin W_{1}$ and $\innp{u_{1},u_{2}} <0$ ensure $\frac{\eta_{1}-\innp{x,u_{1}}}{\norm{u_{1}}^{2}} \innp{u_{1},u_{2}} >0$.	
	On the other hand, using \cref{eq:WH},  \cref{fact:Projec:Hyperplane}, and $x \in H_{2}$, we obtain that
	\begin{align*}
	\R_{H_{1}}x\notin W_{2}   \Leftrightarrow \innp{	\R_{H_{1}}x, u_{2} } >\eta_{2} 
\Leftrightarrow \Innp{x+2\frac{\eta_{1}-\innp{x,u_{1}}}{\norm{u_{1}}^{2}} u_{1} , u_{2} } >\eta_{2} 
 \Leftrightarrow 2\frac{\eta_{1}-\innp{x,u_{1}}}{\norm{u_{1}}^{2}} \innp{u_{1},u_{2}} >0. 
	\end{align*}
	Hence, 	 \cref{lemma:PH1H2notinW2} is true. 
\end{proof}

\end{document}